\documentclass[reqno,10pt]{article}

\usepackage{a4wide}
\usepackage{color}
\usepackage{mathrsfs}
\usepackage{mathtools}
\usepackage{amsmath}
\usepackage{amsthm}
\usepackage{amssymb}
\usepackage{bbm}
\usepackage{tikz-cd}
\usepackage{esint}
\numberwithin{equation}{section}
\usepackage[colorlinks,citecolor=green,linkcolor=red]{hyperref}

\usepackage[latin1]{inputenc}

\newcommand{\nchi}{{\raise.3ex\hbox{\(\chi\)}}}
\newcommand{\N}{\mathbb{N}}
\newcommand{\R}{\mathbb{R}}

\newcommand{\sfd}{{\sf d}}
\renewcommand{\d}{{\mathrm d}}

\newcommand{\X}{{\rm X}}
\newcommand{\Y}{{\rm Y}}

\newcommand{\mm}{\mathfrak{m}}
\newcommand{\1}{\mathbbm 1}

\newcommand{\lims}{\varlimsup}

\newcommand{\Idem}{{\rm Idem}}
\newcommand{\Adm}{{\rm Adm}}

\newcommand{\fr}{\penalty-20\null\hfill\(\blacksquare\)}

\newtheorem{theorem}{Theorem}[section]
\newtheorem{corollary}[theorem]{Corollary}
\newtheorem{lemma}[theorem]{Lemma}
\newtheorem{proposition}[theorem]{Proposition}
\newtheorem{definition}[theorem]{Definition}
\newtheorem{example}[theorem]{Example}
\newtheorem{remark}[theorem]{Remark}

\title{An axiomatic theory of normed modules via Riesz spaces}
\author{Danka Lu\v{c}i\'{c}\footnote{\href{mailto:danka.d.lucic@jyu.fi}{danka.d.lucic@jyu.fi},
Department of Mathematics and Statistics, P.O.\ Box 35 (MaD), FI-40014 University of Jyv\"askyl\"a, Finland.}
\;and Enrico Pasqualetto\footnote{\href{mailto:enrico.e.pasqualetto@jyu.fi}{enrico.e.pasqualetto@jyu.fi},
Department of Mathematics and Statistics, P.O.\ Box 35 (MaD), FI-40014 University of Jyv\"askyl\"a, Finland.}}
\begin{document}
\date{\today}
\maketitle
\begin{abstract}
We introduce and study an axiomatic theory of $V$-normed $U$-modules, where \(V\) is a Riesz space and \(U\) is an \(f\)-algebra; the spaces $U$ and $V$
also have some additional structure and are required to satisfy a compatibility condition. Roughly speaking, a $V$-normed $U$-module is a module over $U$ that is
endowed with a pointwise norm operator taking values in $V$. The aim of our approach is to develop a unified framework, which is tailored to the differential
calculus on metric measure spaces, where as $U$ and $V$ one can take many different spaces of functions.
\end{abstract}
\tableofcontents
\section{Introduction}
\subsection{General overview and motivations}\label{s:general_overview}
On metric measure spaces \((\X,\sfd,\mm)\), the study of first-order Sobolev spaces \(W^{1,p}(\X)\) for \(p\in(1,\infty)\)
has been a fruitful field of research in the last decades, see e.g.\ \cite{Hajlasz96,Cheeger00,Shanmugalingam00,AmbrosioGigliSavare11}.
In order to develop a differential calculus modelled over \(W^{1,p}(\X)\), several notions of `measurable (co)vector fields'
were studied, for example by \cite{Cheeger00} in the setting of doubling spaces supporting a Poincar\'{e} inequality.
One of the objectives of \cite{Gigli14} was to provide a meaningful notion of `space of measurable \(1\)-forms' for arbitrary
metric measure spaces. This is encoded in the concept of \emph{cotangent module}, which we are going to remind. It is proved in
\cite[Section 2.2.1]{Gigli14} that it is possible to construct a vector space \(L^p(T^*\X)\) and a linear operator \(\d\colon W^{1,p}(\X)\to L^p(T^*\X)\)
having the following features:
\begin{itemize}
\item[\(\rm i)\)] The elements of \(L^p(T^*\X)\) can be multiplied by \(L^\infty(\mm)\)-functions; to be precise, \(L^p(T^*\X)\) is a module
over the commutative ring \(L^\infty(\mm)\).
\item[\(\rm ii)\)] There exists a map \(|\cdot|\colon L^p(T^*\X)\to L^p(\mm)^+\) that vanishes only at \(0\), that satisfies
\[
|\omega+\eta|\leq|\omega|+|\eta|,\quad\text{ for every }\omega,\eta\in L^p(T^*\X),
\]
and that is compatible with the \(L^\infty(\mm)\)-module structure, in the sense that \(|f\cdot\omega|=|f||\omega|\) for every \(f\in L^\infty(\mm)\)
and \(\omega\in L^p(T^*\X)\). The map \(|\cdot|\) is said to be a \emph{pointwise norm} operator. Moreover, the norm on \(L^p(T^*\X)\) induced by the pointwise
norm via integration, i.e.
\[
\|\omega\|_{L^p(T^*\X)}\coloneqq\||\omega|\|_{L^p(\mm)},\quad\text{ for every }\omega\in L^p(T^*\X),
\]
is required to be complete.
\item[\(\rm iii)\)] The operator \(\d\colon W^{1,p}(\X)\to L^p(T^*\X)\), which is called the \emph{differential}, satisfies
\[
\|f\|_{W^{1,p}(\X)}=\big(\|f\|_{L^p(\mm)}^p+\||\d f|\|_{L^p(\mm)}^p\big)^{1/p},\quad\text{ for every }f\in W^{1,p}(\X)
\]
and has the property that the \(L^\infty(\mm)\)-module generated by its image is dense in \(L^p(T^*\X)\).
\end{itemize}
Following \cite[Definition 1.2.10]{Gigli14}, any couple \((\mathscr M,|\cdot|)\) verifying i) and ii) is called an \emph{\(L^p(\mm)\)-Banach \(L^\infty(\mm)\)-module}
(in fact, in \cite{Gigli14} the term `\(L^p(\mm)\)-normed \(L^\infty(\mm)\)-module' is used, but in this paper we need to distinguish between complete and non-complete modules).
Nevertheless, a number of variants of this notion have been (unavoidably) considered in \cite{Gigli14} and in the subsequent literature:
\begin{itemize}
\item It might be convenient (and sometimes necessary) to drop the \(L^p\)-integrability assumption. Technically speaking, this is made precise by the notion
of \emph{\(L^0(\mm)\)-Banach \(L^0(\mm)\)-module}; see \cite[Section 1.3]{Gigli14}. For example, to work with \(L^0\)-Banach \(L^0\)-modules is essential when
constructing tensor products of \(L^2(\mm)\)-Hilbert \(L^\infty(\mm)\)-modules, cf.\ with \cite[Section 1.5]{Gigli14}.
\item The case \(p=\infty\) has been studied as well. Indeed, \(L^\infty(\mm)\)-Banach \(L^\infty(\mm)\)-modules are fundamental in order to apply the
lifting theory by von Neumann in the Banach module setting \cite{DMLP21}, which in turn allows to provide `fiberwise descriptions', i.e.\ to show that any Banach
module is the space of sections of some generalised Banach bundle \cite{GLP22}. At present, using fibers is the only way to provide an explicit characterisation
of duals and/of pullbacks of Banach modules, which are useful objects for the applications in metric measure geometry.
\item Under suitable curvature bounds (e.g.\ in the setting of \({\sf RCD}(K,\infty)\) spaces), one is often interested in extending the differential calculus to codimension-one
measures (e.g.\ to perimeter measures). The functional-analytic framework that allows to achieve this goal is based on the concept of \emph{\(L^0({\rm Cap})\)-Banach
\(L^0({\rm Cap})\)-module}, which was introduced in \cite{debin2019quasicontinuous}. Here, \(\rm Cap\) denotes the Sobolev capacity, which is an outer measure on \(\X\)
that is not Borel regular.
\end{itemize}
The aim of this work is to provide a unified theory of Banach modules, which covers (at least) all the notions of Banach modules discussed above.
Indeed, albeit similar on some aspects, the several variants of Banach module often required different ad hoc definitions and proof strategies.
Our goal is to introduce an `axiomatic framework', where instead of function spaces we consider more general classes of Riesz spaces and \(f\)-algebras,
as well as to obtain rather general existence results, which can be applied in all the specific cases we described above, whenever this is needed.
\subsection{Main definitions}\label{s:main_def}
Let us now discuss the various objects we are going to introduce, also motivating the reasons behind our definitions. First, a key feature of all the `functional'
Banach modules from Section \ref{s:general_overview} is the possibility to multiply by \emph{characteristic functions}. This is fundamental, for example, when
constructing the cotangent module. Observe that in \(L^\infty(\mm)\) the characteristic functions of Borel sets are given exactly by the \emph{idempotent elements},
i.e.\ by those \(f\in L^\infty(\mm)\) satisfying \(f^2=f\). Moreover, two different functional spaces appear in the definition of Banach module: the ring of functions
that can be multiplied by the elements of the Banach module (e.g.\ \(L^\infty(\mm)\)), and the vector space of functions where the pointwise norm takes values
(e.g.\ \(L^p(\mm)\)). These two functional spaces must be related, e.g.\ the compatibility requirement between the pointwise norm and the module structure uses
the fact that \(fg\in L^p(\mm)\) whenever \(f\in L^\infty(\mm)\) and \(g\in L^p(\mm)\).
\medskip

Taking all these features into account, we propose in Definition \ref{def:metric_f-structure} the concept of
\[
\textbf{metric \(f\)-structure}\quad(\mathcal U,U,V).
\]
Let us describe informally what a metric \(f\)-structure \((\mathcal U,U,V)\) is:
\begin{itemize}
\item \(\mathcal U\) is an ambient \emph{localisable} \(f\)-algebra (see Definition \ref{def:localisable_f-alg}),
which means that it is an \(f\)-algebra (i.e.\ a Riesz space together with a compatible multiplication operation,
see Definition \ref{def:f-algebra}) having plenty of \emph{idempotent elements} (see \eqref{eq:idem}).
This corresponds, for example, to the fact that simple functions are order-dense in \(L^0(\mm)\).
\item \((U,\sfd_U)\) is a \emph{metric \(f\)-algebra} (see Definition \ref{def:metric_f-algebra})
that is an \(f\)-subalgebra of \(\mathcal U\). This means that \(U\) is an \(f\)-algebra endowed
with a complete distance \(\sfd_U\) that verifies suitable compatibility conditions. For example, the space
\(L^\infty(\mm)\) is an \(f\)-subalgebra of \(L^0(\mm)\), and together with (the distance induced by)
its norm, \(L^\infty(\mm)\) is a metric \(f\)-algebra.
\item \((V,\sfd_V)\) is a \emph{metric Riesz space} (see Definition \ref{def:metric_Riesz}) that is also
a Riesz subspace of \(\mathcal U\) satisfying \(UV=V\). For example, \(L^p(\mm)\) is a metric Riesz space
and \(L^\infty(\mm)\cdot L^p(\mm)=L^p(\mm)\).
\end{itemize}
Our axiomatisation of a metric \(f\)-structure is tailored to the kind of Banach modules
we are interested to. However, already in the framework of differential calculus on metric measure spaces,
some important objects studied in the literature (e.g.\ Lipschitz derivations \cite{Weaver99}
or local vector measures \cite{BrenaGigli22}) are not covered by our theory, roughly speaking because
the \(f\)-algebra of bounded continuous functions is not localisable (as characteristic functions
are typically not continuous).
\medskip

As we discussed above, an example of metric \(f\)-structure is \((L^0(\mm),L^\infty(\mm),L^p(\mm))\).
Taking into consideration the notion of \(L^p(\mm)\)-Banach \(L^\infty(\mm)\)-module from Section \ref{s:general_overview},
one can think of the elements of \(U\) as those that can be multiplied by the elements of the Banach module,
the role of \(V\) is `the space where the pointwise norm takes values', while \(\mathcal U\) is an ambient
space where both \(U\) and \(V\) can be embedded (which is convenient to formulate the requirement
that \(UV=V\)). Having this discussion in mind, we propose in Definition \ref{def:norm_U-mod} the concept of
\[
\textbf{\(V\)-Banach \(U\)-module}\quad\mathscr M.
\]
The definition of \(V\)-Banach \(U\)-module roughly states the following:
\begin{itemize}
\item \(\mathscr M\) is a module over the commutative ring \(U\) endowed with a pointwise norm
\(|\cdot|\colon\mathscr M\to V^+\), which verifies the pointwise triangle inequality and
is compatible with the module operations.
\item \(\mathscr M\) has the \emph{glueing property}, which means every \emph{admissible} sequence
of \emph{disjoint} elements \((v_n)_{n\in\N}\) of \(\mathscr M\) can be `glued together', thus obtaining
a new element \(\sum_{n\in\N}v_n\in\mathscr M\). The order structure of \((\mathcal U,U,V)\) comes into play here, i.e.\ when declaring
which sequences are admissible, see Definition \ref{def:norm_U-mod} ii). We also point out that, in general, \(\sum_{n\in\N}v_n\) is
just a formal series, which does not necessarily coincide with any kind of limit of finite sums.
\item The distance \(\sfd_{\mathscr M}(v,w)\coloneqq\sfd_V(|v-w|,0)\) on \(\mathscr M\) is complete.
\end{itemize}
In the class of \(L^p(\mm)\)-Banach \(L^\infty(\mm)\)-modules with \(p\in[1,\infty)\) we described in Section \ref{s:general_overview},
we did not mention the glueing property, the reason being that in that specific framework it follows automatically from the other axioms.
On the other hand, this is not always the case. For example, the glueing property has to be required when dealing with \(L^\infty(\mm)\)-Banach
\(L^\infty(\mm)\)-modules (see \cite[Example 1.2.5]{Gigli14} or \cite[Remark 2.22]{DMLP21}). Moreover -- differently from what happens with
\(L^p(\mm)\)-Banach \(L^\infty(\mm)\)-modules, where \(|\sum_{n=1}^k v_n-\sum_{n\in\N}v_n|\to 0\) in \(L^p(\mm)\) -- on \(L^\infty(\mm)\)-Banach
\(L^\infty(\mm)\)-modules it is clear that the expression \(\sum_{n\in\N}v_n\) might be only formal: in the space \(L^\infty(\R)\) itself (which is
an \(L^\infty(\R)\)-Banach \(L^\infty(\R)\)-module), the elements \(f_n\coloneqq\1_{[n,n+1)}\) for \(n\in\mathbb Z\) can be `glued together',
obtaining the function identically equal to \(1\); however, \(1\) is not the limit in the \(L^\infty(\R)\)-norm of the partial sums
\(\sum_{n=-k}^k f_k=\1_{[-k,k+1)}\) as \(k\to\infty\). In this example, it is still true that the partial sums converge in some sense to the
glued object (e.g.\ in the weak\(^*\) topology), but this needs not be the case for arbitrary \(L^\infty(\mm)\)-Banach \(L^\infty(\mm)\)-modules,
which do not always have a predual.
\medskip

We also mention that taking duals is very useful in differential calculus on metric measure spaces. For instance, the so-called \emph{tangent module}
\(L^q(T\X)\), which can be regarded as the space of `\(q\)-integrable vector fields' on a metric measure space \((\X,\sfd,\mm)\), is defined as
the Banach module dual of the cotangent module \(L^p(T^*\X)\), see \cite[Definition 2.3.1]{Gigli14}. An important observation is that, according
to \cite[Proposition 1.2.14 i)]{Gigli14}, the dual of an \(L^p(\mm)\)-Banach \(L^\infty(\mm)\)-module \(\mathscr M\) is an \(L^q(\mm)\)-Banach
\(L^\infty(\mm)\)-module \(\mathscr M^*\), where \(p\) and \(q\) are conjugate exponents. This means that in our axiomatisation when constructing the dual of
a Banach module we have to change also the underlying metric \(f\)-structure. To address this issue, we propose in Definition \ref{def:dual_system} the concept of
\[
\textbf{dual system of metric \(f\)-structures}\quad(\mathcal U,U,V,W,Z).
\]
We omit the details here. However, the definition of dual system is given so that the \emph{module dual} of a \(V\)-Banach \(U\)-module \(\mathscr M\) is a
\(W\)-Banach \(U\)-module, see Definition \ref{def:dual_mod}. More generally, the space \(\textsc{Hom}(\mathscr M,\mathscr N)\) of all \emph{homomorphisms}
(Definition \ref{def:hom}) from a \(V\)-Banach \(U\)-module \(\mathscr M\) to a \(Z\)-Banach \(U\)-module \(\mathscr N\) inherits a natural structure
of \(W\)-Banach \(U\)-module (Theorem \ref{thm:Hom_normed_module}). An example of dual system of metric \(f\)-structures is \((L^0(\mm),L^\infty(\mm),L^p(\mm),L^q(\mm),L^1(\mm))\).
When proving finer results about homomorphisms and dual modules, one often has to require a further regularity
property on the underlying \(f\)-algebras and Riesz spaces, namely that they are \emph{countably representable}, or \emph{CR} for short; see Definition \ref{def:CR}.
This assumption amounts to saying that every set that is bounded from above (resp.\ from below) has an supremum (resp.\ an infimum), i.e.\ the space is \emph{Dedekind complete},
and that such supremum (resp.\ infimum) can be expressed as a countable supremum (resp.\ a countable infimum) of elements of the given set. This property is enjoyed e.g.\ by
\(L^p(\mm)\) whenever \(p\in\{0\}\cup[1,\infty]\) and \(\mm\) is a \(\sigma\)-finite measure (Proposition \ref{prop:L0_CR}), but it fails in \(L^0({\rm Cap})\) (Example \ref{ex:L0Cap_no_CR}).
The CR property is also needed, for instance, to construct \emph{local inverses} (Proposition \ref{prop:local_inverses}) or to define the \emph{support} of a metric \(f\)-structure (Definition \ref{def:spt}).
\subsection{Main results}
Another objective of this work is to provide a rather complete toolbox of results and techniques concerning Banach modules over a metric \(f\)-structure,
which we plan to apply in the future, as a `black box' in many particular cases of interest. Our two main achievements are the following:
\begin{itemize}
\item Theorem \ref{thm:module_generated}: given a metric \(f\)-structure \((\mathcal U,U,V)\), a vector space \(\mathscr V\), and a symmetric sublinear map \(\psi\colon\mathscr V\to V^+\),
there exists a unique couple \((\mathscr M_{\langle\psi\rangle},T_{\langle\psi\rangle})\), where \(\mathscr M_{\langle\psi\rangle}\) is a \(V\)-Banach \(U\)-module, while
\(T_{\langle\psi\rangle}\colon\mathscr V\to\mathscr M_{\langle\psi\rangle}\) is a linear operator with `generating image' (in a suitable sense) such that \(|T_{\langle\psi\rangle}{\sf v}|=\psi({\sf v})\)
for every \({\sf v}\in\mathscr V\). The uniqueness is formulated in categorical terms, i.e.\ via a universal property (see also Corollary \ref{cor:module_generated_cat}). This quite general
existence result incorporates most of the existence results for Banach modules considered so far in the related literature. For example, the cotangent module \(L^p(T^*\X)\) and the differential \(\d\)
are given by \((L^p(T^*\X),\d)\cong(\mathscr M_{\langle\psi_p\rangle},T_{\langle\psi_p\rangle})\), where the map \(\psi_p\colon W^{1,p}(\X)\to L^p(\mm)^+\) is defined as \(\psi_p(f)\coloneqq|Df|\).
See Section \ref{ss:applications_fct_nm} for this example, as well as for other relevant constructions of Banach modules induced by a symmetric sublinear map.
\item Theorem \ref{thm:main_ext_hom}: it is an existence criterion for homomorphisms of Banach modules. Indeed, given that the theory of \(V\)-Banach \(U\)-modules fits well in a categorical framework
(see Definition \ref{def:category_nm}), it is natural to couple Theorem \ref{thm:module_generated} with an existence result for homomorphisms. For simplicity of presentation, let us state here only a
corollary of Theorem \ref{thm:Hom_normed_module}: given a dual system \((\mathcal U,U,V,W,Z)\), a \(V\)-Banach \(U\)-module \(\mathscr M\), a \(Z\)-Banach \(U\)-module \(\mathscr N\), a `generating'
vector subspace \(\mathscr V\) of \(\mathscr M\), and a linear operator \(T\colon\mathscr V\to\mathscr N\) satisfying \(|Tv|\leq b|v|\) for some \(b\in W^+\), there is a unique extension
\(\bar T\in\textsc{Hom}(\mathscr M,\mathscr N)\) of \(T\), which still satisfies \(|\bar T v|\leq b|v|\).
\end{itemize}
Finally, we conclude the introduction by briefly mentioning other results we obtain in the paper:
\begin{itemize}
\item Using Theorems \ref{thm:module_generated} and \ref{thm:Hom_normed_module}, we prove that each homomorphism of metric \(f\)-structures induces a \emph{pushforward functor}
(or, to be more precise, a `direct image functor') in the categories of Banach modules; see Section \ref{ss:pushfrwd_mod}.
\item We prove a version of the \emph{Hahn--Banach extension theorem} for \(V\)-Banach \(U\)-modules, see Theorem \ref{thm:Hahn-Banach}. It is used e.g.\ for studying module duals
and embedding operators into the bidual; see Sections \ref{ss:Hahn-Banach} and \ref{ss:refl_mod}.
\item We study \emph{Hilbert modules}, i.e.\ Banach modules whose pointwise norm verifies a pointwise parallelogram identity, see Definition \ref{def:Hilbert_mod}. Among the several results we obtain,
let us mention a Hilbert projection theorem and a Riesz representation theorem; see Section \ref{ss:Hilbert_mod}.
\item We prove that \(V\)-Banach \(U\)-modules admit a \emph{dimensional decomposition} (under the CR assumption on the metric \(f\)-structure); see Section \ref{ss:dim_decomp}.
\end{itemize}
In Sections \ref{s:local_metric_f-struct} and \ref{s:normed_mod} the whole treatment is at the level of `abstract' Riesz spaces and \(f\)-algebras, without ever mentioning any
kind of functional spaces. The applications of our axiomatic theory to the various classes of Banach modules over spaces of functions are discussed in Section \ref{s:applic_fct}.
\bigskip

\noindent\textbf{Acknowledgements.} The authors thank Nicola Gigli and Simone Di Marino for having suggested Corollary \ref{cor:module_generated_cat} and Theorem \ref{thm:Hahn-Banach}, respectively.
\section{Localisable \texorpdfstring{\(f\)}{f}-algebras and metric \texorpdfstring{\(f\)}{f}-structures}\label{s:local_metric_f-struct}
In Section \ref{ss:rem_Riesz} we recall many useful definitions and results concerning Riesz spaces and \(f\)-algebras, which are quite standard and well-established;
our presentation is essentially taken from \cite{Fremlin74,Fremlin3}. In Section \ref{ss:idempotent} we study the set of idempotent elements, while in Sections \ref{ss:local_f-alg}
and \ref{ss:metric_f-struct} we introduce the language of localisable \(f\)-algebras and of (dual systems of) metric \(f\)-structures, respectively.
\subsection{Reminder on Riesz spaces and \texorpdfstring{\(f\)}{f}-algebras}\label{ss:rem_Riesz}
Let \((P,\leq)\) be a partially ordered set and \(S\neq\varnothing\) a subset of \(P\). We recall the following definitions:
\begin{itemize}
\item[\(\rm i)\)] We say that \(S\) is \textbf{directed upwards} if for every \(p,p'\in S\) there exists \(q\in S\) such that
\(p\leq q\) and \(p'\leq q\). We say that \(S\) is \textbf{directed downwards} if for every \(p,p'\in S\) there exists
\(q\in S\) such that \(q\leq p\) and \(q\leq p'\).
\item[\(\rm ii)\)] A sequence \((p_n)_{n\in\N}\subset P\) is said to be \textbf{non-decreasing} provided \(p_n\leq p_{n+1}\)
for every \(n\in\N\), while it is said to be \textbf{non-increasing} provided \(p_n\geq p_{n+1}\) for every \(n\in\N\).
\item[\(\rm iii)\)] An element \(p\in P\) is said to be an \textbf{upper bound} for \(S\) provided \(s\leq p\) holds
for every \(s\in S\). We say that \(p\) is the \textbf{supremum} of \(S\), and we write \(p=\sup S\), provided
\(p\leq p'\) holds for any other upper bound \(p'\in P\) for \(S\). If \(\sup S\) exists, then it is uniquely determined.
\item[\(\rm iv)\)] An element \(q\in P\) is said to be a \textbf{lower bound} for \(S\) provided \(q\leq s\) holds
for every \(s\in S\). We say that \(q\) is the \textbf{infimum} of \(S\), and we write \(q=\inf S\), provided \(q'\leq q\)
holds for any other lower bound \(q'\in P\) for \(S\). If \(\inf S\) exists, then it is uniquely determined.
\item[\(\rm v)\)] We say that \(S\) is \textbf{order-bounded} provided it has both an upper bound and a lower bound.
\item[\(\rm vi)\)] We say that \(P\) is \textbf{Dedekind \(\sigma\)-complete} provided every countable non-empty subset of
\(P\) with an upper bound has a supremum and every countable non-empty subset of \(P\) with a lower bound has an infimum.
\item[\(\rm vii)\)] \(P\) is \textbf{Dedekind complete} if every non-empty subset of \(P\) with an upper
bound has a supremum, or equivalently every non-empty subset of \(P\) with a lower bound has an infimum.
\end{itemize}
A map \(\phi\colon P\to Q\) between partially ordered sets \(P\) and \(Q\) is said to be \textbf{order-preserving} provided
\[
\phi(p)\leq\phi(q),\quad\text{ for every }p,q\in P.
\]
An order-preserving map \(\phi\colon P\to Q\) is said to be \textbf{order-continuous} provided it holds that
\[\begin{split}
\exists\sup\big\{\phi(p)\;\big|\;p\in R\big\}=\phi(\bar p)&,\quad\text{ whenever }
R\subset P\text{ is directed upwards and }\exists\,\bar p\coloneqq\sup R\in P,\\
\exists\inf\big\{\phi(q)\;\big|\;q\in S\big\}=\phi(\bar q)&,\quad\text{ whenever }
S\subset P\text{ is directed downwards and }\exists\,\bar q\coloneqq\inf S\in P.
\end{split}\]
We say that an order-preserving map \(\phi\colon P\to Q\) is \textbf{\(\sigma\)-order-continuous} provided it holds that
\[\begin{split}
\exists\sup_{n\in\N}\phi(p_n)=\phi\Big(\sup_{n\in\N}p_n\Big)&,\quad\text{ whenever }
(p_n)_{n\in\N}\subset P\text{ is non-decreasing and }\sup_{n\in\N}p_n\text{ exists},\\
\exists\inf_{n\in\N}\phi(q_n)=\phi\Big(\inf_{n\in\N}q_n\Big)&,\quad\text{ whenever }
(q_n)_{n\in\N}\subset P\text{ is non-increasing and }\inf_{n\in\N}q_n\text{ exists},
\end{split}\]
where \(\sup_{n\in\N}p_n\) stands for \(\sup\{p_n\}_{n\in\N}\). Note that order-continuity implies \(\sigma\)-order-continuity.
\medskip

A \textbf{lattice} is a partially ordered set \((P,\leq)\) such that \(p\vee q\coloneqq\sup\{p,q\}\) and \(p\wedge q\coloneqq\inf\{p,q\}\) exist
for all \(p,q\in P\). A set \(S\subset P\) is called a \textbf{sublattice} of \(P\) if it is closed under \(\vee\) and \(\wedge\), \emph{i.e.}
\[
p\vee q,p\wedge q\in S,\qquad\text{for every }p,q\in S.
\]
A map \(\phi\colon P\to Q\) between lattices \(P\) and \(Q\) is said to be a \textbf{lattice homomorphism} provided
\[
\phi(p\vee q)=\phi(p)\vee\phi(q),\qquad\phi(p\wedge q)=\phi(p)\wedge\phi(q),\qquad\text{for every }p,q\in P.
\]
\begin{remark}{\rm
For an arbitrary family \(\{P_i\}_{i\in I}\) of partially ordered sets \(P_i=(P_i,\leq_i)\), the product \(P\coloneqq\prod_{i\in I}P_i\)
can be endowed with the following partial order: for any \((p_i)_{i\in I},(q_i)_{i\in I}\in\prod_{i\in I}P_i\), we declare that
\((p_i)_{i\in I}\leq(q_i)_{i\in I}\) if and only if \(p_i\leq_i q_i\) for every \(i\in I\). Observe that \((P,\leq)\) is a lattice
if and only if \((P_i,\leq_i)\) is a lattice for every \(i\in I\).
\fr}\end{remark}

We introduce the following notion (as we did not find any terminology for it in the literature):
\begin{definition}[Countable representability]\label{def:CR}
Let \((P,\leq)\) be a Dedekind complete partially ordered set. Then we say that \(P\) is \textbf{countably representable}
(or \textbf{CR} for short) provided it holds that:
\[\begin{split}
\forall\varnothing\neq S\subset P\text{ having an upper bound}&,\quad\exists(s_n)_{n\in\N}\subset S:
\quad\sup_{n\in\N}s_n=\sup S,\\
\forall\varnothing\neq\tilde S\subset P\text{ having a lower bound}&,\quad\exists(\tilde s_n)_{n\in\N}\subset\tilde S:
\quad\inf_{n\in\N}\tilde s_n=\inf\tilde S.
\end{split}\]
\end{definition}
\subsubsection{The theory of Riesz spaces}
A \textbf{partially ordered linear space} \((U,\leq)\) is a vector space \(U=(U,+,\cdot)\) over the field of real numbers \(\R\),
together with a partial order \(\leq\) on \(U\) such that the following properties are verified:
\[\begin{split}
u+w\leq v+w,&\quad\text{ for every }u,v,w\in U\text{ with }u\leq v,\\
\lambda u\geq 0,&\quad\text{ for every }\lambda\in\R^+\text{ and }u\in U\text{ with }u\geq 0.
\end{split}\]
A \textbf{Riesz space} is a partially ordered linear space \(U=(U,+,\cdot,\leq)\) that is a lattice. We define
\[
u^+\coloneqq u\vee 0,\qquad u^-\coloneqq(-u)\vee 0,\qquad|u|\coloneqq(-u)\vee u,
\]
for every \(u\in U\). We have that \(|u|\geq 0\) holds for every \(u\in U\), with equality if and only if \(u=0\).
\begin{proposition}[Basic properties of Riesz spaces]\label{prop:basic_prop_Riesz}
Let \(U\) be a Riesz space. Then it holds that:
\begin{subequations}\begin{align}
\label{eq:basic_prop_f-alg_1}\lambda(u\vee v)=\lambda u\vee\lambda v,&\quad\text{ for every }\lambda\in\R\text{ with }\lambda>0\text{ and }u,v\in U,\\
\label{eq:basic_prop_f-alg_1b}|\lambda u|=\lambda|u|,&\quad\text{ for every }\lambda\in\R^+\text{ and }u\in U,\\
\label{eq:basic_prop_f-alg_2}-u\vee v=(-u)\wedge(-v),&\quad\text{ for every }u,v\in U,\\
\label{eq:basic_prop_f-alg_3}u+v\vee w=(u+v)\vee(u+w),&\quad\text{ for every }u,v,w\in U,\\
\label{eq:basic_prop_f-alg_4}u+v\wedge w=(u+v)\wedge(u+w),&\quad\text{ for every }u,v,w\in U,\\
\label{eq:basic_prop_f-alg_4b}u\vee v+u\wedge v=u+v,&\quad\text{ for every }u,v\in U,\\
\label{eq:basic_prop_f-alg_5}u=u^+-u^-,&\quad\text{ for every }u\in U,\\
\label{eq:basic_prop_f-alg_6}|u|=u^+\vee u^-=u^+ +u^-,&\quad\text{ for every }u\in U,\\
\label{eq:basic_prop_f-alg_6b}u^+\wedge u^-=0,&\quad\text{ for every }u\in U,\\
\label{eq:basic_prop_f-alg_6c}(u+v)^+\leq u^+ +v^+,&\quad\text{ for every }u,v\in U,\\
\label{eq:basic_prop_f-alg_6d}|u+v|\leq|u|+|v|,&\quad\text{ for every }u,v\in U,\\
\label{eq:basic_prop_f-alg_6e}u\wedge(v+w)\leq u\wedge v+u\wedge w,&\quad\text{ for every }u,v,w\in U^+.
\end{align}\end{subequations}
\end{proposition}
\begin{proof}
Albeit well-known and rather elementary, we report the proof for the reader's usefulness.\\
{\color{blue}\eqref{eq:basic_prop_f-alg_1}} Since \(\lambda u\leq\lambda(u\vee v)\) and \(\lambda v\leq\lambda(u\vee v)\),
we obtain \(\lambda u\vee\lambda v\leq\lambda(u\vee v)\). This estimate implies that \(u\vee v=(\lambda^{-1}\lambda u)\vee(\lambda^{-1}\lambda v)
\leq\lambda^{-1}(\lambda u\vee\lambda v)\). Multiplying by \(\lambda\), we get \(\lambda(u\vee v)\leq\lambda u\vee\lambda v\).\\
{\color{blue}\eqref{eq:basic_prop_f-alg_1b}} Trivial if \(\lambda=0\). If \(\lambda>0\), then \eqref{eq:basic_prop_f-alg_1} yields
\(\lambda|u|=\lambda((-u)\vee u)=(-\lambda u)\vee\lambda u=|\lambda u|\).\\
{\color{blue}\eqref{eq:basic_prop_f-alg_2}} Since \(-u\vee v\leq -u,-v\leq -u\wedge v\), we have
\(-u\vee v\leq(-u)\wedge(-v)\) and \(-u\wedge v\geq(-u)\vee(-v)\). The latter estimate applied to
\(-u\), \(-v\) gives \((-u)\wedge(-v)\leq-u\vee v\), thus \(-u\vee v=(-u)\wedge(-v)\).\\
{\color{blue}\eqref{eq:basic_prop_f-alg_3}} Notice that \(v=(u+v)-u\leq(u+v)\vee(u+w)-u\) and \(w\leq(u+v)\vee(u+w)-u\),
so that one has \(v\vee w\leq(u+v)\vee(u+w)-u\). By applying it to \(u'\coloneqq-u\), \(v'\coloneqq u+v\), and \(w'\coloneqq u+w\),
we obtain also the converse inequality \((u+v)\vee(u+w)\leq v\vee w+u\).\\
{\color{blue}\eqref{eq:basic_prop_f-alg_4}} It follows from \eqref{eq:basic_prop_f-alg_2} and \eqref{eq:basic_prop_f-alg_3} that
\[
-u-v\wedge w=-u+(-v)\vee(-w)=(-u-v)\vee(-u-w)=-(u+v)\wedge(u+w).
\]
{\color{blue}\eqref{eq:basic_prop_f-alg_4b}} Observe that
\[\begin{split}
u\vee v+u\wedge v&\overset{\eqref{eq:basic_prop_f-alg_3}}=u+0\vee(v-u)+u\wedge v\overset{\eqref{eq:basic_prop_f-alg_4}}=
u+0\vee(v-u)+v+(u-v)\wedge 0\\&\overset{\eqref{eq:basic_prop_f-alg_2}}=u-0\wedge(u-v)+v+(u-v)\wedge 0=u+v.
\end{split}\]
{\color{blue}\eqref{eq:basic_prop_f-alg_5}} By using \eqref{eq:basic_prop_f-alg_3}, we obtain \(u^+-u^-=u\vee 0-(-u)\vee 0=u+0\vee(-u)-(-u)\vee 0=u\).\\
{\color{blue}\eqref{eq:basic_prop_f-alg_6}} First, we observe that \(u^+\vee u^-=(u\vee 0)\vee((-u)\vee 0)=(u\vee(-u))\vee 0=|u|\vee 0=|u|\). By using
\eqref{eq:basic_prop_f-alg_5}, \eqref{eq:basic_prop_f-alg_1}, and \eqref{eq:basic_prop_f-alg_3}, we get that \(u^+ +u^-=2u^+ -u=(2u)\vee 0-u=u\vee(-u)=|u|\).\\
{\color{blue}\eqref{eq:basic_prop_f-alg_6b}} Just notice that
\[
u^+\wedge u^-\overset{\eqref{eq:basic_prop_f-alg_4b}}=u^+ +u^- -u^+\vee u^-\overset{\eqref{eq:basic_prop_f-alg_6}}=|u|-|u|=0.
\]
{\color{blue}\eqref{eq:basic_prop_f-alg_6c}} Given that \(u+v\leq u\vee 0+v\vee 0\) and \(0\leq u\vee 0+v\vee 0\), we deduce that
\[
(u+v)^+=(u+v)\vee 0\leq u\vee 0+v\vee 0=u^+ +v^+.
\]
{\color{blue}\eqref{eq:basic_prop_f-alg_6d}} Observe that
\[\begin{split}
|u+v|&\overset{\eqref{eq:basic_prop_f-alg_6}}=(u+v)^+ +(u+v)^-=(u+v)^+ +(-u-v)^+
\overset{\eqref{eq:basic_prop_f-alg_6c}}\leq u^+ +v^+ +(-u)^+ +(-v)^+\\
&\overset{\phantom{\eqref{eq:basic_prop_f-alg_6}}}=u^+ +v^+ +u^- +v^-\overset{\eqref{eq:basic_prop_f-alg_6}}=|u|+|v|.
\end{split}\]
{\color{blue}\eqref{eq:basic_prop_f-alg_6e}} It follows from the estimates
\[\begin{split}
u\wedge(v+w)&=(u\wedge(v+w))\wedge u\leq((u+w)\wedge(v+w))\wedge u\overset{\eqref{eq:basic_prop_f-alg_4}}=(w+u\wedge v)\wedge u\\
&\leq(w+u\wedge v)\wedge(u+u\wedge v)\overset{\eqref{eq:basic_prop_f-alg_4}}=u\wedge v+u\wedge w.
\end{split}\]
Therefore, the proof of the statement is complete.
\end{proof}
A \textbf{Riesz subspace} of \(U\) is a linear subspace which is also a sublattice. A \textbf{homomorphism of Riesz spaces}
\(\phi\colon U\to V\) is a linear operator such that
\[
\phi(u)\wedge\phi(v)=0,\quad\text{ for every }u,v\in U\text{ such that }u\wedge v=0.
\]
By virtue of \cite[352G]{Fremlin3}, each homomorphism of Riesz spaces \(\phi\colon U\to V\) has the following property:
\begin{equation}\label{eq:hom_Riesz_and_norm}
|\phi(u)|=\phi(|u|),\quad\text{ for every }u\in U.
\end{equation}
We denote by \(U^+\) the \textbf{positive cone} of a Riesz space \(U\), namely
\[
U^+\coloneqq\big\{u\in U\;\big|\;u\geq 0\big\}.
\]
A Riesz subspace \(V\) of a Riesz space \(U\) is said to be \textbf{super-order-dense} in \(U\) if for any \(u\in U^+\)
there exists a non-decreasing sequence \((u_n)_{n\in\N}\subset V^+\) such that \(u=\sup_{n\in\N}u_n\).
Moreover, a Riesz subspace \(V\) of a Riesz space \(U\) is said to be \textbf{solid} provided \(v\in V\) holds
whenever \(v\in U\) and there exists \(u\in V\) such that \(|v|\leq|u|\). For the reader's usefulness, we prove the following result:
\begin{proposition}\label{prop:Dedekind_is_Archimedean}
Any Dedekind \(\sigma\)-complete Riesz space \(U\) is \textbf{Archimedean}, \emph{i.e.} for any \(u,v\in U\)
\[
nu\leq v,\,\text{ for every }n\in\N\quad\Longrightarrow\quad u\leq 0.
\]
\end{proposition}
\begin{proof}
Suppose \(nu\leq v\) holds for every \(n\in\N\). This implies that \(\{nu\,:\,n\in\N\}\) has an upper
bound, thus the supremum \(w\coloneqq\sup_{n\in\N}nu\in U\) exists. It is then easy to verify that
\[
u+w=u+\sup_{n\in\N}nu=\sup_{n\in\N}(u+nu)=\sup_{n\in\N}(n+1)u\leq w,
\]
whence (by subtracting \(w\)) it follows that \(u\leq 0\). Therefore, the statement is achieved.
\end{proof}
\begin{definition}[Disjoint set]
Let \(U\) be a Riesz space. Let \(S\) be a non-empty subset of \(U\). Then we say that \(S\) is \textbf{disjoint} provided it holds that
\[
|u|\wedge|v|=0,\quad\text{ for every }u,v\in S\text{ such that }u\neq v.
\]
When \(S\) is a finite disjoint set \(\{u_1,\ldots,u_n\}\subset U\), we say that the elements \(u_1,\ldots,u_n\) are disjoint.
\end{definition}

Observe that if \(\phi\colon U\to V\) is a homomorphism of Riesz spaces, then it holds that
\begin{equation}\label{eq:img_disjoint}
\big\{\phi(u)\;\big|\;u\in S\big\}\subset V\,\text{ is disjoint,}\quad\text{ for every }\varnothing\neq S\subset U\text{ disjoint.}
\end{equation}
Indeed, if \(u,v\in S\) and \(\phi(u)\neq\phi(v)\), then \(u\neq v\) and \(|\phi(u)|\wedge|\phi(v)|=\phi(|u|)\wedge\phi(|v|)=0\) by \eqref{eq:hom_Riesz_and_norm}.
\subsubsection{The theory of \(f\)-algebras}
Next, we recall the definition of \emph{\(f\)-algebra}, which is -- roughly speaking -- a Riesz space endowed with a multiplication operation that verifies suitable compatibility properties.
\begin{definition}[\(f\)-algebra]\label{def:f-algebra}
An \textbf{\(f\)-algebra} \(U=(U,+,\cdot,\leq,\times)\) is a Riesz space \((U,+,\cdot,\leq)\) together with a map
\(\times\colon U\times U\to U\) -- called a \textbf{multiplication} -- such that the following properties hold:
\begin{subequations}\begin{align}
\label{eq:def_f-alg_a}u\times(v\times w)=(u\times v)\times w,&\quad\text{ for every }u,v,w\in U,\\
\label{eq:def_f-alg_b}(u+v)\times w=(u\times w)+(v\times w),&\quad\text{ for every }u,v,w\in U,\\
\label{eq:def_f-alg_c}\lambda(u\times v)=(\lambda u)\times v,&\quad\text{ for every }u,v\in U\text{ and }\lambda\in\R,\\
\label{eq:def_f-alg_d}u\times v=v\times u,&\quad\text{ for every }u,v\in U,\\
\label{eq:def_f-alg_e}u\times v\geq 0,&\quad\text{ for every }u,v\in U^+,\\
\label{eq:def_f-alg_f}(u\times w)\wedge v=0,&\quad\text{ for every }u,v\in U\text{ with }u\wedge v=0\text{ and }w\in U^+,\\
\label{eq:def_f-alg_g}\exists\,{\bf 1}_U\in U:\;\;\;u\times{\bf 1}_U=u,&\quad\text{ for every }u\in U.
\end{align}\end{subequations}
A \textbf{homomorphism of \(f\)-algebras} \(\phi\colon U\to V\) is a homomorphism of Riesz spaces that is uniferent, \emph{i.e.}
\(\phi({\bf 1}_U)={\bf 1}_V\), and preserves the multiplication, \emph{i.e.} \(\phi(u\times v)=\phi(u)\times\phi(v)\) for all \(u,v\in U\).
An \textbf{\(f\)-subalgebra} of \(U\) is a Riesz subspace \(V\) of \(U\) closed under multiplication and with \({\bf 1}_V={\bf 1}_U\).
\end{definition}
\begin{remark}{\rm
Some comments on Definition \ref{def:f-algebra} are in order:
\begin{itemize}
\item[\(\rm i)\)] The structure \((U,+,\cdot,\leq,\times)\) introduced in Definition \ref{def:f-algebra} is usually called a
\emph{commutative \(f\)-algebra with multiplicative identity}. For the sake of brevity, we call it just an \emph{\(f\)-algebra}.
\item[\(\rm ii)\)] It follows from \eqref{eq:def_f-alg_a}, \eqref{eq:def_f-alg_b}, \eqref{eq:def_f-alg_d}, and \eqref{eq:def_f-alg_g}
that the triple \((U,+,\times)\) is a commutative ring with identity \({\bf 1}_U\). The field \(\R\) can be viewed as a subring of \(U\)
via the map \(\R\ni\lambda\mapsto\lambda{\bf 1}_U\in U\).
\item[\(\rm iii)\)] It follows from \eqref{eq:def_f-alg_c} and \eqref{eq:def_f-alg_g} that \(\lambda u=(\lambda{\bf 1}_U)\times u\)
holds for every \(\lambda\in\R\) and \(u\in U\), thus the multiplicative identity \({\bf 1}_U\) can be unambiguously denoted by \(1\).
\end{itemize}
Given any \(u,v\in U\), for the sake of brevity we will typically write \(uv\) instead of \(u\times v\).
\fr}\end{remark}
\begin{example}{\rm
The real line \(\R=(\R,+,\cdot,\leq,\cdot)\) is an \(f\)-algebra.
\fr}\end{example}
\begin{proposition}[Basic properties of \(f\)-algebras]\label{prop:basic_prop_f-alg}
Let \(U\) be an \(f\)-algebra. Then it holds that:
\begin{subequations}\begin{align}
\label{eq:basic_prop_f-alg_7}u^+ u^-=0,&\quad\text{ for every }u\in U,\\
\label{eq:basic_prop_f-alg_8}(uv)^+=u v^+,&\quad\text{ for every }u\in U^+\text{ and }v\in U,\\
\label{eq:prod_preserv_1}|u-v|=|u+v|,&\quad\text{ for every }u,v\in U\text{ with }u\wedge v=0,\\
\label{eq:prod_preserv_2}|u+v|=|u|+|v|,&\quad\text{ for every }u,v\in U\text{ with }|u|\wedge|v|=0,\\
\label{eq:prod_preserv}|uv|=|u||v|,&\quad\text{ for every }u,v\in U,\\
\label{eq:basic_prop_f-alg_9}uv\leq uw,&\quad\text{ for every }u\in U^+\text{ and }v,w\in U\text{ with }v\leq w.
\end{align}\end{subequations}
\end{proposition}
\begin{proof}\ \\
{\color{blue}\eqref{eq:basic_prop_f-alg_7}} Given that \(u^+\wedge u^-=0\) by \eqref{eq:basic_prop_f-alg_6b} and \(u^-\geq 0\), we obtain that
\(u^+ u^-\wedge u^-=0\) by \eqref{eq:def_f-alg_f}. Since also \(u^+\geq 0\), by using again \eqref{eq:def_f-alg_f} we can conclude
that \(u^+u^-=u^+u^-\wedge u^+u^-=0\).\\
{\color{blue}\eqref{eq:basic_prop_f-alg_8}} Since \(v^+\wedge v^-=0\) by \eqref{eq:basic_prop_f-alg_6b} and \(u\geq 0\),
we deduce from \eqref{eq:def_f-alg_f} that \(uv^+\wedge uv^-=0\), so
\[
(uv)^+\overset{\eqref{eq:basic_prop_f-alg_5}}=(uv^+-uv^-)^+\overset{\eqref{eq:basic_prop_f-alg_3}}=uv^+ +(-uv^-)\vee(-uv^+)
\overset{\eqref{eq:basic_prop_f-alg_2}}=uv^+ -uv^+\wedge uv^-=uv^+.
\]
{\color{blue}\eqref{eq:prod_preserv_1}} Note that \eqref{eq:basic_prop_f-alg_3} yields \((u-v)^+=(u-v)\vee 0=u-u\wedge v=u\) and \((u-v)^-=v\).
Then an application of \eqref{eq:basic_prop_f-alg_6} gives \(|u-v|=(u-v)^+ +(u-v)^-=u+v=|u+v|\), thus getting \eqref{eq:prod_preserv_1}.\\
{\color{blue}\eqref{eq:prod_preserv_2}} Let us start by observing that
\[\begin{split}
(u^+ +v^+)\wedge(u^- +v^-)&\overset{\eqref{eq:basic_prop_f-alg_6e}}\leq u^+\wedge u^- +v^+\wedge u^- +u^+\wedge v^- +v^+\wedge v^-\\
&\overset{\eqref{eq:basic_prop_f-alg_6b}}=v^+\wedge u^- +u^+\wedge v^-\leq 2\,|u|\wedge|v|=0.
\end{split}\]
Hence, \eqref{eq:prod_preserv_1} yields \(|u+v|=|(u^+ +v^+)-(u^- +v^-)|=|u^+ +v^+|+|u^- +v^-|=|u|+|v|\).\\
{\color{blue}\eqref{eq:prod_preserv}} Given that \(u^+\wedge u^-=v^+\wedge v^-=0\) by \eqref{eq:basic_prop_f-alg_6b}, we deduce from
\eqref{eq:def_f-alg_f} that \(w\wedge w'=0\) holds whenever \(w,w'\in\{u^+ v^+, u^+ v^-,u^- v^+,u^- v^-\}\) satisfy \(w\neq w'\).
Then by applying \eqref{eq:prod_preserv_2} we get
\[\begin{split}
|uv|&=|(u^+ -v^-)(v^+ -v^-)|=|u^+ v^+ -u^+ v^- -u^- v^+ +u^- v^-|\\&=u^+ v^+ +u^+ v^- +u^- v^+ +u^- v^-=(u^+ +u^-)(v^+ +v^-)=|u||v|.
\end{split}\]
{\color{blue}\eqref{eq:basic_prop_f-alg_9}} Since \(w-v\geq 0\), we know from \eqref{eq:def_f-alg_e} that \(uw-uv=(w-v)u\geq 0\), as desired.
\end{proof}
\begin{proposition}\label{prop:char_disjoint}
Let \(U\) be an \(f\)-algebra. Let \(S\subset U\) be a given non-empty set. Then it holds
\[
S\text{ is disjoint}\quad\Longrightarrow\quad uv=0,\text{ for every }u,v\in S\text{ such that }u\neq v.
\]
If in addition the \(f\)-algebra \(U\) is Archimedean, then the converse implication is verified as well.
\end{proposition}
\begin{proof}
Let us prove the first part of the statement. Fix any \(u,v\in S\) with \(u\neq v\) and \(|u|\wedge|v|=0\). We can argue as in the proof
of \eqref{eq:basic_prop_f-alg_7}: using \eqref{eq:def_f-alg_f} twice, we first obtain that \(|u||v|\wedge|v|=0\) and then that
\(|u||v|=|u||v|\wedge|u||v|=0\). Therefore, \eqref{eq:prod_preserv} yields \(|uv|=0\) and thus accordingly \(uv=0\).

To prove the second part of the statement, assume that \(U\) is Archimedean. We aim to show that if there exist \(u,v\in S\)
such that \(w\coloneqq|u|\wedge|v|\neq 0\), then \(S\) is not disjoint. Notice that there exists \(n\in\N\) such that \(nw\nleq 1\).
Denote \(w_+\coloneqq(nw-1)^+\) and \(w_-\coloneqq(nw-1)^-\), thus \(w_+\neq 0\). Observe also that \(w_-=1-nw\wedge 1\) by
\eqref{eq:basic_prop_f-alg_2},\eqref{eq:basic_prop_f-alg_3} and that \(w_-=(|u|\wedge|v|)^2\leq|u||v|\). Therefore,
\[
w_+=w_+(w_-+nw\wedge 1)=w_+ w_- +w_+(nw\wedge 1)\overset{\eqref{eq:basic_prop_f-alg_7}}=w_+(nw\wedge 1)\leq nw(nw\wedge 1)\leq(nw)^2\leq n^2|u||v|.
\]
Given that \(w_+\neq 0\), we finally deduce that \(|uv|=|u||v|\neq 0\), yielding the sought conclusion.
\end{proof}
\subsection{Idempotent elements}\label{ss:idempotent}
Given an \(f\)-algebra \(U\), we define the family of all \textbf{idempotent elements} of \(U\) as follows:
\begin{equation}\label{eq:idem}
\Idem(U)\coloneqq\big\{u\in U\;\big|\;u^2=u\big\},
\end{equation}
where we adopt the shorthand notation
\[
u^k\coloneqq\underset{k\text{ times}}{\underbrace{u\times\dots\times u}}\in U,\quad\text{ for every }u\in U\text{ and }k\in\N,
\]
with the convention that \(u^0\coloneqq 1\). Note that \(0\in\Idem(U)\) and \(1\in\Idem(U)\) for every \(f\)-algebra \(U\).
\begin{lemma}[Properties of \(\Idem(U)\)]\label{lem:prop_idem}
Let \(U\) be an \(f\)-algebra. Then the following properties hold:
\begin{itemize}
\item[\(\rm i)\)] \(uv\in\Idem(U)\) for every \(u,v\in\Idem(U)\).
\item[\(\rm ii)\)] \(u+v-2uv\in\Idem(U)\) for every \(u,v\in\Idem(U)\).
\item[\(\rm iii)\)] \(1-u\in\Idem(U)\) for every \(u\in\Idem(U)\).
\item[\(\rm iv)\)] \(0\leq u\leq 1\) for every \(u\in\Idem(U)\). In particular, \(\Idem(U)\) is order-bounded in \(U\).
\item[\(\rm v)\)] If \(u,v\in\Idem(U)\) satisfy \(uv=0\), then \(u+v\in\Idem(U)\) and \(u+v=u\vee v\).
\item[\(\rm vi)\)] If \(u\in U\) and \(v\in\Idem(U)\), then \(u-uv\) and \(v\) are disjoint.
\end{itemize}
\end{lemma}
\begin{proof}\ \\
{\color{blue}\(\rm i)\)} Trivially, it holds that \((uv)^2=uvuv=u^2 v^2=uv\).\\
{\color{blue}\(\rm ii)\)} It follows from the observation that
\[\begin{split}
(u+v-2uv)^2&=u^2+uv-2u^2 v+vu+v^2-2uv^2-2u^2 v-2uv^2+4u^2 v^2\\&=u+uv-2uv+uv+v-2uv-2uv-2uv+4uv=u+v-2uv.
\end{split}\]
{\color{blue}\(\rm iii)\)} Just observe that \((1-u)(1-u)=1-2u+u^2=1-2u+u=1-u\).\\
{\color{blue}\(\rm iv)\)} Given any \(u\in U\), it holds \(u=u^+ -u^-\) and \(u^+ u^-=0\) by \eqref{eq:basic_prop_f-alg_5}
and \eqref{eq:basic_prop_f-alg_7}, respectively. Then
\[
u^2=(u^+ -u^-)(u^+ -u^-)=(u^+)^2-u^+ u^- -u^-u^+ +(u^-)^2=(u^+)^2+(u^-)^2\geq 0,
\]
where the last inequality follows from \eqref{eq:def_f-alg_e}. In particular, \(u=u^2\geq 0\) for every \(u\in\Idem(U)\).
Since \(1-u\in\Idem(U)\) by item ii), we also have that \(1-u\geq 0\), or equivalently that \(u\leq 1\).\\
{\color{blue}\(\rm v)\)} First, we may compute \((u+v)^2=u^2+2uv+v^2=u+v\), which shows that \(u+v\in\Idem(U)\). Moreover,
thanks to the fact that \(u\leq u+v\) and \(v\leq u+v\), we have that \(u\vee v\leq u+v\). Conversely, it holds that
\(u(u+v)=u^2\leq u(u\vee v)\) and \((1-u)(u+v)=u+v-u^2-uv=(1-u)v\leq(1-u)(u\vee v)\), thus accordingly
\(u+v=u(u+v)+(1-u)(u+v)\leq u(u\vee v)+(1-u)(u\vee v)=u\vee v\).\\
{\color{blue}\(\rm vi)\)} First of all, we aim to show that \((1-v)\wedge v=0\). Item iv) ensures that \((1-v)\wedge v\geq 0\).
Conversely, if \(w\in U\) is a lower bound for \(\{v,1-v\}\), then using \eqref{eq:basic_prop_f-alg_9} we can estimate
\[
w=vw+(1-v)w\leq v(1-v)+(1-v)v=v-v^2+v-v^2=0,
\]
which yields \((1-v)\wedge v=0\). Hence, \eqref{eq:def_f-alg_f} ensures that \(|u-uv|\wedge v=\big(|u|(1-v)\big)\wedge v=0\).
\end{proof}

Observe that if \(\phi\colon U\to V\) is a homomorphism of \(f\)-algebras, then it holds that
\begin{equation}\label{eq:img_Idem}
\phi(u)\in\Idem(V),\quad\text{ for every }u\in\Idem(U).
\end{equation}
Indeed, since \(\phi\) preserves the multiplication, we have \(\phi(u)^2=\phi(u^2)=\phi(u)\) for every \(u\in\Idem(U)\).
\begin{remark}\label{rmk:char_idem_ineq}{\rm
Given an \(f\)-algebra \(U\) and two elements \(u,v\in\Idem(U)\), it holds that
\[
u\leq v\qquad\Longleftrightarrow\qquad uv=u.
\]
Indeed, if \(uv=u\), then \(u=uv\leq v\). On the other hand, if \(u\leq v\), then \(u=u^2\leq uv\leq u\).
\fr}\end{remark}
\begin{definition}[Finite partition]
Let \(U\) be an \(f\)-algebra. Then a given set \((u_i)_{i=1}^n\subset\Idem(U)\) is said to be
a \textbf{finite partition} of an element \(u\in\Idem(U)\) provided it is disjoint and it satisfies
\[
u_1+\ldots+u_n=u.
\]
We denote by \(\mathcal P_f(u)\) the family of all finite partitions of \(u\).
\end{definition}

Thanks to Lemma \ref{lem:prop_idem} v), a disjoint family \((u_i)_{i=1}^n\subset\Idem(U)\) belongs to \(\mathcal P_f(u)\) if and only if
\[
\sup\{u_1,\ldots,u_n\}=u.
\]
Notice also that if \(\phi\colon U\to V\) is a homomorphism of \(f\)-algebras, then it holds that
\begin{equation}\label{eq:img_fin_partition}
\big(\phi(u_i)\big)_{i=1}^n\in\mathcal P_f\big(\phi(u)\big),\quad\text{ for every }u\in\Idem(U)\text{ and }(u_i)_{i=1}^n\in\mathcal P_f(u).
\end{equation}
Indeed, one has \(\phi(u_i)\in\Idem\big(\phi(v)\big)\) for every \(i=1,\ldots,n\) by \eqref{eq:img_Idem}, the elements \(\phi(u_1),\ldots,\phi(u_n)\) are disjoint
by \eqref{eq:img_disjoint}, and \(\phi(u_1)+\ldots+\phi(u_n)=\phi(u_1+\ldots+u_n)=\phi(u)\) by the linearity of \(\phi\).
\begin{definition}[Simple elements]
The \textbf{simple elements} of an \(f\)-algebra \(U\) are defined as
\[
\mathcal S(U)\coloneqq\bigg\{\sum_{i=1}^n\lambda_i u_i\;\bigg|\;n\in\N,\,(\lambda_i)_{i=1}^n\subset\R,\,(u_i)_{i=1}^n\in\mathcal P_f(\boldsymbol{1}_U)\bigg\}\subset U.
\]
The family of all \textbf{non-negative simple elements} of \(U\) is defined as \(\mathcal S^+(U)\coloneqq\mathcal S(U)\cap U^+\).
\end{definition}
\begin{lemma}\label{lem:finite_part}
Let \(U\) be an \(f\)-algebra and \(u\in\Idem(U)\). Then it holds that
\begin{equation}\label{eq:inters_P_f}
(u_i v_j)_{i,j}\in\mathcal P_f(u),\quad\text{ for every }(u_i)_{i=1}^n,(v_j)_{j=1}^m\in\mathcal P_f(u).
\end{equation}
In particular, the space \(\mathcal S(U)\) is an \(f\)-subalgebra of \(U\). More precisely, it holds that
\[
u+v=\sum_{i,j}(\lambda_i+\mu_j)u_i v_j,\quad uv=\sum_{i,j}\lambda_i\mu_j u_i v_j,\quad
u\vee v=\sum_{i,j}(\lambda_i\vee\mu_j)u_i v_j,\quad u\wedge v=\sum_{i,j}(\lambda_i\wedge\mu_j)u_i v_j,
\]
for every \(u=\sum_{i=1}^n\lambda_i u_i\in\mathcal S(U)\) and \(v=\sum_{j=1}^m\mu_j v_j\in\mathcal S(U)\).
\end{lemma}
\begin{proof}
Let us only check \eqref{eq:inters_P_f}. Once \eqref{eq:inters_P_f} is established, the remaining part of the statement follows via
elementary computations. Fix any \((u_i)_{i=1}^n,(v_j)_{j=1}^m\in\mathcal P_f(u)\). Lemma \ref{lem:prop_idem} i) ensures that
\(u_i v_j\in\Idem(U)\) for every \(i=1,\ldots,n\) and \(j=1,\ldots,m\). Moreover, whenever \((i,j)\neq(i',j')\) we have that
\((u_i v_j)\wedge(u_{i'}v_{j'})\leq(u_i\wedge u_{i'})\wedge(v_j\wedge v_{j'})=0\), thus \((u_i v_j)_{i,j}\) is a disjoint set.
Finally, it holds that \(\sum_{i,j}u_i v_j=(u_1+\ldots+u_n)(v_1+\ldots+v_m)=u\), which gives \((u_i v_j)_{i,j}\in\mathcal P_f(u)\).
\end{proof}
\subsection{Localisable \texorpdfstring{\(f\)}{f}-algebras}\label{ss:local_f-alg}
Let us now introduce the concept of \emph{localisable \(f\)-algebra}, which is a Dedekind \(\sigma\)-complete \(f\)-algebra `having plenty of idempotent elements'. Namely:

\begin{definition}[Localisable \(f\)-algebra]\label{def:localisable_f-alg}
Let \(U\) be a Dedekind \(\sigma\)-complete \(f\)-algebra whose multiplication map is \(\sigma\)-order-continuous on \(U^+\times U^+\).
Then we say that \(U\) is \textbf{localisable} provided the space of simple elements \(\mathcal S(U)\) is super-order-dense in \(U\).
By a \textbf{homomorphism of localisable \(f\)-algebras} we mean a \(\sigma\)-order-continuous homomorphism of \(f\)-algebras.
\end{definition}
\begin{remark}{\rm
On any Dedekind \(\sigma\)-complete \(f\)-algebra \(U\), the sum operator \(+\colon U\times U\to U\) is \(\sigma\)-order-continuous
on \(U^+\times U^+\). Indeed, if \((u_n)_{n\in\N},(v_m)_{m\in\N}\subset U^+\) are non-decreasing sequences, and we set
\(u\coloneqq\sup_{n\in\N}u_n\in U^+\) and \(v\coloneqq\sup_{m\in\N}v_m\in U^+\), then for any \(n,m\in\N\) we have that
\[
u_n=(u_n+v_m)-v_m\leq(u_{n\vee m}+v_{n\vee m})-v_m\leq\sup_{k\in\N}(u_k+v_k)-v_m.
\]
Thanks to the arbitrariness of \(n\in\N\), we deduce that \(u\leq\sup_{k\in\N}(u_k+v_k)-v_m\). By arbitrariness of \(m\in\N\),
we conclude that \(u+v\leq\sup_{k\in\N}(u_k+v_k)\). The converse inequality is trivial.
\fr}\end{remark}
\begin{lemma}\label{lem:order_idem}
Let \(U\) be a localisable \(f\)-algebra. Then it holds that
\[
\sup_{n\in\N}u_n\in\Idem(U),\qquad\inf_{n\in\N}u_n\in\Idem(U),\quad\text{ for every }(u_n)_{n\in\N}\subset\Idem(U).
\]
\end{lemma}
\begin{proof}
Since the set \(\Idem(U)\) is order-bounded by Lemma \ref{lem:prop_idem} iv), both \(v\coloneqq\sup_{n\in\N}u_n\in U^+\)
and \(w\coloneqq\inf_{n\in\N}u_n\in U^+\) exist thanks to the Dedekind \(\sigma\)-completeness of \(U\). Define \(u'_1\coloneqq u_1\)
and \(u'_n\coloneqq u_n-\sum_{k<n}u_n u'_k\) for every \(n\geq 2\). By using items iii), v), vi) of Lemma \ref{lem:prop_idem} and an
induction argument, one can show that the sequence \((u'_n)_{n\in\N}\) is disjoint and made of idempotent elements. Lemma \ref{lem:prop_idem}
v) also yields \(\sup_{k\leq n}u'_k=\sup_{k\leq n}u_k\) for every \(n\in\N\), so that accordingly
\[
\sup_{n\in\N}u'_n=\sup_{n\in\N}\sup_{k\leq n}u_k=\sup_{n\in\N}u_n=v.
\]
Now define \(v_n\coloneqq\sum_{k=1}^n u'_k\) for every \(n\in\N\). Lemma \ref{lem:prop_idem} v) ensures that
\((v_n)_{n\in\N}\subset\Idem(U)\) and that \(v_n=\sup_{k\leq n}u'_k\) for every \(n\in\N\), thus \(v=\sup_{n\in\N}v_n\).
Given that the sequence \((v_n)_{n\in\N}\) is non-decreasing by construction, the \(\sigma\)-order-continuity of the multiplication
on \(U^+\times U^+\) guarantees that \(v^2=(\sup_{n\in\N}v_n)^2=\sup_{n\in\N}v_n^2=\sup_{n\in\N}v_n=v\), proving that \(v\in\Idem(U)\).
Finally, notice that we have \(1-w=\sup_{n\in\N}(1-u_n)\in\Idem(U)\), so that \(w\in\Idem(U)\) by Lemma \ref{lem:prop_idem} iii).
\end{proof}
\begin{definition}[Countable partition]
Let \(U\) be a Dedekind \(\sigma\)-complete \(f\)-algebra. Then a disjoint family \((u_n)_{n\in\N}\subset\Idem(U)\)
is said to be a \textbf{countable partition} of \(u\in\Idem(U)\) provided
\[
\sup_{n\in\N}u_n=u.
\]
We denote by \(\mathcal P(u)\) the family of all countable partitions of \(u\). Observe that \(\mathcal P_f(u)\subset\mathcal P(u)\).
\end{definition}
\begin{proposition}\label{prop:inters_partitions}
Let \(U\) be a localisable \(f\)-algebra and \(u\in\Idem(U)\). Then it holds that
\[
(u_n v_m)_{n,m\in\N}\in\mathcal P(u),\quad\text{ for every }(u_n)_{n\in\N},(v_m)_{m\in\N}\in\mathcal P(u).
\]
\end{proposition}
\begin{proof}
Lemma \ref{lem:prop_idem} i) ensures that \(u_n v_m\in\Idem(U)\) for every \(n,m\in\N\). Arguing as in Lemma \ref{lem:finite_part},
we see that \((u_n v_m)_{n\in\N}\) is a disjoint set. It remains to show that \(v\coloneqq\sup_{n,m\in\N}u_n v_m=u\).
Since \(u_n v_m\leq 1\) for all \(n,m\in\N\), we have \(v\leq 1\). Let \(w\in U\) be an upper bound for \((u_n v_m)_{n,m\in\N}\).
Then \(u_n(w-v_m)=u_n(w-u_n v_m)\geq 0\) for every \(n,m\in\N\), so that \(u_n(w-v_m)^+\geq u_n(w-v_m)^-\) holds for every \(n,m\in\N\)
as a consequence of \eqref{eq:basic_prop_f-alg_8}, thus accordingly
\[
(w-v_m)^+=\sup_{n\in\N}u_n(w-v_m)^+\geq\sup_{n\in\N}u_n(w-v_m)^-=(w-v_m)^-,\quad\text{ for every }m\in\N.
\]
This means that \(w\geq v_m\) for every \(m\in\N\), which gives \(w\geq 1\). We conclude that \(v=u\).
\end{proof}

Observe that if \(\phi\colon U\to V\) a homomorphism of localisable \(f\)-algebras, then it holds that
\begin{equation}\label{eq:img_partition}
\big(\phi(u_n)\big)_{n\in\N}\in\mathcal P(\phi(u)),\quad\text{ for every }u\in\Idem(U)\text{ and }(u_n)_{n\in\N}\in\mathcal P(u).
\end{equation}
Indeed, the sequence \(\big(\phi(u_n)\big)_{n\in\N}\subset V\) is made of idempotent elements by \eqref{eq:img_Idem}, is a disjoint set by \eqref{eq:img_disjoint},
and \(\sup_{n\in\N}\phi(u_n)=\phi\big(\sup_{n\in\N}u_n\big)=\phi(u)\) thanks to the \(\sigma\)-order-continuity of \(\phi\).
\subsection{Metric \texorpdfstring{\(f\)}{f}-structures}\label{ss:metric_f-struct}
For our purposes, the algebraic and order properties of a localisable \(f\)-algebra are not sufficient. Rather, we want to consider localisable \(f\)-algebras
(and Riesz spaces) endowed with a well-behaved complete distance. In this regard, the first concept we introduce is that of \emph{metric Riesz space}:
\begin{definition}[Metric Riesz space]\label{def:metric_Riesz}
By a \textbf{metric Riesz space} we mean a couple \((U,\sfd_U)\) -- where \(U\) is a Dedekind \(\sigma\)-complete Riesz space
and \(\sfd_U\) is a complete distance on \(U\) -- such that:
\begin{itemize}
\item[\(\rm i)\)] The identity \(\sfd_U(u,0)=\sfd_U(|u|,0)\) holds for every \(u\in U\).
\item[\(\rm ii)\)] The distance \(\sfd_U\) is \textbf{translation-invariant}, in the sense that
\[
\sfd_U(u,v)=\sfd_U(u+w,v+w),\quad\text{ for every }u,v,w\in U.
\]
\item[\(\rm iii)\)] The distance-from-zero function \(\sfd_U(\cdot,0)\colon U^+\to\R^+\) is order-preserving.
\end{itemize}
A \textbf{homomorphism of metric Riesz spaces} is a Lipschitz homomorphism of Riesz spaces.
\end{definition}
\begin{remark}{\rm
Given a metric Riesz space \((U,\sfd_U)\), it holds that
\begin{equation}\label{eq:sum_subadd}
\sfd_U(u+v,0)\leq\sfd_U(u,0)+\sfd_U(v,0),\quad\text{ for every }u,v\in U.
\end{equation}
Indeed, by using the translation-invariance of \(\sfd_U\), we obtain that
\[
\sfd_U(u+v,0)=\sfd_U(u,-v)\leq\sfd_U(u,0)+\sfd_U(0,-v)=\sfd_U(u,0)+\sfd_U(v,0).
\]
Repeatedly applying \eqref{eq:sum_subadd}, we get \(\sfd_U\big(\sum_{i=1}^n u_i,0\big)\leq\sum_{i=1}^n\sfd_U(u_i,0)\)
for all \(u_1,\ldots,u_n\in U\).
\fr}\end{remark}

In Definitions \ref{def:metric_f-algebra} and \ref{def:metric_f-structure} below, given two metric spaces \((\X,\sfd_\X)\) and \((\Y,\sfd_\Y)\),
we will consider the distance \(\sfd_\X\times\sfd_\Y\) on the Cartesian product \(\X\times\Y\), which is given by
\[
(\sfd_\X\times\sfd_\Y)\big((x,y),(\tilde x,\tilde y)\big)\coloneqq\sfd_\X(x,\tilde x)+\sfd_\Y(y,\tilde y),\quad\text{ for every }(x,y),(\tilde x,\tilde y)\in\X\times\Y.
\]
Next, we introduce the family of \emph{metric \(f\)-algebras}:
\begin{definition}[Metric \(f\)-algebra]\label{def:metric_f-algebra}
By a \textbf{metric \(f\)-algebra} we mean a couple \((U,\sfd_U)\), where:
\begin{itemize}
\item[\(\rm i)\)] \(U\) is a localisable \(f\)-algebra and \((U,\sfd_U)\) is a metric Riesz space.
\item[\(\rm ii)\)] The multiplication \(\times\colon U\times U\to U\) is continuous from \((U\times U,\sfd_U\times\sfd_U)\) to \((U,\sfd_U)\).
\item[\(\rm iii)\)] The family \(\mathcal S(U)\) of all simple elements of \(U\) is dense in \((U,\sfd_U)\).
\item[\(\rm iv)\)] \(\sfd_U(\varepsilon{\bf 1}_U,0)\to 0\) as \(\varepsilon\searrow 0\).
\end{itemize}
A \textbf{homomorphism of metric \(f\)-algebras} is a Lipschitz homomorphism of localisable \(f\)-algebras.
\end{definition}

Having the notions of metric Riesz space and of metric \(f\)-algebra at disposal, we can finally introduce \emph{metric \(f\)-structures}:
\begin{definition}[Metric \(f\)-structure]\label{def:metric_f-structure}
A \textbf{metric \(f\)-structure} is a triple \((\mathcal U,U,V)\), where:
\begin{itemize}
\item[\(\rm i)\)] \(\mathcal U\) is a localisable \(f\)-algebra.
\item[\(\rm ii)\)] \(U=(U,\sfd_U)\) is a metric \(f\)-algebra such that \(U\) is a solid \(f\)-subalgebra of \(\mathcal U\).
\item[\(\rm iii)\)] \(V=(V,\sfd_V)\) is metric Riesz space such that \(V\) is a solid Riesz subspace of \(\mathcal U\).
\item[\(\rm iv)\)] It holds \(UV=V\) and the multiplication is continuous from \((U\times V,\sfd_U\times\sfd_V)\) to \((V,\sfd_V)\).
\item[\(\rm v)\)] Given any \((u_n)_{n\in\N}\in\mathcal P(\boldsymbol{1}_U)\) and \(\varepsilon>0\), there exists \(\delta>0\) such that for any
\((v_n)_{n\in\N}\subset V^+\) with \(\sum_{n\in\N}\sfd_V(u_n v_n,0)\leq\delta\) it holds that
\[
(u_n v_n)_{n\in\N}\,\text{ is order-bounded in }V,\qquad\sfd_V\Big(\sup_{n\in\N}u_n v_n,0\Big)\leq\varepsilon.
\]
\end{itemize}
We say that \((\mathcal U,U,V)\) is a \textbf{CR metric \(f\)-structure} provided the spaces \(\mathcal U\), \(U\), and \(V\) are CR.

A \textbf{homomorphism of metric \(f\)-structures} between two metric \(f\)-structures \((\mathcal U_1,U_1,V_1)\) and \((\mathcal U_2,U_2,V_2)\)
is a homomorphism \(\varphi\colon\mathcal U_1\to\mathcal U_2\) of \(f\)-algebras such that \(\varphi|_{U_1}\colon U_1\to U_2\) is a homomorphism
of metric \(f\)-algebras and \(\varphi|_{V_1}\colon V_1\to V_2\) a homomorphism of metric Riesz spaces.
\end{definition}
\begin{example}\label{ex:U=V}{\rm
If \(U=(U,\sfd_U)\) is a metric \(f\)-algebra, then \((U,U,U)\) is a metric \(f\)-structure.
\fr}\end{example}

As we discussed in the introduction, in order to study dual modules (and, more generally, spaces of homomorphisms) we also have to define the
\emph{dual systems} of metric \(f\)-structures:
\begin{definition}[Dual system of metric \(f\)-structures]\label{def:dual_system}
A quintuplet \((\mathcal U,U,V,W,Z)\) is said to be a \textbf{dual system of metric \(f\)-structures} provided the following conditions are verified:
\begin{itemize}
\item[\(\rm i)\)] \((\mathcal U,U,V)\), \((\mathcal U,U,W)\), and \((\mathcal U,U,Z)\) are metric \(f\)-structures.
\item[\(\rm ii)\)] It holds \(Z=VW\) and the multiplication is continuous from \((V\times W,\sfd_V\times\sfd_W)\) to \((Z,\sfd_Z)\).
\end{itemize}
We say that \((\mathcal U,U,V,W,Z)\) is a \textbf{complete dual system of metric \(f\)-structures} if in addition:
\begin{itemize}
\item[\(\rm iii)\)] \(U\), \(V\), \(W\), \(Z\) are Dedekind complete and the multiplication is order-continuous from \(U^+\times V^+\)
to \(V^+\), from \(U^+\times W^+\) to \(W^+\), from \(U^+\times Z^+\) to \(Z^+\), and from \(V^+\times W^+\) to \(Z^+\).
\end{itemize}
Also, we say that \((\mathcal U,U,V,W,Z)\) is \textbf{CR} if in addition \((\mathcal U,U,V)\), \((\mathcal U,U,W)\), and \((\mathcal U,U,Z)\) are CR.

By a \textbf{homomorphism of dual systems} between two given dual systems of metric \(f\)-structures \((\mathcal U_1,U_1,V_1,W_1,Z_1)\) and
\((\mathcal U_2,U_2,V_2,W_2,Z_2)\) we mean a map \(\varphi\colon\mathcal U_1\to\mathcal U_2\) that is a homomorphism of metric \(f\)-structures
from \((\mathcal U_1,U_1,V_1)\) to \((\mathcal U_2,U_2,V_2)\), from \((\mathcal U_1,U_1,W_1)\) to \((\mathcal U_2,U_2,W_2)\), and
from \((\mathcal U_1,U_1,Z_1)\) to \((\mathcal U_2,U_2,Z_2)\).
\end{definition}
\begin{example}\label{ex:UVUV}{\rm
Let \((\mathcal U,U,V)\) be a metric \(f\)-structure. Then \((\mathcal U,U,V,U,V)\) is a dual system of metric \(f\)-structures.
If \(U\) is Dedekind complete and the multiplication map is order-continuous from \(U^+\times V^+\) to \(V^+\), then
\((\mathcal U,U,V,U,V)\) is a complete dual system of metric \(f\)-structures.
\fr}\end{example}
\begin{remark}\label{rmk:invert_dual_syst}{\rm
If \((\mathcal U,U,V,W,Z)\) is a dual system of metric \(f\)-structures, then \((\mathcal U,U,W,V,Z)\) is a dual system of
metric \(f\)-structures as well. Moreover, if \((\mathcal U,U,V,W,Z)\) is a complete (resp.\ CR) dual system, then
\((\mathcal U,U,W,V,Z)\) is a complete (resp.\ CR) dual system as well.
\fr}\end{remark}
\section{Normed modules over a metric \texorpdfstring{\(f\)}{f}-structure}\label{s:normed_mod}
In Sections \ref{ss:def_norm_mod} and \ref{ss:hom_norm_mod} we introduce the category of Banach modules over a metric \(f\)-structure;
in the former we study the objects, while in the latter we study the morphisms. In Section \ref{ss:constr_norm_mod} we prove some existence results concerning
Banach modules and their homomorphisms, as well as some of their consequences. In Sections \ref{ss:Hahn-Banach}, \ref{ss:Hilbert_mod}, and \ref{ss:dim_decomp}
we study the Hahn--Banach theorem, the class of Hilbert modules, and the dimensional decomposition of a Banach module, respectively.
\subsection{Definitions and basic properties}\label{ss:def_norm_mod}
First of all, let us give the definition of \emph{normed/Banach module} over a metric \(f\)-structure:
\begin{definition}[Normed module]\label{def:norm_U-mod}
Let \((\mathcal U,U,V)\) be a metric \(f\)-structure and \(\mathscr M\) a module over \(U\). Then we say that \(\mathscr M\)
is a \textbf{\(V\)-normed \(U\)-module} provided it is endowed with a map \(|\cdot|\colon\mathscr M\to V^+\) -- called a
\textbf{\(V\)-pointwise norm} operator on \(\mathscr M\) -- such that the following properties are verified:
\begin{itemize}
\item[\(\rm i)\)] Given any \(u\in U\) and \(v,w\in\mathscr M\), it holds that
\begin{subequations}\begin{align}\label{eq:ptwse_norm_1}
&|v|=0\quad\Longleftrightarrow\quad v=0,\\
\label{eq:ptwse_norm_2}&|v+w|\leq|v|+|w|,\\
\label{eq:ptwse_norm_3}&|u\cdot v|=|u||v|.
\end{align}\end{subequations}
\item[\(\rm ii)\)] \textsc{Glueing property.} Let \((u_n)_{n\in\N}\in\mathcal P(\boldsymbol{1}_U)\) and \((v_n)_{n\in\N}\subset\mathscr M\) be chosen
so that the family \((|u_n\cdot v_n|)_{n\in\N}\) is order-bounded in \(V\). Then there exists an element \(v\in\mathscr M\) such that
\begin{equation}\label{eq:glueing}
u_n\cdot v=u_n\cdot v_n,\quad\text{ for every }n\in\N.
\end{equation}
We call \(\Adm(\mathscr M)\) the set of all families \((u_n,v_n)_{n\in\N}\) as above, while \(\sum_{n\in\N}u_n\cdot v_n\) stands
for the element \(v\in\mathscr M\) satisfying \eqref{eq:glueing} -- whose uniqueness follows from Lemma \ref{lem:locality_prop} below.
\end{itemize}
Moreover, we endow the space \(\mathscr M\) with the distance \(\sfd_{\mathscr M}\), which is defined as
\begin{equation}\label{eq:def_dist_NMod}
\sfd_{\mathscr M}(v,w)\coloneqq\sfd_V(|v-w|,0),\quad\text{ for every }v,w\in\mathscr M.
\end{equation}
Whenever \((\mathscr M,\sfd_{\mathscr M})\) is a complete metric space, we say that \(\mathscr M\) is a \textbf{\(V\)-Banach \(U\)-module}.
\end{definition}

\begin{lemma}[Locality property]\label{lem:locality_prop}
Let \((\mathcal U,U,V)\) be a metric \(f\)-structure and let \(\mathscr M\) be a \(V\)-normed \(U\)-module.
Let \((u_n)_{n\in\N}\in\mathcal P({\bf 1}_U)\) and \(v\in\mathscr M\) satisfy \(u_n\cdot v=0\) for every \(n\in\N\). Then \(v=0\).
\end{lemma}
\begin{proof}
Given that the multiplication map is \(\sigma\)-order-continuous on \(\mathcal U^+\times\mathcal U^+\), we deduce that
\[
|v|=|v|\sup_{n\in\N}u_n=\sup_{n\in\N}u_n|v|=\sup_{n\in\N}|u_n\cdot v|=0,
\]
whence it follows that \(v=0\), as we claimed in the statement.
\end{proof}

Given any non-empty subset \(S\) of a \(V\)-normed \(U\)-module \(\mathscr M\), we denote by \(\mathscr G(S)\subset\mathscr M\)
the family of those elements that can be obtained by glueing together elements of \(S\). Namely, we set
\[
\mathscr G(S)=\mathscr G_{\mathscr M}(S)\coloneqq\bigg\{\sum_{n\in\N}u_n\cdot v_n\;\bigg|\;(u_n)_{n\in\N}\in\mathcal P(\boldsymbol{1}_U),\,
(v_n)_{n\in\N}\subset S,\,(u_n,v_n)_{n\in\N}\in\Adm(\mathscr M)\bigg\}.
\]
Observe that if \(S\) is a vector subspace of \(\mathscr M\), then \(\mathscr G(S)\) is a vector subspace of \(\mathscr M\) as well.
\begin{proposition}\label{prop:Riesz_is_Banach_module}
Let \((\mathcal U,U,V)\) be a metric \(f\)-structure. Then \(V\) is a \(V\)-Banach \(U\)-module, the scalar multiplication
\(\cdot\colon U\times V\to V\) being given by the multiplication \(\times\) in \(\mathcal U\). Moreover, it holds
\begin{equation}\label{eq:glue_V_as_NMod}
\sum_{n\in\N}u_n v_n=\sup_{n\in\N}u_n v_n^+-\sup_{n\in\N}u_n v_n^-,\quad\text{ for every }(u_n,v_n)_{n\in\N}\in\Adm(V).
\end{equation}
\end{proposition}
\begin{proof}
The fact that \(V\) is a \(U\)-module verifying item i) of Definition \ref{def:norm_U-mod} readily follows from the very definition
of a metric \(f\)-algebra. Moreover, the distance on \(V\) defined as in \eqref{eq:def_dist_NMod} coincides with the original distance
\(\sfd_V\) itself, which is complete by assumption. It only remains to check the validity of the glueing property. To this aim,
fix any \((u_n,v_n)_{n\in\N}\in\Adm(V)\). In particular, both sequences \((u_n v_n^+)_{n\in\N}\) and \((u_n v_n^-)_{n\in\N}\)
are order-bounded, thus the Dedekind \(\sigma\)-completeness of \(V\) yields existence of \(w_+\coloneqq\sup_{n\in\N}u_n v_n^+\in V^+\)
and \(w_-\coloneqq\sup_{n\in\N}u_n v_n^-\in V^+\). We claim that
\begin{equation}\label{eq:glue_V_as_NMod_aux}
u_n(w_+ - w_-)=u_n v_n,\quad\text{ for every }n\in\N.
\end{equation}
To prove it, notice that \(u_n w_+=u_n v_n^+\) for every \(n\in\N\): the inequality \(\geq\) is trivial, while to get the converse one
it suffices to observe that \((1-u_n)w_+ +u_n v_n^+\) is an upper bound for \((u_m v_m^+)_{m\in\N}\). Similarly, one can show
that \(u_n w_-=u_n v_n^-\), whence it follows that \(u_n(w_+ -w_-)=u_n v_n^+ -u_n v_n^-\), yielding \eqref{eq:glue_V_as_NMod_aux}.
This proves the validity of the glueing property, as well as the formula \eqref{eq:glue_V_as_NMod}.
\end{proof}
\begin{proposition}[Continuity of normed module operations]\label{prop:cont_norm_mod_ops}
Let \((\mathcal U,U,V)\) be a metric \(f\)-structure and \(\mathscr M\) a \(V\)-normed \(U\)-module.
Then the following properties are verified:
\begin{itemize}
\item[\(\rm i)\)] The map \(|\cdot|\colon\mathscr M\to V^+\) is \(1\)-Lipschitz from \((\mathscr M,\sfd_{\mathscr M})\) to \((V,\sfd_V)\).
\item[\(\rm ii)\)] The map \(+\colon\mathscr M\times\mathscr M\to\mathscr M\) is \(1\)-Lipschitz from
\((\mathscr M\times\mathscr M,\sfd_{\mathscr M}\times\sfd_{\mathscr M})\) to \((\mathscr M,\sfd_{\mathscr M})\).
\item[\(\rm iii)\)] The map \(\cdot\colon U\times\mathscr M\to\mathscr M\) is continuous from \((U\times\mathscr M,\sfd_U\times\sfd_{\mathscr M})\)
to \((\mathscr M,\sfd_{\mathscr M})\).
\end{itemize}
\end{proposition}
\begin{proof}\ \\
{\color{blue}\(\rm i)\)} Given that \(|v|\leq|v-w|+|w|\) and \(|w|\leq|w-v|+|v|\) hold for every \(v,w\in\mathscr M\), we deduce that
\[
\big||v|-|w|\big|\leq|v-w|,\quad\text{ for every }v,w\in\mathscr M.
\]
In particular, for any \(v,w\in\mathscr M\) one has \(\sfd_V(|v|,|w|)=\sfd_V\big(\big||v|-|w|\big|,0)\leq\sfd_V(|v-w|,0)=\sfd_{\mathscr M}(v,w)\),
which shows that \(|\cdot|\colon\mathscr M\to V^+\) is a \(1\)-Lipschitz mapping from \((\mathscr M,\sfd_{\mathscr M})\) to \((V,\sfd_V)\), as required.\\
{\color{blue}\(\rm ii)\)} For any \(v,v',w,w'\in\mathscr M\) we have \(|(v+w)-(v'+w')|\leq|v-v'|+|w-w'|\), thus accordingly
\[\begin{split}
\sfd_{\mathscr M}(v+w,v'+w')&=\sfd_V\big(\big|(v+w)-(v'+w')\big|,0\big)\leq\sfd_V\big(|v-v'|+|w-w'|,0\big)\\
&\leq\sfd_V(|v-v'|,0)+\sfd_V(|w-w'|,0)=\sfd_{\mathscr M}(v,v')+\sfd_{\mathscr M}(w,w')\\
&=(\sfd_{\mathscr M}\times\sfd_{\mathscr M})\big((v,v'),(w,w')\big),
\end{split}\]
for every \(v,v',w,w'\in\mathscr M\). This proves that \(+\) is \(1\)-Lipschitz from
\((\mathscr M\times\mathscr M,\sfd_{\mathscr M}\times\sfd_{\mathscr M})\) to \((\mathscr M,\sfd_{\mathscr M})\).\\
{\color{blue}\(\rm iii)\)} Fix \((u_n)_{n\in\N}\subset U\) and \(u\in U\) with \(\lim_{n\to\infty}\sfd_U(u_n,u)=0\).
Fix \((v_n)_{n\in\N}\subset\mathscr M\) and \(v\in\mathscr M\) with \(\lim_{n\to\infty}\sfd_{\mathscr M}(v_n,v)=0\).
The continuity of \(|\cdot|\colon\mathscr M\to V^+\) from i) ensures that \(|u_n|\to|u|\) in \((U,\sfd_U)\) and
\(|v_n-v|\to 0\) in \((V,\sfd_V)\). Hence, by letting \(n\to\infty\) in \(|u_n\cdot v_n-u\cdot v|\leq|u_n||v_n-v|+|u_n-u||v|\) we obtain
\(\lim_{n\to\infty}\sfd_{\mathscr M}(u_n\cdot v_n,u\cdot v)=0\), which shows that \(\cdot\colon U\times\mathscr M\to\mathscr M\) is continuous.
\end{proof}
Let \((\mathcal U,U,V)\) be a metric \(f\)-structure and let \(\mathscr M\) be a \(V\)-normed \(U\)-module. Then a given \(U\)-submodule \(\mathscr N\)
of \(\mathscr M\) is said to be a \textbf{\(V\)-normed \(U\)-submodule} of \(\mathscr M\) provided it satisfies
\[
\mathscr G_{\mathscr M}(\mathscr N)=\mathscr N.
\]
In the case where \(\mathscr M\) is a \(V\)-Banach \(U\)-module and \(\mathscr N\) is \(\sfd_{\mathscr M}\)-closed
in \(\mathscr M\), we say that \(\mathscr N\) is a \textbf{\(V\)-Banach \(U\)-submodule} of \(\mathscr M\).
These are some useful examples of \(V\)-Banach \(U\)-submodule:
\begin{itemize}
\item The `localised' module \(u\cdot\mathscr M\coloneqq\{u\cdot v\,:\,v\in\mathscr M\}\) for every \(u\in\Idem(U)\).
\item The `one-dimensional' module \(U\cdot v\coloneqq\{u\cdot v\,:\,u\in U\}\) for every \(v\in\mathscr M\).
\item The sum \(\mathscr N_1+\mathscr N_2\coloneqq\big\{v+w\,:\,v\in\mathscr N_1,\,w\in\mathscr N_2\big\}\)
where \(\mathscr N_1\), \(\mathscr N_2\) are \(V\)-Banach \(U\)-submodules of \(\mathscr M\).
We say that \(\mathscr M\) is the \textbf{direct sum} of \(\mathscr N_1\) and \(\mathscr N_2\), and we write
\[
\mathscr M=\mathscr N_1\oplus\mathscr N_2,
\]
if \(\mathscr M=\mathscr N_1+\mathscr N_2\) and \(\mathscr N_1\cap\mathscr N_2=\{0\}\).
The map \(\mathscr N_1\times\mathscr N_2\ni(v,w)\mapsto v+w\in\mathscr M\) is bijective.
\end{itemize}
If \((\mathcal U,U,V)\) is a metric \(f\)-structure such that \(V\) is Dedekind complete, \(\mathscr M\)
is a \(V\)-Banach \(U\)-module,  and \(\mathscr N\) is a \(V\)-Banach \(U\)-submodule of \(\mathscr M\),
then the \textbf{quotient module} \(\mathscr M/\mathscr N\) is a \(V\)-Banach \(U\)-module if endowed with
the following \(V\)-pointwise norm operator:
\[
|v+\mathscr N|\coloneqq\inf\big\{|v+w|\;\big|\;w\in\mathscr N\big\}\in V^+,\quad\text{ for every }v+\mathscr N\in\mathscr M/\mathscr N.
\]
\begin{definition}[Hilbert module]\label{def:Hilbert_mod}
Let \((\mathcal U,U,V,V,Z)\) be a dual system of metric \(f\)-structures. Then by a \textbf{\(V\)-Hilbert \(U\)-module}
we mean a \(V\)-Banach \(U\)-module \(\mathscr H\) such that
\begin{equation}\label{eq:ptwse_parall_rule}
|v+w|^2+|v-w|^2=2|v|^2+2|w|^2,\quad\text{ for every }v,w\in\mathscr H.
\end{equation}
We refer to \eqref{eq:ptwse_parall_rule} as the \textbf{pointwise parallelogram rule} of \(\mathscr H\).
\end{definition}

The \textbf{pointwise scalar product} on \(\mathscr H\) is defined as follows:
\[
\mathscr H\times\mathscr H\ni(v,w)\mapsto v\cdot w\coloneqq\frac{1}{2}\big(|v+w|^2-|v|^2-|w|^2\big)\in Z,
\quad\text{ for every }v,w\in\mathscr H.
\]
One can readily check that the pointwise scalar product is \textbf{\(U\)-bilinear}, which means that
\[\begin{split}
\mathscr H\ni v&\mapsto v\cdot z\in Z,\quad\text{ is }U\text{-linear,}\\
\mathscr H\ni w&\mapsto z\cdot w\in Z,\quad\text{ is }U\text{-linear,}
\end{split}\]
for any given \(z\in\mathscr H\). We will study Hilbert modules more in detail in Section \ref{ss:Hilbert_mod}. 
\subsubsection{Local invertibility}
Let \((\mathcal U,U,V)\) be a countably representable metric \(f\)-structure and
\(\mathscr M\) a \(V\)-normed \(U\)-module. Then each \(v\in\mathscr M\) is associated
with the element \(\nchi_{\{v=0\}}\in\Idem(U)\), which we define as
\[
\nchi_{\{v=0\}}\coloneqq\sup\big\{u\in\Idem(U)\;\big|\;u\cdot v=0\big\}\in\Idem(U).
\]
The idempotency of \(\nchi_{\{v=0\}}\) follows from Lemma \ref{lem:order_idem} and the fact that \(U\)
is CR. The glueing property of \(\mathscr M\) ensures that \(\nchi_{\{v=0\}}\cdot v=0\).
We also define \(\nchi_{\{v\neq 0\}}\coloneqq 1-\nchi_{\{v=0\}}\in\Idem(U)\), so that \(v=\nchi_{\{v\neq 0\}}\cdot v\).
Similarly, we define \(\nchi_{\{v=w\}}\coloneqq\nchi_{\{v-w=0\}}\), and so on.
\begin{proposition}[Local inverses]\label{prop:local_inverses}
Let \(U\) be a CR metric \(f\)-algebra. Let \(u\in U^+\) be given. Then there exists a partition
\((u_n)_{n\in\N}\) of \(\nchi_{\{u>0\}}\) and a sequence \((w_n)_{n\in\N}\subset U^+\) such that
\[
u_n(uw_n-1)=0,\quad\text{ for every }n\in\N.
\]
\end{proposition}
\begin{proof}
First, recall that \((U,U,U)\) is a metric \(f\)-structure (Example \ref{ex:U=V}) and that \(U\) is a \(U\)-Banach \(U\)-module
(Proposition \ref{prop:Riesz_is_Banach_module}). Since \(\mathcal S(U)\) is super-order-dense in \(U\) (as guaranteed by the
very definition of a localisable \(f\)-algebra), we can find a non-decreasing sequence \((s_n)_{n\in\N}\subset\mathcal S^+(U)\)
such that \(u=\sup_{n\in\N}s_n\). Note that, setting \(b_n\coloneqq\nchi_{\{s_n>0\}}\) for every \(n\in\N\),
we have \(\nchi_{\{u>0\}}=\sup_{n\in\N}b_n\). Indeed, on the one hand \(\nchi_{\{s_n>0\}}\leq\nchi_{\{u>0\}}\) for every \(n\in\N\)
and thus \(\sup_n b_n\leq\nchi_{\{u>0\}}\). On the other hand, if we denote \(s\coloneqq\nchi_{\{u>0\}}-\sup_n b_n\in\Idem(U)\), then
\(s=0\) in view of the fact that
\[
su=s\sup_{n\in\N}s_n=\sup_{n\in\N}ss_n=0.
\]
Moreover, for any \(n\in\N\) there exists a real number \(\lambda_n>0\) with \(\lambda_n b_n\leq u\). More precisely, if \(s_n\) is written
as \(\sum_{i=1}^{k_n}\lambda_n^i u_n^i\) for some \(k_n\in\N\), \((\lambda_n^i)_{i=1}^{k_n}\subset\R\cap(0,+\infty)\), and
\((u_n^i)_{i=0}^{k_n}\in\mathcal P_f({\bf 1}_U)\), then we have that \(b_n=u_n^1\vee\dots\vee u_n^{k_n}\) and that \(\lambda_n^1\wedge\dots\wedge\lambda_n^{k_n}\)
can be chosen as \(\lambda_n\). Now we define \(u_1\coloneqq b_1\) and \(u_{n+1}\coloneqq b_{n+1}(1-b_n)\) for every \(n\in\N\).
Observe that \((u_n)_{n\in\N}\) is a partition of \(\nchi_{\{u>0\}}\). Next we consider the simple elements
\(t_j\coloneqq\sum_{i=1}^{k_j}\frac{1}{\lambda_j^i}u_j^i\) for every \(j\in\N\). Given any \(n\in\N\), we have that the sequence
\((u_n t_j)_{j=n}^\infty\) is non-increasing and satisfies \(0\leq u_n t_j\leq\frac{1}{\lambda_n}u_n\) for every \(j\geq n\).
Then the infimum \(w_n\coloneqq\inf_{j\geq n}u_n t_j\in U^+\) exists and, since \(s_j t_j=b_j\geq u_n\) for every \(j\geq n\), it holds that
\[\begin{split}
u_n u w_n&=\lambda_n^{-1}u_n u-u_n u(\lambda_n^{-1}u_n-w_n)
=\lambda_n^{-1}u_n u-u_n\Big(\sup_{j\geq n}s_j\Big)\Big(\lambda_n^{-1}u_n-\inf_{j\geq n}u_n t_j\Big)\\
&=\lambda_n^{-1}u_n u-\Big(\sup_{j\geq n}s_j\Big)\sup_{j\geq n}(\lambda_n^{-1}u_n-u_n t_j)
=\lambda_n^{-1}u_n u-\sup_{j\geq n}(\lambda_n^{-1}u_n s_j-u_n s_j t_j)\\
&=\lambda_n^{-1}u_n u-\sup_{j\geq n}(\lambda_n^{-1}u_n s_j-u_n)=\lambda_n^{-1}u_n u-u_n(\lambda_n^{-1}u-1)=u_n.
\end{split}\]
Consequently, the statement is finally achieved.
\end{proof}
\subsubsection{Support of a metric \texorpdfstring{\(f\)}{f}-structure}
Given a metric \(f\)-structure \((\mathcal U,U,V)\) that is countably representable, we can define its \emph{support}, which is the `largest idempotent element
where at least an element of \(V\) does not vanish'. Namely:
\begin{definition}[Support]\label{def:spt}
Let \((\mathcal U,U,V)\) be a CR metric \(f\)-structure. Then we define \({\rm S}(V)\) as
\[
{\rm S}(V)\coloneqq\sup\big\{\nchi_{\{v\neq 0\}}\;\big|\;v\in V^+\big\}\in\Idem(U).
\]
We say that \({\rm S}(V)\) is the \textbf{support} of \(V\), or of the metric \(f\)-structure \((\mathcal U,U,V)\).
\end{definition}

Let us check that the previous definition is well-posed. Since \(\nchi_{\{v\neq 0\}}\leq 1\) for every \(v\in V^+\), the Dedekind
completeness of \(V\) ensures that \({\rm S}(V)\) exists. Moreover, the countable representability assumption ensures
the existence of a sequence \((v_n)_{n\in\N}\subset V^+\) such that \({\rm S}(V)=\sup_n\nchi_{\{v_n\neq 0\}}\).
Taking into account Lemma \ref{lem:order_idem}, it also follows that \({\rm S}(V)\in\Idem(U)\).
\begin{remark}\label{rmk:relation_H(V)}{\rm
If \((\mathcal U,U,V,W,Z)\) is a CR dual system of metric \(f\)-structures, then it holds that
\[
{\rm S}(V)\wedge{\rm S}(W)\leq{\rm S}(Z).
\]
In order to prove it, fix two sequences \((v_i)_{i\in\N}\subset V^+\) and \((w_j)_{j\in\N}\subset W^+\)
with \({\rm S}(V)=\sup_i\nchi_{\{v_i\neq 0\}}\) and \({\rm S}(W)=\sup_j\nchi_{\{w_j\neq 0\}}\).
Notice that \(\nchi_{\{v_i\neq 0\}}\wedge\nchi_{\{w_j\neq 0\}}\leq\nchi_{\{v_i w_j\neq 0\}}\leq{\rm S}(Z)\)
for every \(i,j\in\N\). Taking the supremum over \(i,j\in\N\), we conclude that \({\rm S}(V)\wedge{\rm S}(W)\leq{\rm S}(Z)\),
as we claimed.
\fr}\end{remark}

Next we prove two technical results concerning the support of a metric \(f\)-structure.
\begin{lemma}\label{lem:supporting_element}
Let \((\mathcal U,U,V)\) be a CR metric \(f\)-structure. Then there exists an element \(h\in V^+\cap U^+\)
such that \(h\leq 1\) and \(\nchi_{\{h>0\}}={\rm S}(V)\).
\end{lemma}
\begin{proof}
Pick a sequence \((v_n)_{n\in\N}\subset V^+\) such that \(v_n\leq 1\) for every \(n\in\N\) and \({\rm S}(V)=\sup_n\nchi_{\{v_n\neq 0\}}\).
Define \(u_1\coloneqq\nchi_{\{v_1\neq 0\}}\in\Idem(U)\) and, recursively, \(u_{n+1}\coloneqq\nchi_{\{v_{n+1}\neq 0\}}(1-u_1)\dots(1-u_n)\in\Idem(U)\)
for every \(n\in\N\). Notice that \((u_n)_{n\in\N}\) is a partition of \({\rm S}(V)\). Recalling item v) of Definition \ref{def:metric_f-structure}, we can
find \(\delta>0\) such that \((u_n w_n)_{n\in\N}\) is order-bounded in \(V\) whenever \((w_n)_{n\in\N}\subset V^+\) is chosen so that
\(\sum_{n\in\N}\sfd_V(u_n w_n,0)\leq\delta\). Now pick \((\lambda_n)_{n\in\N}\subset(0,1)\) with \(\sfd_V(\lambda_n u_n v_n,0)\leq\frac{\delta}{2^n}\)
for all \(n\in\N\). Therefore, the element \(h\coloneqq\sum_{n\in\N}\lambda_n u_n v_n\in V^+\) exists. Notice that \(h\leq 1\) and
\(\nchi_{\{h\neq 0\}}={\rm S}(V)\). In particular, \(\nchi_{\{h=0\}}|v|=(1-{\rm S}(V))\nchi_{\{v=0\}}|v|=0\) for every \(v\in V\),
whence the statement follows.
\end{proof}
\begin{lemma}\label{lem:tech_localisation}
Let \((\mathcal U,U,V)\) be a CR metric \(f\)-structure. Let \(W\subset\mathcal U\) be a CR metric Riesz space
such that \((\mathcal U,U,W)\) is a metric \(f\)-structure and \({\rm S}(V)\leq{\rm S}(W)\). Let \(v\in V^+\)
be given. Then there exists a partition \((u_n)_{n\in\N}\subset W\cap\Idem(U)\) of \(\nchi_{\{v\neq 0\}}\)
such that \(u_n v\in W\cap U\) holds for every \(n\in\N\).
\end{lemma}
\begin{proof}
Thanks to Lemma \ref{lem:supporting_element}, we can find an element \(h\in W^+\cap U^+\) such that \(\nchi_{\{h\neq 0\}}={\rm S}(W)\).
We then define \(\tilde s_i\coloneqq\nchi_{\{0<v\leq ih\}}\) and \(\tilde t_j\coloneqq\nchi_{\{v\neq 0\}}\nchi_{\{h\geq j^{-1}{\bf 1}_U\}}\)
for every \(i,j\in\N\). We claim that
\begin{equation}\label{eq:tech_localisation}
\sup_{i\in\N}\tilde s_i=\nchi_{\{v\neq 0\}}=\sup_{j\in\N}\tilde t_j.
\end{equation}
We prove the first equality, since the proof of the second one is similar. Clearly, \(\sup_i\tilde s_i\leq\nchi_{\{v\neq 0\}}\).
For the converse inequality, we argue by contradiction: suppose that \(s\coloneqq\nchi_{\{v\neq 0\}}-\sup_i\tilde s_i\neq 0\).
Then \(ish\leq v\) for every \(i\in\N\), whence it follows (since \(\mathcal U\) is Archimedean) that \(sh=0\), which leads to
a contradiction. Therefore, the claim \eqref{eq:tech_localisation} is proved. Now let us define \(s_1\coloneqq\tilde s_1\), \(t_1\coloneqq\tilde t_1\)
and, recursively, \(s_{i+1}\coloneqq\tilde s_{i+1}-s_1\ldots s_i\tilde s_{i+1}\) and \(t_{j+1}\coloneqq\tilde t_{j+1}-t_1\ldots t_j\tilde t_{j+1}\)
for every \(i,j\in\N\). Notice that \eqref{eq:tech_localisation} implies that \((s_i)_{i\in\N}\) and \((t_j)_{j\in\N}\) are partitions
of \(\nchi_{\{v\neq 0\}}\). Moreover, \(s_i v\leq ih\in W^+\cap U^+\) and \(t_j\leq jh\in W^+\) for every \(i,j\in\N\), thus accordingly
\(s_i v\in W^+\cap U^+\) and \(t_j\in W^+\). Relabeling the family \(\{s_i t_j\,:\,i,j\in\N\}\) as \(\{u_n\}_{n\in\N}\), we finally obtain
a partition \((u_n)_{n\in\N}\subset W\cap\Idem(U)\) of the element \(\nchi_{\{v\neq 0\}}\) satisfying \(u_n v\in W\cap U\) for every \(n\in\N\), as desired.
\end{proof}
\subsection{Homomorphisms of normed modules}\label{ss:hom_norm_mod}
To begin with, we introduce the notion of a homomorphism between normed modules.
\begin{definition}[Homomorphism of normed modules]\label{def:hom}
Let \((\mathcal U,U,V,W,Z)\) be a dual system of metric \(f\)-structures, \(\mathscr M\)
a \(V\)-normed \(U\)-module, and \(\mathscr N\) a \(Z\)-normed \(U\)-module. Then we define
\[
\textsc{Hom}(\mathscr M,\mathscr N)\coloneqq\Big\{T\colon\mathscr M\to\mathscr N\;\,U\text{-linear}\;\Big|
\;\exists\,w\in W^+:\,|Tv|\leq w|v|,\text{ for every }v\in\mathscr M\Big\}.
\]
\end{definition}

Next, we endow \(\textsc{Hom}(\mathscr M,\mathscr N)\) with a \(U\)-module structure. Given any \(T,S\in\textsc{Hom}(\mathscr M,\mathscr N)\)
and \(u\in U\), we define the elements \(T+S\in\textsc{Hom}(\mathscr M,\mathscr N)\) and \(u\cdot T\in\textsc{Hom}(\mathscr M,\mathscr N)\) as
\[\begin{split}
(T+S)v\coloneqq Tv+Sv,&\quad\text{ for every }v\in\mathscr M,\\
(u\cdot T)v\coloneqq u\cdot Tv,&\quad\text{ for every }v\in\mathscr M,
\end{split}\]
respectively. One can readily check that the triple \(\big(\textsc{Hom}(\mathscr M,\mathscr N),+,\cdot\big)\) is a module over \(U\).
\medskip

In the case where \(W\) is Dedekind complete, for any given \(T\in\textsc{Hom}(\mathscr M,\mathscr N)\) it holds that
\[
\exists|T|\coloneqq\inf\Big\{w\in W^+\;\Big|\;|Tv|\leq w|v|,\text{ for every }v\in\mathscr M\Big\}\in W^+.
\]
The space of homomorphisms between two normed modules inherits a normed module structure:
\begin{theorem}\label{thm:Hom_normed_module}
Let \((\mathcal U,U,V,W,Z)\) be a complete dual system of metric \(f\)-structures. Let \(\mathscr M\) be a \(V\)-normed
\(U\)-module and \(\mathscr N\) a \(Z\)-normed \(U\)-module. Let \(T\in\textsc{Hom}(\mathscr M,\mathscr N)\) be given. Then
\[
|Tv|\leq|T||v|,\quad\text{ for every }v\in\mathscr M.
\]
Moreover, the space \(\big(\textsc{Hom}(\mathscr M,\mathscr N),|\cdot|\big)\) is a \(W\)-normed \(U\)-module and
\begin{equation}\label{eq:norm_glueing}
\bigg|\sum_{n\in\N}u_n\cdot T_n\bigg|=\sum_{n\in\N}u_n|T_n|,\quad\text{ for every }
(u_n,T_n)_{n\in\N}\in\Adm\big(\textsc{Hom}(\mathscr M,\mathscr N)\big).
\end{equation}
If in addition \(\mathscr N\) is a \(Z\)-Banach \(U\)-module, then \(\textsc{Hom}(\mathscr M,\mathscr N)\) is a \(W\)-Banach \(U\)-module.
\end{theorem}
\begin{proof}\ \\
{\color{blue}\textsc{Verification of the \(W\)-pointwise norm axioms.}} Given any \(T\in\textsc{Hom}(\mathscr M,\mathscr N)\), we define
\begin{equation}\label{eq:def_F_T}
\mathcal F_T\coloneqq\Big\{w\in W^+\;\Big|\;|Tv|\leq w|v|,\text{ for every }v\in\mathscr M\Big\}\neq\varnothing.
\end{equation}
Since \(\mathcal F_T\) is a sublattice of \(W^+\) -- thus in particular it is directed downwards -- we deduce that
\[
|Tv|\leq\inf_{w\in\mathcal F_T}w|v|=|v|\inf_{w\in\mathcal F_T}w=|T||v|,\quad\text{ for every }v\in\mathscr M,
\]
as a consequence of the order-continuity of the multiplication from \(V^+\times W^+\) to \(Z^+\). It readily follows that
\(|\cdot|\colon\textsc{Hom}(\mathscr M,\mathscr N)\to W^+\) satisfies \eqref{eq:ptwse_norm_1} and \eqref{eq:ptwse_norm_2}.
It is also easy to check that the identity \(|\lambda T|=|\lambda||T|\) holds for every \(\lambda\in\R\) and \(T\in\textsc{Hom}(\mathscr M,\mathscr N)\).
We now pass to the verification of \eqref{eq:ptwse_norm_3}. For any \(u\in U\) and \(T\in\textsc{Hom}(\mathscr M,\mathscr N)\),
one has \(|(u\cdot T)v|=|u||Tv|\leq|u||T||v|\) for every \(v\in\mathscr M\), whence it follows that \(|u\cdot T|\leq|u||T|\).
On the other hand, we claim that also
\begin{equation}\label{eq:check_ptwse_norm_Hom}
|u||T|\leq|u\cdot T|,\quad\text{ for every }u\in U\text{ and }T\in\textsc{Hom}(\mathscr M,\mathscr N).
\end{equation}
In the case where \(u\in\Idem(U)\), the inequality stated in \eqref{eq:check_ptwse_norm_Hom} follows from the observation that
\[
u|T|=u|(1-u)\cdot T+u\cdot T|\leq u|(1-u)\cdot T|+u|u\cdot T|\leq u(1-u)|T|+|u\cdot T|=|u\cdot T|.
\]
Moreover, if \(u=\sum_{i=1}^k\lambda_i u_i\in\mathcal S^+(U)\) is given, then for any \(j=1,\ldots,k\) it holds that
\[
u_j\sum_{i=1}^k\lambda_i|u_i\cdot T|\leq\sum_{i=1}^k\lambda_i u_i u_j|T|=\lambda_j u_j^2|T|=|u_j u\cdot T|\leq u_j|u\cdot T|,
\]
which implies that \(u|T|=\sum_{i=1}^k\lambda_i u_i|T|=\sum_{i=1}^k\lambda_i|u_i\cdot T|\leq|u\cdot T|\), proving \eqref{eq:check_ptwse_norm_Hom}
for all \(u\in\mathcal S^+(U)\). Given any \(u\in U^+\), we can pick \((u_n)_{n\in\N}\subset\mathcal S^+(U)\) such that
\(\lim_{n\to\infty}\sfd_U(u_n,u)=0\) and thus
\[
|u||T|=\lim_{n\to\infty}|u_n||T|\leq\lim_{n\to\infty}|u_n\cdot T|=|u\cdot T|,
\]
which proves \eqref{eq:check_ptwse_norm_Hom} for all \(u\in U^+\). Finally, given an arbitrary element \(u\in U\) we have that
\[
|u^+\cdot T|\wedge|u^-\cdot T|\leq(u^+|T|)\wedge(u^-|T|)=(u^+\wedge u^-)|T|\overset{\eqref{eq:basic_prop_f-alg_7}}=0,
\]
so that \(|u\cdot T|=|u^+\cdot T|+|u^-\cdot T|\) by \eqref{eq:prod_preserv_2} and thus \(|u||T|=u^+|T|+u^-|T|\leq|u^+\cdot T|+|u^-\cdot T|=|u\cdot T|\).
This proves \eqref{eq:check_ptwse_norm_Hom} for general \(u\in U\). Therefore, the proof of the validity of \eqref{eq:ptwse_norm_3} is complete.\\
{\color{blue}\textsc{Verification of the glueing property.}} To prove that \(\big(\textsc{Hom}(\mathscr M,\mathscr N),|\cdot|\big)\)
is a \(W\)-normed \(U\)-module, it only remains to show the validity of the glueing property. Fix any \((u_n)_{n\in\N}\in\mathcal P({\bf 1}_U)\)
and \((T_n)_{n\in\N}\subset\textsc{Hom}(\mathscr M,\mathscr N)\) with \((|u_n\cdot T_n|)_{n\in\N}\) order-bounded in \(W\).
Since \(|u_n\cdot T_n|=|u_n||T_n|\), this means that \((u_n,|T_n|)_{n\in\N}\in{\rm Adm}(W)\) and thus \(w\coloneqq\sum_{n\in\N}u_n|T_n|\in W^+\)
exists; cf.\ with Proposition \ref{prop:Riesz_is_Banach_module}. Given any \(v\in\mathscr M\), we have that \(|u_n\cdot T_n v|\leq|u_n\cdot T_n||v|\)
for every \(n\in\N\), which ensures that \((u_n,T_n v)_{n\in\N}\in\Adm(\mathscr N)\), so it makes sense to consider
\(Tv\coloneqq\sum_{n\in\N}u_n\cdot T_n v\in\mathscr N\). The \(U\)-linearity of the resulting map \(T\colon\mathscr M\to\mathscr N\) can be easily checked.
Given that for any \(n\in\N\) and \(v\in\mathscr M\) one has that \(u_n|Tv|=|u_n\cdot Tv|=|u_n\cdot T_n v|\leq u_n|T_n||v|=u_n w|v|\), we deduce that
\[
|Tv|=\sup_{n\in\N}u_n|Tv|\leq\sup_{n\in\N}u_n w|v|=w|v|,\quad\text{ for every }v\in\mathscr M.
\]
This yields \(T\in\textsc{Hom}(\mathscr M,\mathscr N)\) and \(|T|\leq w\). Note that \((u_n\cdot T)v=u_n\cdot Tv=u_n\cdot T_n v=(u_n\cdot T_n)v\)
for every \(v\in\mathscr M\), so that \(T=\sum_{n\in\N}u_n\cdot T_n\). Finally, for any \(n\in\N\) we have \(u_n\cdot T_n=u_n\cdot T\) and thus
\(u_n|T_n|=u_n|T|\), which gives \(w=\sum_{n\in\N}u_n|T_n|=\sup_{n\in\N}u_n|T|=|T|\). This proves \eqref{eq:norm_glueing}.\\
{\color{blue}\textsc{Completeness.}} Suppose \((\mathscr N,\sfd_{\mathscr N})\) is complete. Let
\((T_n)_{n\in\N}\subset\textsc{Hom}(\mathscr M,\mathscr N)\) be a Cauchy sequence. Up to taking a not relabeled subsequence,
we may assume that \(\sfd_W(|T_{n+1}-T_n|,0)\leq 2^{-n}\) for every \(n\in\N\). Define \(S_n\coloneqq|T_1|+\sum_{k<n}|T_{k+1}-T_k|\in W^+\)
for every \(n\in\N\). If \(n,m\in\N\) are such that \(n<m\), then we have that \(S_m-S_n=\sum_{k=n}^{m-1}|T_{k+1}-T_k|\) and thus accordingly
\[
\sfd_W(|S_m-S_n|,0)\leq\sum_{k=n}^{m-1}\sfd_W(|T_{k+1}-T_k|,0)\leq\sum_{k=n}^{m-1}\frac{1}{2^k}\leq\frac{1}{2^{n-1}}.
\]
This shows that \((S_n)_{n\in\N}\) is a Cauchy sequence in \(W^+\), so that it \(\sfd_W\)-converges to some \(S\in W^+\).
Notice that \((S_n)_{n\in\N}\) is a non-decreasing sequence by construction, thus in particular it holds
\begin{equation}\label{eq:unif_bound_Hom}
|T_n|=\big|(T_n-T_{n-1})+(T_{n-1}-T_{n-2})+\dots+(T_2-T_1)+T_1\big|\leq S_n\leq S,\quad\forall n\in\N.
\end{equation}
Since the multiplication is continuous from \(V\times W\) to \(Z\) and \(\sfd_W(|T_m-T_n|,0)\to 0\) as \(n,m\to\infty\),
for any fixed element \(v\in\mathscr M\) we have that \(\sfd_{\mathscr N}(T_m v,T_n v)\leq\sfd_Z(|T_m-T_n||v|,0)\to 0\)
as \(n,m\to\infty\), which shows that \((T_n v)_{n\in\N}\) is a Cauchy sequence in \(\mathscr N\). We then define
\(Tv\coloneqq\lim_{n\to\infty}T_n v\in\mathscr N\). The resulting mapping \(T\colon\mathscr M\to\mathscr N\) is \(U\)-linear,
as it is a pointwise limit of \(U\)-linear maps. Also,
\[
|Tv|=\lim_{n\to\infty}|T_n v|\overset{\eqref{eq:unif_bound_Hom}}\leq S|v|,\quad\text{ for every }v\in\mathscr M,
\]
whence it follows that \(T\in\textsc{Hom}(\mathscr M,\mathscr N)\) and \(|T|\leq S\). Now set
\(P^n_m\coloneqq\sum_{k=n}^{m-1}|T_{k+1}-T_k|\in W^+\) for every \(n,m\in\N\) with \(n<m\). Arguing as we did before,
we see that \((P^n_m)_{m\in\N}\) is \(\sfd_W\)-Cauchy, thus \(\lim_{m\to\infty}P^n_m=P^n\) for some \(P^n\in W^+\).
Note that \(P^n_m\leq P^n\) and \(\sfd_W(P^n,0)\leq 2^{-n+1}\). Hence,
\[
|(T-T_n)v|=\lim_{m\to\infty}|(T_m-T_n)v|\leq\lim_{m\to\infty}P^n_m|v|=P^n|v|,\quad\text{ for every }v\in\mathscr M,
\]
which implies \(|T-T_n|\leq P^n\to 0\) as \(n\to\infty\). The completeness of \(\textsc{Hom}(\mathscr M,\mathscr N)\) follows.
\end{proof}

In the case where the complete dual system under consideration is countably representable, the \(W\)-pointwise norm \(|T|\)
of any given \(T\in\textsc{Hom}(\mathscr M,\mathscr N)\) can be also characterised as follows:
\begin{lemma}\label{lem:equiv_Hom_pn}
Let \((\mathcal U,U,V,W,Z)\) be a CR dual system of metric \(f\)-structures, \(\mathscr M\) a \(V\)-normed \(U\)-module,
and \(\mathscr N\) a \(Z\)-normed \(U\)-module. Then it holds that
\begin{equation}\label{eq:equiv_Hom_pn}
|T|=\sup\big\{|Tv|\;\big|\;v\in\mathscr M,\,|v|\leq 1\big\},\quad\text{ for every }T\in\textsc{Hom}(\mathscr M,\mathscr N).
\end{equation}
\end{lemma}
\begin{proof}
On the one hand, \(|Tv|\leq|T||v|\leq|T|\) for all \(v\in\mathscr M\) with \(|v|\leq 1\), thus the right-hand side in
\eqref{eq:equiv_Hom_pn} exists and defines an element \(b\in W^+\) with \(b\leq|T|\). On the other hand, we claim that
\begin{equation}\label{eq:equiv_Hom_pn_aux}
|Tv|\leq b|v|,\quad\text{ for every }v\in\mathscr M.
\end{equation}
In order to prove it, fix any \(v\in\mathscr M\). By using Lemma \ref{lem:tech_localisation} and Proposition \ref{prop:local_inverses},
we can find a partition \((u_n)_{n\in\N}\) of \(\nchi_{\{v\neq 0\}}\) and a sequence \((w_n)_{n\in\N}\subset U^+\) such that
\(u_n|v|\in U\) and \(u_n w_n|v|=u_n\) for every \(n\in\N\). Letting \(v_n\coloneqq(u_n w_n)\cdot v\in\mathscr M\), we have
\(|v_n|=u_n\leq 1\) and thus \(|Tv_n|\leq b\). Then
\[
|Tv|=\sum_{n\in\N}(u_n w_n|v|)|Tv|=\sum_{n\in\N}|v||Tv_n|\leq\sum_{n\in\N}u_n b|v|=b|v|.
\]
This proves \eqref{eq:equiv_Hom_pn_aux} and accordingly that \(b=|T|\), thus concluding the proof of the statement.
\end{proof}
Given a dual system of metric \(f\)-structures \((\mathcal U,U,V,W,Z)\), a \(V\)-Banach \(U\)-module \(\mathscr M\), and a
\(Z\)-Banach \(U\)-module \(\mathscr N\), we define the \textbf{kernel} of a homomorphism \(T\in\textsc{Hom}(\mathscr M,\mathscr N)\) as
\[
{\rm ker}(T)\coloneqq T^{-1}(\{0\})=\{v\in\mathscr M\;|\;Tv=0\big\}.
\]
It can be readily checked that \({\rm ker}(T)\) is a \(V\)-Banach \(U\)-submodule of \(\mathscr M\).
\medskip

It is natural to introduce the categories of normed modules and of Banach modules:
\begin{definition}[Category of normed modules]\label{def:category_nm}
Let \((\mathcal U,U,V)\) be a metric \(f\)-structure. Then we define the category \({\bf NormMod}_{(\mathcal U,U,V)}\) of normed modules over \((\mathcal U,U,V)\) as follows:
\begin{itemize}
\item[\(\rm i)\)] The objects of \({\bf NormMod}_{(\mathcal U,U,V)}\) are given by the \(V\)-normed \(U\)-modules.
\item[\(\rm ii)\)] For any two objects \(\mathscr M\) and \(\mathscr N\) of \({\bf NormMod}_{(\mathcal U,U,V)}\), the morphisms between \(\mathscr M\) and
\(\mathscr N\) are given by those homomorphisms \(T\in\textsc{Hom}(\mathscr M,\mathscr N)\) satisfying \(|Tv|\leq|v|\) for every \(v\in\mathscr M\).
\end{itemize}
Moreover, we denote by \({\bf BanMod}_{(\mathcal U,U,V)}\) the full subcategory of \({\bf NormMod}_{(\mathcal U,U,V)}\) whose objects are the \(V\)-Banach \(U\)-modules.
\end{definition}

It would be interesting -- but outside the scopes of this work -- to investigate which results of \cite{Pas22} are valid for Banach modules over a more general class of metric \(f\)-structures.
\medskip

Let us also define the dual of a Banach module:
\begin{definition}[Dual Banach module]\label{def:dual_mod}
Let \((\mathcal U,U,V,W,Z)\) be a complete dual system of metric \(f\)-structures and \(\mathscr M\)
a \(V\)-normed \(U\)-module. Then the \textbf{dual} of \(\mathscr M\) is the \(W\)-Banach \(U\)-module
\[
\mathscr M^*\coloneqq\textsc{Hom}(\mathscr M,Z).
\]
\end{definition}

The duality pairing between \(\mathscr M\) and \(\mathscr M^*\) is given by
\[
\mathscr M^*\times\mathscr M\ni(\omega,v)\mapsto\langle\omega,v\rangle\coloneqq\omega(v)\in Z.
\]
\subsubsection{On the existence of homomorphisms}
As we will see later (in Section \ref{ss:example_norm_mod}), when working with normed modules over a function space,
one often has to prove the existence of homomorphisms of normed modules verifying suitable properties.
All these existence statements can be deduced from the following general result.
\begin{theorem}\label{thm:main_ext_hom}
Let \((\mathcal U_1,U_1,V_1)\) be a metric \(f\)-structure and \((\mathcal U_2,U_2,V_2,W_2,Z_2)\) a dual system of metric
\(f\)-structures. Let \(\phi\colon(\mathcal U_1,U_1,V_1)\to(\mathcal U_2,U_2,V_2)\) be a homomorphism of metric
\(f\)-structures. Let \(\mathscr M\) be a \(V_1\)-normed \(U_1\)-module
and \(\mathscr N\) a \(Z_2\)-Banach \(U_2\)-module. Let \(\mathscr V\) be a vector subspace of \(\mathscr M\) such that
\(\mathscr G(\mathscr V)\) is dense in \(\mathscr M\). Fix \(T\colon\mathscr V\to\mathscr N\) linear such that for some \(b\in W_2^+\) it holds
\begin{equation}\label{eq:main_ext_hom_hp}
|Tv|\leq b\,\phi(|v|),\quad\text{ for every }v\in\mathscr V.
\end{equation}
Then there exists a unique linear operator \(\bar T\colon\mathscr M\to\mathscr N\) such that \(\bar T|_{\mathscr V}=T\) and
\begin{equation}\label{eq:main_ext_hom_th1}
|\bar T v|\leq b\,\phi(|v|),\quad\text{ for every }v\in\mathscr M.
\end{equation}
In particular, the map \(\bar T\colon\mathscr M\to\mathscr N\) is continuous. Moreover, it holds that
\begin{equation}\label{eq:main_ext_hom_th2}
\bar T(u\cdot v)=\phi(u)\cdot\bar T(v),\quad\text{ for every }u\in U_1\text{ and }v\in\mathscr M.
\end{equation}
\end{theorem}
\begin{proof}
First of all, we define the operator \(S\colon\mathscr G(\mathscr V)\to\mathscr N\) as
\[
S\bigg(\sum_{n\in\N}u_n\cdot v_n\bigg)\coloneqq\sum_{n\in\N}\phi(u_n)\cdot Tv_n,
\quad\text{ for every }\sum_{n\in\N}u_n\cdot v_n\in\mathscr G(\mathscr V).
\]
Let us check that \(S\) is well-defined. Letting \(z\in V_1^+\) be an upper bound for \((u_n|v_n|)_{n\in\N}\), we infer from
\eqref{eq:main_ext_hom_hp} that \(|\phi(u_n)\cdot Tv_n|=\phi(u_n)|Tv_n|\leq b\,\phi(u_n|v_n|)\leq b\,\phi(z)\in Z_2^+\)
for every \(n\in\N\). Given that \(\big(\phi(u_n)\big)_{n\in\N}\in\mathcal P(U_2)\) by \eqref{eq:img_partition}, we deduce
that \(\big(\phi(u_n),Tv_n\big)_{n\in\N}\in{\rm Adm}(\mathscr N)\) and thus accordingly \(\sum_{n\in\N}\phi(u_n)\cdot Tv_n\in\mathscr N\)
is well-defined. Also, if \(\sum_{n\in\N}u_n\cdot v_n=\sum_{m\in\N}\tilde u_m\cdot\tilde v_m\), then
\[\begin{split}
\big|\phi(u_n\tilde u_m)\cdot Tv_n-\phi(u_n\tilde u_m)\cdot T\tilde v_m\big|&=\phi(u_n\tilde u_m)\big|T(v_n-\tilde v_m)\big|
\overset{\eqref{eq:main_ext_hom_hp}}\leq b\,\phi\big(u_n\tilde u_m|v_n-\tilde v_m|\big)\\
&=b\,\phi\big(\big|(u_n\tilde u_m)\cdot v_n-(u_n\tilde u_m)\cdot\tilde v_m\big|\big)=0
\end{split}\]
holds for every \(n,m\in\N\), which implies that \(\sum_{n\in\N}\phi(u_n)\cdot v_n=\sum_{m\in\N}\phi(\tilde u_m)\cdot\tilde v_m\).
All in all, the definition of \(S\) is well-posed. Observe that \(S\) is linear by construction. Moreover, one has that
\begin{equation}\label{eq:main_ext_hom_aux1}
|Sw|=\sum_{n\in\N}\phi(u_n)|Tv_n|\overset{\eqref{eq:main_ext_hom_hp}}\leq b\sum_{n\in\N}\phi(u_n)\phi(|v_n|)=b\,\phi(|w|),
\quad\forall w=\sum_{n\in\N}u_n\cdot v_n\in\mathscr G(\mathscr V).
\end{equation}
Notice that \(S\) is the unique linear map from \(\mathscr G(\mathscr V)\) to \(\mathscr N\) satisfying both \(S|_{\mathscr V}=T\) and \eqref{eq:main_ext_hom_aux1}.
It also follows from \eqref{eq:main_ext_hom_aux1} that the map \(S\) is Cauchy-continuous: if a given sequence \((w_i)_{i\in\N}\subset\mathscr G(\mathscr V)\) is
\(\sfd_{\mathscr M}\)-Cauchy, then \(\sfd_{V_2}\big(\phi(|w_i-w_j|),0\big)\to 0\) as \(i,j\to\infty\) and thus
\[
\lims_{i,j\to\infty}\sfd_{Z_2}\big(|Sw_i-Sw_j|,0\big)\overset{\eqref{eq:main_ext_hom_aux1}}\leq\lims_{i,j\to\infty}\sfd_{Z_2}\big(b\,\phi(|w_i-w_j|),0\big)=0,
\]
which shows that \((Sw_i)_{i\in\N}\) is \(\sfd_{\mathscr N}\)-Cauchy. Therefore, \(S\) can be uniquely extended to a linear, continuous map
\(\bar T\colon\mathscr M\to\mathscr N\). Thanks to an approximation argument, we can deduce from \eqref{eq:main_ext_hom_aux1} that \(\bar T\)
verifies \eqref{eq:main_ext_hom_th1}, whence the continuity of \(\bar T\) immediately follows. Finally, we can estimate
\[\begin{split}
\big|\bar T(u\cdot v)-\phi(u)\cdot\bar T(v)\big|&\leq\big|\bar T(u\cdot v)-\phi(u)\cdot\bar T(u\cdot v)\big|+\big|\phi(u)\cdot\bar T(u\cdot v)-\phi(u)\cdot\bar T(v)\big|\\
&=\big(1-\phi(u)\big)|\bar T(u\cdot v)|+\phi(u)|\bar T(u\cdot v-v)|\\
&\leq b\,\phi(1-u)\phi(|u\cdot v|)+b\,\phi(u)\phi(|u\cdot v-v|)\\
&=2b\,\phi\big(u(1-u)|v|\big)=0,
\end{split}\]
for every \(v\in\mathscr M\) and \(u\in\Idem(U_1)\), proving \eqref{eq:main_ext_hom_th2} when \(u\) is idempotent. By linearity, we deduce that \eqref{eq:main_ext_hom_th2}
holds whenever \(u\in\mathcal S(U_1)\). Thanks to the density of \(\mathcal S(U_1)\) in \(U_1\), as well as to the continuity of \(\bar T\), of \(\phi\), and of the scalar
multiplications, we conclude that \eqref{eq:main_ext_hom_th2} is verified.
\end{proof}

Let us isolate a useful byproduct of the last part of the proof of Theorem \ref{thm:main_ext_hom}:
\begin{lemma}\label{lem:criterion_Hom_phi}
Let \((\mathcal U_1,U_1,V_1)\) be a metric \(f\)-structure and \((\mathcal U_2,U_2,V_2,W_2,Z_2)\) a dual system of metric
\(f\)-structures. Let \(\phi\colon(\mathcal U_1,U_1,V_1)\to(\mathcal U_2,U_2,V_2)\) be a homomorphism of metric
\(f\)-structures. Let \(\mathscr M\) be a \(V_1\)-normed \(U_1\)-module and \(\mathscr N\) a \(Z_2\)-normed \(U_2\)-module.
Let \(T\colon\mathscr M\to\mathscr N\) be a linear operator having the following property: there exists an element \(b\in W_2^+\) such that
\begin{equation}\label{eq:criterion_Hom_hp}
|Tv|\leq b\,\phi(|v|),\quad\text{ for every }v\in\mathscr M.
\end{equation}
Then \(T\) is a continuous operator satisfying \(T(u\cdot v)=\phi(u)\cdot T(v)\) for every \(u\in U_1\) and \(v\in\mathscr M\).
\end{lemma}

As an immediate consequence of Lemma \ref{lem:criterion_Hom_phi}, we obtain a criterion to detect homomorphisms:
\begin{corollary}\label{cor:criterion_Hom}
Let \((\mathcal U,U,V,W,Z)\) be a dual system of metric \(f\)-structures, \(\mathscr M\) a \(V\)-normed \(U\)-module, and
\(\mathscr N\) a \(Z\)-normed \(U\)-module. Let \(T\colon\mathscr M\to\mathscr N\) be a linear operator such that
\[
|Tv|\leq w|v|,\quad\text{ for every }v\in\mathscr M,
\]
for some \(w\in W^+\). Then \(T\) is \(U\)-linear and continuous, thus in particular \(T\in\textsc{Hom}(\mathscr M,\mathscr N)\).
\end{corollary}
\begin{proof}
Apply Lemma \ref{lem:criterion_Hom_phi} with \((\mathcal U_1,U_1,V_1)\coloneqq(\mathcal U,U,V)\), \((\mathcal U_2,U_2,V_2,W_2,Z_2)\coloneqq(\mathcal U,U,V,W,Z)\),
and \(\phi\coloneqq{\rm id}_{\mathcal U}\).
\end{proof}
\subsection{Constructions of normed modules}\label{ss:constr_norm_mod}
\subsubsection{Banach module induced by a symmetric sublinear map}
We fix some terminology. Consider a vector space \(\mathscr V\), a Riesz space \(U\), and a map \(\psi\colon\mathscr V\to U^+\).
Then we say that \(\psi\) is \textbf{positively homogeneous} if \(\psi(\lambda{\sf v})=\lambda\,\psi({\sf v})\) for every
\({\sf v}\in\mathscr V\) and \(\lambda\in\R^+\), while we say that \(\psi\) is \textbf{subadditive} if
\(\psi({\sf v}+{\sf w})\leq\psi({\sf v})+\psi({\sf w})\) for every \({\sf v},{\sf w}\in\mathscr V\).
The map \(\psi\) is said to be \textbf{sublinear} provided it is both positively homogeneous and subadditive.
Finally, we say that the map \(\psi\) is \textbf{symmetric} provided it satisfies \(\psi(-{\sf v})=\psi({\sf v})\) for every \({\sf v}\in\mathscr V\).
\begin{theorem}[Banach module generated by a symmetric sublinear map]\label{thm:module_generated}
Let \((\mathcal U,U,V)\) be a metric \(f\)-structure. Let \(\mathscr V\) be a vector space and \(\psi\colon\mathscr V\to V^+\)
a symmetric, sublinear mapping. Then there exists a unique couple \((\mathscr M_{\langle\psi\rangle},T_{\langle\psi\rangle})\)
-- where \(\mathscr M_{\langle\psi\rangle}\) is a \(V\)-Banach \(U\)-module and the operator
\(T_{\langle\psi\rangle}\colon\mathscr V\to\mathscr M_{\langle\psi\rangle}\) is linear -- such that the following properties are verified:
\begin{itemize}
\item[\(\rm i)\)] \(|T_{\langle\psi\rangle}{\sf v}|=\psi({\sf v})\) for every \({\sf v}\in\mathscr V\).
\item[\(\rm ii)\)] \(\mathscr G(T_{\langle\psi\rangle}(\mathscr V))\) is dense in \(\mathscr M_{\langle\psi\rangle}\).
\end{itemize}
Uniqueness is up to unique isomorphism: given another couple \((\mathscr M,T)\) with the same properties, there exists
a unique isomorphism of \(V\)-Banach \(U\)-modules \(\Phi\colon\mathscr M_{\langle\psi\rangle}\to\mathscr M\) such that
\[\begin{tikzcd}
\mathscr V \arrow[r,"T_{\langle\psi\rangle}"] \arrow[rd,swap,"T"] &
\mathscr M_{\langle\psi\rangle} \arrow[d,"\Phi"] \\ & \mathscr M
\end{tikzcd}\]
is a commutative diagram.
\end{theorem}
\begin{proof}\ \\
{\color{blue}\textsc{Existence.}} Let us denote by \(\bar{\mathcal F}\) the family of all sequences \((u_n,{\sf v}_n)_{n\in\N}\) such that
\((u_n)_{n\in\N}\in\mathcal P({\bf 1}_U)\), \(({\sf v}_n)_{n\in\N}\subset\mathscr V\), and \(\big(u_n\psi({\sf v}_n)\big)_{n\in\N}\) is order-bounded
in \(V\). We introduce a relation \(\sim\) on the set \(\bar{\mathcal F}\): given any \((u_n,{\sf v}_n)_n,(\tilde u_m,\tilde{\sf v}_m)_m\in\bar{\mathcal F}\),
we declare that \((u_n,{\sf v}_n)_n\sim(\tilde u_m,\tilde{\sf v}_m)_m\) if and only if
\[
u_n\tilde u_m\psi({\sf v}_n-\tilde{\sf v}_m)=0,\quad\text{ for every }n,m\in\N.
\]
One can readily check that \(\sim\) is an equivalence relation on \(\bar{\mathcal F}\): reflexivity follows from \(\psi(0)=0\),
symmetry from the symmetry of \(\psi\), and transitivity from the subadditivity of \(\psi\). We then define
\[
\mathcal F\coloneqq\bar{\mathcal F}/\sim.
\]
For any \((u_n,{\sf v}_n)_n\in\bar{\mathcal F}\), we denote by \([u_n,{\sf v}_n]_n\in\bar{\mathcal F}\) its equivalence class with
respect to \(\sim\). We set
\[\begin{split}
[u_n,{\sf v}_n]_n+[\tilde u_m,\tilde{\sf v}_m]_m\coloneqq[u_n\tilde u_m,{\sf v}_n+\tilde{\sf v}_m]_{n,m}&,
\quad\forall[u_n,{\sf v}_n]_n,[\tilde u_m,\tilde{\sf v}_m]_m\in\mathcal F,\\
u\cdot[\tilde u_m,\tilde{\sf v}_m]_m\coloneqq[u_i\tilde u_m,\lambda_i\tilde{\sf v}_m]_{i,m}&,
\quad\forall u=\sum_{i=1}^n\lambda_i u_i\in\mathcal S(U),\,[\tilde u_m,\tilde{\sf v}_m]_m\in\mathcal F,\\
\big|[u_n,{\sf v}_n]_n\big|\coloneqq\sup_{n\in\N}u_n\psi({\sf v}_n)&,\quad\forall[u_n,{\sf v}_n]_n\in\mathcal F.
\end{split}\]
Routine verifications show the well-posedness of the resulting operations
\begin{equation}\label{eq:pre_operations}
+\colon\mathcal F\times\mathcal F\to\mathcal F,\qquad\cdot\colon\mathcal S(U)\times\mathcal F\to\mathcal F,\qquad|\cdot|\colon\mathcal F\to V^+.
\end{equation}
We also define the map \(\tilde T\colon\mathscr V\to\mathcal F\) as \(\tilde T({\sf v})\coloneqq[{\bf 1}_U,{\sf v}]\) for all \({\sf v}\in\mathscr V\),
and the distance \(\sfd_{\mathcal F}\) on \(\mathcal F\) as
\[
\sfd_{\mathcal F}(w,\tilde w)\coloneqq\sfd_V(|w-\tilde w|,0),\quad\text{ for every }w,\tilde w\in\mathcal F.
\]
Next we denote by \((\mathscr M_{\langle\psi\rangle},\sfd_{\mathscr M_{\langle\psi\rangle}})\) the metric completion of \((\mathcal F,\sfd_{\mathcal F})\)
and by \(\iota\colon\mathcal F\to{\mathscr M}_{\langle\psi\rangle}\) the canonical isometric embedding map. Define also
\(T_{\langle\psi\rangle}\colon\mathscr V\to\mathscr M_{\langle\psi\rangle}\) as \(T_{\langle\psi\rangle}\coloneqq\iota\circ\tilde T\).
With a slight abuse of notation, we regard \(\mathcal F\) as a subset of \(\mathscr M_{\langle\psi\rangle}\). Standard verifications show that
the operations in \eqref{eq:pre_operations} are Cauchy-continuous, so they can be uniquely extended to continuous maps
\[
+\colon\mathscr M_{\langle\psi\rangle}\times\mathscr M_{\langle\psi\rangle}\to\mathscr M_{\langle\psi\rangle},\qquad
\cdot\colon U\times\mathscr M_{\langle\psi\rangle}\to\mathscr M_{\langle\psi\rangle},\qquad|\cdot|\colon\mathscr M_{\langle\psi\rangle}\to V^+.
\]
By an approximation argument, one can show that \(\mathscr M_{\langle\psi\rangle}\) is a \(U\)-module, that
\(|\cdot|\colon\mathscr M_{\langle\psi\rangle}\to V^+\) verifies the \(V\)-pointwise norm axioms, and that
\[
\sfd_{\mathscr M_{\langle\psi\rangle}}(w,\tilde w)\coloneqq\sfd_V(|w-\tilde w|,0),\quad\text{ for every }w,\tilde w\in\mathscr M_{\langle\psi\rangle}.
\]
In order to prove that \(\mathscr M_{\langle\psi\rangle}\) is a \(V\)-Banach \(U\)-module, it only remains to check the validity of the
glueing property. Fix any \((u_n)_{n\in\N}\in\mathcal P({\bf 1}_U)\) and \((w_n)_{n\in\N}\subset\mathscr M_{\langle\psi\rangle}\) such that \((u_n|w_n|)_{n\in\N}\)
is order-bounded in \(V\). In view of Definition \ref{def:metric_f-structure} v), for any \(k\in\N\) we can find \((w_n^k)_{n\in\N}\subset\mathcal F\)
such that \((u_n|w_n-w_n^k|)_{n\in\N}\) is order-bounded in \(V\) -- thus \((u_n|w_n^k|)_{n\in\N}\) is order-bounded in \(V\) -- and
\[
\sum_{n\in\N}\sfd_V(u_n|w_n-w_n^k|,0)\leq\frac{1}{2^k},\qquad\sfd_V\Big(\sup_{n\in\N}u_n|w_n-w_n^k|,0\Big)\leq\frac{1}{2^k}.
\]
Writing \(w_n^k\) in the form \([\tilde u_{n,j}^k,{\sf v}_{n,j}^k]_j\), we define the element \(z^k\in\mathcal F\) as
\(z^k\coloneqq[u_n\tilde u_{n,j}^k,{\sf v}_{n,j}^k]_{n,j}\). Then
\[\begin{split}
\sfd_{\mathscr M_{\langle\psi\rangle}}(z^k,z^{k+1})&=\sfd_V(|z^k-z^{k+1}|,0)=\sfd_V\Big(\sup_{n\in\N}u_m|w_n^k-w_n^{k+1}|,0\Big)\\
&\leq\sfd_V\Big(\sup_{n\in\N}u_m|w_n^k-w_n|,0\Big)+\sfd_V\Big(\sup_{n\in\N}u_m|w_n-w_n^{k+1}|,0\Big)\leq\frac{3}{2^{k+1}},
\end{split}\]
for every \(k\in\N\). This implies that \((z^k)_{k\in\N}\) is a Cauchy sequence in \((\mathscr M_{\langle\psi\rangle},\sfd_{\mathscr M_{\langle\psi\rangle}})\),
so that it converges to some element \(z\in\mathscr M_{\langle\psi\rangle}\). For any \(n\in\N\), we deduce that
\[
\sfd_{\mathscr M_{\langle\psi\rangle}}(u_n\cdot z,u_n\cdot w_n)=\lim_{k\to\infty}\sfd_{\mathscr M_{\langle\psi\rangle}}(u_n\cdot z^k,u_n\cdot w_n)
=\lim_{k\to\infty}\sfd_V(u_n|w_n-w_n^k|,0)\leq\lim_{k\to\infty}\frac{1}{2^k}=0,
\]
so that \(u_n\cdot z=u_n\cdot w_n\). This shows that \(z=\sum_{n\in\N}u_n\cdot w_n\), thus the glueing property is proved.

To conclude, let us check that \((\mathscr M_{\langle\psi\rangle},T_{\langle\psi\rangle})\) verifies i) and ii). The mapping
\(T_{\langle\psi\rangle}\) is linear and satisfies \(|T_{\langle\psi\rangle}{\sf v}|=|[1,{\sf v}]|=\psi({\sf v})\) for every
\({\sf v}\in\mathscr V\), so that i) is proved. Finally, \(\mathscr G(T_{\langle\psi\rangle}(\mathscr V))=\mathcal F\) and
\(\sum_{n\in\N}u_n\cdot T_{\langle\psi\rangle}{\sf v}_n=[u_n,{\sf v}_n]_n\) for all \([u_n,{\sf v}_n]_n\in\mathcal F\).
Being \(\mathcal F\) dense in \(\mathscr M_{\langle\psi\rangle}\), also ii) is proved.\\
{\color{blue}\textsc{Uniqueness.}} It is a consequence of Corollary \ref{cor:module_generated_cat}. Indeed, initial objects are colimits (of empty diagrams),
thus in particular they are unique up to a unique isomorphism (cf.\ with \cite{MacLane98}).
\end{proof}
\begin{proposition}\label{prop:module_generated_hom}
Let \((\mathcal U_1,U_1,V_1)\) be a metric \(f\)-structure and let \((\mathcal U_2,U_2,V_2,W_2,Z_2)\) be a dual system of metric
\(f\)-structures. Let \(\phi\colon(\mathcal U_1,U_1,V_1)\to(\mathcal U_2,U_2,V_2)\) be a homomorphism of metric \(f\)-structures.
Let \(\mathscr V\) be a vector space and \(\psi\colon\mathscr V\to V_1^+\) a symmetric, sublinear mapping. Let \(\mathscr N\) be
a \(Z_2\)-Banach \(U_2\)-module. Let \(S\colon\mathscr V\to\mathscr N\) be a linear map such that for some \(b\in W_2^+\) it holds
\[
|S{\sf v}|\leq b\,(\phi\circ\psi)({\sf v}),\quad\text{ for every }{\sf v}\in\mathscr V.
\]
Then there exists a unique linear operator \(\Phi\colon\mathscr M_{\langle\psi\rangle}\to\mathscr N\) such that \(\Phi\circ T_{\langle\psi\rangle}=S\) and
\[
|\Phi v|\leq b\,\phi(|v|),\quad\text{ for every }v\in\mathscr M_{\langle\psi\rangle}.
\]
In particular, the map \(\Phi\colon\mathscr M_{\langle\psi\rangle}\to\mathscr N\) is continuous. Moreover, it holds that
\[
\Phi(u\cdot v)=\phi(u)\cdot\Phi(v),\quad\text{ for every }u\in U_1\text{ and }v\in\mathscr M_{\langle\psi\rangle}.
\]
\end{proposition}
\begin{proof}
Since \(T_{\langle\psi\rangle}\) is linear, we have that \(\mathscr W\coloneqq T_{\langle\psi\rangle}(\mathscr V)\) is a vector subspace of \(\mathscr M_{\langle\psi\rangle}\).
Item ii) of Theorem \ref{thm:module_generated} ensures that \(\mathscr G(\mathscr W)\) is dense in \(\mathscr M_{\langle\psi\rangle}\). Moreover, let us define \(T\colon\mathscr W\to\mathscr N\) as
\[
T(T_{\langle\psi\rangle}{\sf v})\coloneqq S{\sf v},\quad\text{ for every }{\sf v}\in\mathscr V.
\]
Since \(|S{\sf v}|\leq b\,(\phi\circ\psi)({\sf v})=b\,\phi(|T_{\langle\psi\rangle}{\sf v}|)\) by item i) of Theorem \ref{thm:module_generated}, we deduce that
\(T\) is well-defined. Notice also that \(T\) is linear and satisfies \(|Tv|\leq b\,\phi(|v|)\) for every \(v\in\mathscr W\). Therefore, the statement follows
from Theorem \ref{thm:main_ext_hom} applied to the operator \(T\).
\end{proof}
\begin{corollary}\label{cor:module_generated_hom}
Let \((\mathcal U,U,V)\) be a metric \(f\)-structure. Let \(\mathscr V\) be a vector space and \(\psi\colon\mathscr V\to V^+\) a symmetric, sublinear mapping.
Let \(\mathscr N\) be a \(V\)-Banach \(U\)-module and \(S\colon\mathscr V\to\mathscr N\) a linear operator such that \(|S{\sf v}|\leq\psi({\sf v})\) for
every \({\sf v}\in\mathscr V\). Let \((\mathscr M_{\langle\psi\rangle},T_{\langle\psi\rangle})\) be as in Theorem \ref{thm:module_generated}. Then there exists
a unique homomorphism of \(V\)-Banach \(U\)-modules \(\Phi\colon\mathscr M_{\langle\psi\rangle}\to\mathscr N\) with \(\Phi\circ T_{\langle\psi\rangle}=S\).
Moreover, it holds that \(|\Phi v|\leq|v|\) for every \(v\in\mathscr M_{\langle\psi\rangle}\).
\end{corollary}
\begin{proof}
Use Proposition \ref{prop:module_generated_hom} with \((\mathcal U_1,U_1,V_1)\coloneqq(\mathcal U,U,V)\), \((\mathcal U_2,U_2,V_2,W_2,Z_2)\coloneqq(\mathcal U,U,V,U,V)\),
\(\phi\coloneqq{\rm id}_{\mathcal U}\), and \(b\coloneqq{\bf 1}_U\).
\end{proof}

Next, we aim to interpret the couple \((\mathscr M_{\langle\psi\rangle},T_{\langle\psi\rangle})\) given by Theorem \ref{thm:module_generated} in categorical terms.
Given \((\mathcal U,U,V)\), \(\mathscr V\), and \(\psi\) as in Theorem \ref{thm:module_generated}, we denote by \({\bf C}_\psi\) the category defined as follows:
\begin{itemize}
\item[\(\rm i)\)] The objects of \({\bf C}_\psi\) are the couples \((\mathscr M,T)\), with \(\mathscr M\) a \(V\)-Banach \(U\)-module and \(T\colon\mathscr V\to\mathscr M\)
a linear map satisfying \(|T{\sf v}|\leq\psi({\sf v})\) for every \({\sf v}\in\mathscr V\).
\item[\(\rm ii)\)] The morphisms between two objects \((\mathscr M_1,T_1)\) and \((\mathscr M_2,T_2)\) of \({\bf C}_\psi\) are given by those morphisms \(\Phi\colon\mathscr M_1\to\mathscr M_2\)
in the category \({\bf BanMod}_{(\mathcal U,U,V)}\) for which
\[\begin{tikzcd}
\mathscr V  \arrow[r,"T_1"] \arrow[rd,swap,"T_2"] & \mathscr M_1 \arrow[d,"\Phi"] \\
& \mathscr M_2
\end{tikzcd}\]
is a commutative diagram.
\end{itemize}
\begin{corollary}\label{cor:module_generated_cat}
Let \((\mathcal U,U,V)\) be a metric \(f\)-structure. Let \(\mathscr V\) be a vector space and \(\psi\colon\mathscr V\to V^+\) a symmetric, sublinear mapping.
Then \((\mathscr M_{\langle\psi\rangle},T_{\langle\psi\rangle})\) is the initial object of the category \({\bf C}_\psi\), meaning that for any object
\((\mathscr M,T)\) of \({\bf C}_\psi\) there exists a unique morphism \(\Phi\colon(\mathscr M_{\langle\psi\rangle},T_{\langle\psi\rangle})\to(\mathscr M,T)\).
\end{corollary}
\begin{proof}
Let \((\mathscr M,T)\) be an arbitrary object of \({\bf C}_\psi\). Then Corollary \ref{cor:module_generated_hom} ensures that there exists a unique morphism
\(\Phi\colon(\mathscr M_{\langle\psi\rangle},T_{\langle\psi\rangle})\to(\mathscr M,T)\) in \({\bf C}_\psi\) such that \(\Phi\circ T_{\langle\psi\rangle}=T\), as desired.
\end{proof}
\subsubsection{Pushforward of a normed module}\label{ss:pushfrwd_mod}
As a consequence of Theorem \ref{thm:module_generated}, we can prove that each homomorphism of metric \(f\)-structures induces a `pushforward functor'
between the respective categories of Banach modules:
\begin{theorem}[Pushforward of a normed module]\label{thm:pshfrwd_mod}
Let \((\mathcal U_1,U_1,V_1)\) and \((\mathcal U_2,U_2,V_2)\) be metric \(f\)-structures. Let \(\phi\colon\mathcal U_1\to\mathcal U_2\)
be a homomorphism of metric \(f\)-structures. Let \(\mathscr M\) be a \(V_1\)-normed \(U_1\)-module. Then there exists a unique
couple \((\phi_*\mathscr M,\phi_*)\) -- where \(\phi_*\mathscr M\) is a \(V_2\)-Banach \(U_2\)-module and the operator
\(\phi_*\colon\mathscr M\to\phi_*\mathscr M\) is linear -- such that the following properties are verified:
\begin{itemize}
\item[\(\rm i)\)] \(|\phi_*v|=\phi(|v|)\) for every \(v\in\mathscr M\).
\item[\(\rm ii)\)] \(\mathscr G(\phi_*(\mathscr M))\) is dense in \(\phi_*\mathscr M\).
\end{itemize}
Uniqueness is up to unique isomorphism: given another couple \((\mathscr N,T)\) with the same properties, there
exists a unique isomorphism of \(V_2\)-Banach \(U_2\)-modules \(\Phi\colon\phi_*\mathscr M\to\mathscr N\) such that
\[\begin{tikzcd}
\mathscr M \arrow[r,"\phi_*"] \arrow[rd,swap,"T"] &
\phi_*\mathscr M \arrow[d,"\Phi"] \\ & \mathscr N
\end{tikzcd}\]
is a commutative diagram.
\end{theorem}
\begin{proof}
Define \(\psi\colon\mathscr M\to V_2^+\) as \(\psi(v)\coloneqq\phi(|v|)\) for every \(v\in\mathscr M\). Notice that \(\psi\) is symmetric and
sublinear, thus it makes sense to consider the \(V_2\)-Banach \(U_2\)-module \(\phi_*\mathscr M\coloneqq\mathscr M_{\langle\psi\rangle}\) and
the linear operator \(\phi_*\coloneqq T_{\langle\psi\rangle}\colon\mathscr M\to\phi_*\mathscr M\). Observe that \(|\phi_*v|=\psi(v)=\phi(|v|)\)
for every \(v\in\mathscr M\), which shows i), while ii) is only a rephrasing of Theorem \ref{thm:module_generated} ii). This proves the existence
part of the statement. Finally, the uniqueness part follows from the uniqueness stated in Theorem \ref{thm:module_generated}.
\end{proof}
\begin{proposition}\label{prop:univ_prop}
Let \((\mathcal U_1,U_1,V_1)\) be a metric \(f\)-structure and let \((\mathcal U_2,U_2,V_2,W_2,Z_2)\) be a dual system of metric
\(f\)-structures. Let \(\phi\colon(\mathcal U_1,U_1,V_1)\to(\mathcal U_2,U_2,V_2)\) be a homomorphism of metric \(f\)-structures.
Fix a \(V_1\)-normed \(U_1\)-module \(\mathscr M\) and a \(Z_2\)-Banach \(U_2\)-module \(\mathscr N\).
Let \(T\colon\mathscr M\to\mathscr N\) be a linear operator such that there exists an element \(b\in W_2^+\) satisfying
\begin{equation}\label{eq:univ_prop_hp}
|Tv|\leq b\,\phi(|v|),\quad\text{ for every }v\in\mathscr M.
\end{equation}
Then there exists a unique homomorphism \(\hat T\in\textsc{Hom}(\phi_*\mathscr M,\mathscr N)\) such that
\begin{equation}\label{eq:univ_prop_cl}
\hat T(\phi_*v)=Tv,\quad\text{ for every }v\in\mathscr M.
\end{equation}
Moreover, it holds that \(|\hat T w|\leq b|w|\) for every \(w\in\phi_*\mathscr M\).
\end{proposition}
\begin{proof}
Recall from (the proof of) Theorem \ref{thm:pshfrwd_mod} that, letting \(\psi\colon\mathscr M\to V_2^+\)
be \(\psi(v)\coloneqq\phi(|v|)\), it holds
\((\phi_*\mathscr M,\phi_*)\cong(\mathscr M_{\langle\psi\rangle},T_{\langle\psi\rangle})\).
The statement then follows from Proposition \ref{prop:module_generated_hom}.
\end{proof}
\begin{corollary}\label{cor:phfwd_functor}
Let \((\mathcal U_1,U_1,V_1),(\mathcal U_2,U_2,V_2)\) be metric \(f\)-structures and \(\phi\colon\mathcal U_1\to\mathcal U_2\)
a homomorphism of metric \(f\)-structures. Let \(\mathscr M,\mathscr N\) be \(V_1\)-normed \(U_1\)-modules.
Let \(T\in\textsc{Hom}(\mathscr M,\mathscr N)\) be given. Then there exists a unique homomorphism
\(\phi_*T\in\textsc{Hom}(\phi_*\mathscr M,\phi_*\mathscr N)\) such that
\[\begin{tikzcd}
\mathscr M \arrow[r,"T"] \arrow[d,swap,"\phi_*"] & \mathscr N \arrow[d,"\phi_*"] \\
\phi_*\mathscr M \arrow[r,swap,"\phi_*T"] & \phi_*\mathscr N
\end{tikzcd}\]
is a commutative diagram. Moreover, if both \(U_1\) and \(U_2\) are Dedekind complete, then it holds that
\[
|\phi_*T|\leq\phi(|T|).
\]
\end{corollary}
\begin{proof}
Define \(\tilde T\coloneqq\phi_*\circ T\colon\mathscr M\to\phi_*\mathscr N\).
Letting \(w\in U_1^+\) be so that \(|Tv|\leq w|v|\) for every \(v\in\mathscr M\),
\begin{equation}\label{eq:phfwrd_hom}
|\tilde T v|=|\phi_*(Tv)|=\phi(|Tv|)\leq\phi(w|v|)=\phi(w)\phi(|v|),\quad\text{ for every }v\in\mathscr M.
\end{equation}
Therefore, the existence and the uniqueness of \(\phi_*T\) follow from Proposition \ref{prop:univ_prop}
applied to \(\tilde T\). Finally, suppose \(U_1\) and \(U_2\) are Dedekind complete. Then we can choose
\(w\coloneqq|T|\) in \eqref{eq:phfwrd_hom}, so that we have \(|(\phi_*T)w|\leq\phi(|T|)|w|\) for
every \(w\in\phi_*\mathscr M\), whence it follows that \(|\phi_*T|\leq\phi(|T|)\).
\end{proof}

Combining Theorem \ref{thm:pshfrwd_mod} and Corollary \ref{cor:phfwd_functor}, any homomorphism \(\phi\colon(\mathcal U_1,U_1,V_1)\to(\mathcal U_2,U_2,V_2)\) induces a functor
\[
\phi_*\colon{\bf BanMod}_{(\mathcal U_1,U_1,V_1)}\to{\bf BanMod}_{(\mathcal U_2,U_2,V_2)}.
\]
We point out that, even though we chose the term `pushforward' by analogy with \cite[Section 1.6]{Gigli14}, from the categorical perspective the correct
term would be \emph{direct image functor}.
\begin{remark}\label{rmk:pshfrwd_Hilbert}{\rm
Let \((\mathcal U_1,U_1,V_1,V_1,Z_1)\), \((\mathcal U_2,U_2,V_2,V_2,Z_2)\) be dual systems of metric \(f\)-structures. Let
\(\phi\colon\mathcal U_1\to\mathcal U_2\) be a homomorphism of dual systems. Let \(\mathscr H\) be a \(V_1\)-Hilbert \(U_1\)-module. Then
\begin{subequations}\begin{align}
\label{eq:pshfrwd_Hilbert_1}
\phi_*\mathscr H,&\quad\text{ is a }V_2\text{-Hilbert }U_2\text{-module,}\\
\label{eq:pshfrwd_Hilbert_2}
\phi_*v\cdot\phi_*w=\phi(v\cdot w),&\quad\text{ for every }v,w\in\mathscr H.
\end{align}\end{subequations}
To prove \eqref{eq:pshfrwd_Hilbert_1}, notice that Theorem \ref{thm:pshfrwd_mod} i) implies that the elements of \(\mathscr G(\phi_*(\mathscr H))\) --
thus all the elements of \(\phi_*\mathscr H\), thanks to Theorem \ref{thm:pshfrwd_mod} ii) -- satisfy the pointwise parallelogram rule. Also,
\[
\phi_*v\cdot\phi_*w=\frac{1}{2}\big(|\phi_*(v+w)|^2-|\phi_*v|^2-|\phi_*w|^2\big)
=\phi\bigg(\frac{1}{2}\big(|v+w|^2-|v|^2-|w|^2\big)\bigg)=\phi(v\cdot w)
\]
hold for every \(v,w\in\mathscr H\), which shows that \eqref{eq:pshfrwd_Hilbert_2} is verified.
\fr}\end{remark}
\begin{theorem}\label{thm:embed_pshfrwd_dual}
Let \((\mathcal U_1,U_1,V_1,W_1,Z_1)\) and \((\mathcal U_2,U_2,V_2,W_2,Z_2)\) be CR dual systems of metric \(f\)-structures.
Let \(\phi\colon\mathcal U_1\to\mathcal U_2\) be a homomorphism of dual systems. Let \(\mathscr M\) be a \(V_1\)-Banach \(U_1\)-module.
Then there exists a unique homomorphism \({\rm I}_\phi\in\textsc{Hom}\big(\phi_*\mathscr M^*,(\phi_*\mathscr M)^*\big)\) such that
\begin{equation}\label{eq:embed_pshfrwd_dual}
\langle{\rm I}_\phi(\phi_*\omega),\phi_*v\rangle=\phi(\langle\omega,v\rangle),\quad\text{ for every }\omega\in\mathscr M^*\text{ and }v\in\mathscr M.
\end{equation}
Moreover, it holds that
\begin{equation}\label{eq:embed_pshfrwd_dual_2}
|{\rm I}_\phi(\eta)|=|\eta|,\quad\text{ for every }\eta\in\phi_*\mathscr M^*.
\end{equation}
\end{theorem}
\begin{proof}
Given any \(\omega\in\mathscr M^*\), we define the operator \(\tilde i_\phi(\omega)\colon\mathscr M\to Z_2\) as
\[
\tilde i_\phi(\omega)v\coloneqq\phi(\langle\omega,v\rangle),\quad\text{ for every }v\in\mathscr M.
\]
Note that \(\tilde i_\phi(\omega)\) is linear and satisfies \(|\tilde i_\phi(\omega)v|\leq\phi(|\omega|)\phi(|v|)\) for all \(v\in\mathscr M\).
Hence, we know from Proposition \ref{prop:univ_prop} that there is a unique element \(i_\phi(\omega)\in(\phi_*\mathscr M)^*\) such that
\(|i_\phi(\omega)|\leq\phi(|\omega|)\) and
\[
\langle i_\phi(\omega),\phi_*v\rangle=\tilde i_\phi(\omega)v=\phi(\langle\omega,v\rangle),\quad\text{ for every }v\in\mathscr M.
\]
Since the resulting operator \(i_\phi\colon\mathscr M^*\to(\phi_*\mathscr M)^*\) is linear, by applying Proposition \ref{prop:univ_prop}
again we deduce that there exists a unique \({\rm I}_\phi\in\textsc{Hom}\big(\phi_*\mathscr M^*,(\phi_*\mathscr M)^*\big)\) with
\(|{\rm I}_\phi|\leq 1\) such that \eqref{eq:embed_pshfrwd_dual} holds.

It remains to check \eqref{eq:embed_pshfrwd_dual_2}. Thanks to Proposition \ref{prop:cont_norm_mod_ops} i) and
Theorem \ref{thm:pshfrwd_mod} ii), it suffices to prove \eqref{eq:embed_pshfrwd_dual_2} for \(\eta\in\mathscr G(\phi_*(\mathscr M^*))\),
say \(\eta=\sum_{n\in\N}u_n\cdot\phi_*\omega_n\) with \((u_n)_n\in\mathcal P({\bf 1}_{U_2})\) and \((\omega_n)_n\subset\mathscr M^*\). Given any \(n\in\N\),
we deduce from Lemma \ref{lem:equiv_Hom_pn} that \(|\omega_n|=\sup\big\{\langle\omega_n,v\rangle\,:\,v\in\mathscr M,\,|v|\leq 1\big\}\),
so we can pick a sequence \((v_n^i)_{i\in\N}\subset\mathscr M\) such that \(|v_n^i|\leq 1\) for all \(i\in\N\) and
\(|\omega_n|=\sup_i\langle\omega_n,v_n^i\rangle\). Then
\begin{equation}\label{eq:embed_pshfrwd_dual_aux}
\phi(|\omega_n|)=|\phi_*\omega_n|=\sup_{i\in\N}\phi(\langle\omega_n,v_n^i\rangle)\overset{\eqref{eq:equiv_Hom_pn}}=
\sup_{i\in\N}\,\big\langle{\rm I}_\phi(\phi_*\omega_n),\phi_*v_n^i\big\rangle.
\end{equation}
Multiplying by \(u_n\) and summing over \(n\), we deduce (using Lemma \ref{lem:equiv_Hom_pn} again and \(|\phi_*v_n^i|\leq 1\)) that
\[\begin{split}
|\eta|&=\sum_{n\in\N}u_n|\phi_*\omega_n|\overset{\eqref{eq:embed_pshfrwd_dual_aux}}=
\sum_{n\in\N}u_n\sup_{i\in\N}\,\langle{\rm I}_\phi(\phi_*\omega_n),\phi_*v_n^i\rangle
=\sum_{n\in\N}\sup_{i\in\N}\,\langle u_n\cdot{\rm I}_\phi(\phi_*\omega_n),\phi_*v_n^i\rangle\\
&=\sum_{n\in\N}\sup_{i\in\N}\,\langle{\rm I}_\phi(u_n\cdot\eta),\phi_*v_n^i\rangle\leq
\sum_{n\in\N}|{\rm I}_\phi(u_n\cdot\eta)|=\sum_{n\in\N}u_n|{\rm I}_\phi(\eta)|=|{\rm I}_\phi(\eta)|.
\end{split}\]
Since \(|{\rm I}_\phi|\leq 1\) yields the converse inequality \(|{\rm I}_\phi(\eta)|\leq|\eta|\),
the statement is finally achieved.
\end{proof}
\subsubsection{Completion of a normed module}
It follows from Theorem \ref{thm:pshfrwd_mod} that each \(V\)-normed \(U\)-module can be `completed' to a \(V\)-Banach \(U\)-module,
much like the metric completion of a normed space has a Banach space structure:
\begin{theorem}[Completion of a normed module]
Let \((\mathcal U,U,V)\) be a metric \(f\)-structure and let \(\mathscr M\) be a \(V\)-normed \(U\)-module. Then there exists a unique couple
\((\bar{\mathscr M},\iota)\) such that \(\bar{\mathscr M}\) is a \(V\)-Banach \(U\)-module and \(\iota\colon\mathscr M\to\bar{\mathscr M}\)
is a \(U\)-linear map with dense range satisfying \(|\iota v|=|v|\) for every \(v\in\mathscr M\).
Uniqueness is up to unique isomorphism: given another couple \((\tilde{\mathscr M},\tilde\iota)\) with the same properties,
there is a unique isomorphism \(\Phi\colon\bar{\mathscr M}\to\tilde{\mathscr M}\) of \(V\)-Banach \(U\)-modules such that
\[\begin{tikzcd}
\mathscr M \arrow[r,"\iota"] \arrow[rd,swap,"\tilde\iota"] &
\bar{\mathscr M} \arrow[d,"\Phi"] \\ & \tilde{\mathscr M}
\end{tikzcd}\]
is a commutative diagram. We say that the couple \((\bar{\mathscr M},\iota)\), or just \(\bar{\mathscr M}\), is the \textbf{completion} of \(\mathscr M\).

Moreover, if \(\mathscr M,\mathscr N\) are \(V\)-normed \(U\)-modules and \(T\in\textsc{Hom}(\mathscr M,\mathscr N)\) is given,
then there exists a unique homomorphism \(\bar T\in\textsc{Hom}(\bar{\mathscr M},\bar{\mathscr N})\) such that \(\bar T|_{\mathscr M}=T\),
where we regard \(\mathscr M\) and \(\mathscr N\) as subsets of \(\bar{\mathscr M}\) and \(\bar{\mathscr N}\), respectively.
If in addition \(U\) is Dedekind complete, then it holds \(|\bar T|=|T|\).
\end{theorem}
\begin{proof}
The identity mapping \({\rm id}_{\mathcal U}\colon\mathcal U\to\mathcal U\) is a homomorphism of metric \(f\)-structures from
\((\mathcal U,U,V)\) to itself, thus we can define \(\bar{\mathscr M}\coloneqq({\rm id}_{\mathcal U})_*\mathscr M\) and
\(\iota\coloneqq({\rm id}_{\mathcal U})_*\colon\mathscr M\to\bar{\mathscr M}\).
By Theorem \ref{thm:pshfrwd_mod}, to prove the first part of the statement amounts to showing that \(\iota\) is \(U\)-linear and that \(\iota(\mathscr M)\)
is dense in \(\bar{\mathscr M}\). The former property follows from Corollary \ref{cor:criterion_Hom}. About the latter, recall
that \(\mathscr G(\iota(\mathscr M))\) is dense in \(\bar{\mathscr M}\), thus it only remains to show that
\(\iota(\mathscr M)=\mathscr G(\iota(\mathscr M))\). The inclusion \(\iota(\mathscr M)\subset\mathscr G(\iota(\mathscr M))\) is trivial.
Conversely, fix \(w=\sum_{n\in\N}u_n\cdot\iota v_n\in\mathscr G(\iota(\mathscr M))\). Since \((u_n,\iota v_n)_{n\in\N}\in\Adm(\bar{\mathscr M})\),
we have that \((u_n,v_n)_{n\in\N}\in\Adm(\mathscr M)\) and thus it makes sense to consider \(v\coloneqq\sum_{n\in\N}u_n\cdot v_n\in\mathscr M\).
We claim that \(\iota v=w\), whence it follows that \(w\in\iota(\mathscr M)\) and thus \(\mathscr G(\iota(\mathscr M))\subset\iota(\mathscr M)\). Given that
\[
u_n\cdot\iota v=\iota(u_n\cdot v)=\iota(u_n\cdot v_n)=u_n\cdot\iota v_n=u_n\cdot w,\quad\text{ for every }n\in\N,
\]
we deduce from Lemma \ref{lem:locality_prop} that \(\iota v=w\). Therefore, the first part of the statement is proved.

About the second part of the statement, observe that \(\bar T\coloneqq({\rm id}_{\mathcal U})_*T\in\textsc{Hom}(\bar{\mathscr M},\bar{\mathscr N})\)
is the unique homomorphism extending \(T\). Note also that if \(\bar u\in U^+\) satisfies \(|Tv|\leq\bar u|v|\) for all \(v\in\mathscr M\), then
\(|\bar T v|\leq\bar u|v|\) for every \(v\in\bar{\mathscr M}\) by approximation. Finally, borrowing the notation from
the proof of Theorem \ref{thm:Hom_normed_module} (see \eqref{eq:def_F_T}), we get that \(\mathcal F_T=\mathcal F_{\bar T}\),
so that (assuming that \(U\) is Dedekind complete) we conclude that \(|\bar T|=|T|\).
\end{proof}
\subsection{Hahn--Banach extension theorem}\label{ss:Hahn-Banach}
The aim of this section is to obtain a normed module version of the Hahn--Banach extension theorem, as well as to investigate some of its basic consequences.
\medskip

Let \((\mathcal U,U,V)\) be a metric \(f\)-structure and \(\mathscr M\) a module over \(U\).
Then we say that a given map \(p\colon\mathscr M\to V^+\) is \textbf{\(U\)-sublinear} if it is subadditive
(i.e.\ \(p(v+w)\leq p(v)+p(w)\) for every \(v,w\in\mathscr M\)) and \textbf{positively \(U\)-homogeneous},
which means that \(p(u\cdot v)=u\,p(v)\) for all \(u\in U^+\) and \(v\in\mathscr M\).
\begin{lemma}[One-dimensional dominated extension]\label{lem:1D_ext}
Let \((\mathcal U,U,V)\) be a CR metric \(f\)-structure. Let \(\mathscr M\) be a \(V\)-normed \(U\)-module
and \(\mathscr N\subsetneq\mathscr M\) a \(V\)-normed \(U\)-submodule of \(\mathscr M\). Fix any \(z\in\mathscr M\)
such that \(u\cdot z\notin\mathscr N\) for every \(u\in\Idem(U)\setminus\{0\}\) with \(u\leq\nchi_{\{z\neq 0\}}\).
Let \(f\colon\mathscr N\to V\) be a \(U\)-linear map and \(p\colon\mathscr M\to V^+\) a \(U\)-sublinear map with \(f\leq p\) 
on \(\mathscr N\). Define \(\mathscr N_{+z}\coloneqq\mathscr N+U\cdot z\). Then there exists a \(U\)-linear map
\(\bar f\colon\mathscr N_{+z}\to V\) with \(\bar f|_{\mathscr N}=f\) such that \(\bar f\leq p\) on \(\mathscr N_{+z}\).
\end{lemma}
\begin{proof}
Given any \(v,w\in\mathscr N\), we can estimate
\[
f(v)-f(w)=f(v-w)\leq p(v-w)=p\big(v+z-(w+z)\big)\leq p(v+z)+p(-w-z),
\]
whence it follows that \(-p(-w-z)-f(w)\leq p(v+z)-f(v)\) for every \(v,w\in\mathscr N\). Plugging \(w=0\), we obtain that \(p(z)\leq p(v+z)-f(v)\)
for every \(v\in\mathscr N\), thus the Dedekind completeness of \(V\) ensures that \(b\coloneqq\inf\big\{p(v+z)-f(v)\,:\,v\in\mathscr N\big\}\in V\)
exists. Then we have that
\begin{equation}\label{eq:1D_ext_aux1}
-p(-w-z)-f(w)\leq b\leq p(v+z)-f(v),\quad\text{ for every }v,w\in\mathscr N.
\end{equation}
Plugging \(v=w=0\) in \eqref{eq:1D_ext_aux1} and multiplying by \(\nchi_{\{z=0\}}\), we deduce that \(\nchi_{\{z=0\}}\cdot b=0\), so that
\begin{equation}\label{eq:1D_ext_aux2}
\nchi_{\{z=0\}}\leq\nchi_{\{b=0\}}.
\end{equation}
Next we claim that for any \(v,\tilde v\in\mathscr N\) and \(u,\tilde u\in U\) it holds that
\begin{equation}\label{eq:1D_ext_aux3}
v+u\cdot z=\tilde v+\tilde u\cdot z\qquad\Longrightarrow\qquad v=\tilde v\;\text{ and }\;u\cdot b=\tilde u\cdot b.
\end{equation}
To prove it, suppose that \((u-\tilde u)\cdot z=\tilde v-v\). Pick a partition \((u_n)_{n\in\N}\) of \(\nchi_{\{u\neq\tilde u\}}\)
and \((w_n)_{n\in\N}\subset U\) such that \(u_n(u-\tilde u)w_n=u_n\) for every \(n\in\N\). Multiplying \((u-\tilde u)\cdot z=\tilde v-v\) by \(u_n w_n\), we obtain
\[
u_n\cdot z=(u_n(u-\tilde u)w_n)\cdot z=(u_n w_n)\cdot(\tilde v-v)\in\mathscr N,\quad\text{ for every }n\in\N.
\]
Hence, the glueing property of \(\mathscr N\) ensures that \(\nchi_{\{u\neq\tilde u\}}\cdot z\in\mathscr N\), thus accordingly
\(\nchi_{\{u\neq\tilde u\}}\leq\nchi_{\{z=0\}}\). This implies that \(\tilde v-v=u\cdot z-\tilde u\cdot z=0\) and (recalling
\eqref{eq:1D_ext_aux2}) that \(u\cdot b=\tilde u\cdot b\), proving \eqref{eq:1D_ext_aux3}.
Therefore, the map \(\bar f\colon\mathscr N_{+z}\to V\) defined as follows is well-posed:
\[
\bar f(v+u\cdot z)\coloneqq f(v)+u\cdot b,\quad\text{ for every }v\in\mathscr N\text{ and }u\in U.
\]
It is immediate to check that \(\bar f\) is a \(U\)-linear extension of \(f\). It only remains to show that \(\bar f\leq p\) on \(\mathscr N_{+z}\).
To this aim, fix \(v\in\mathscr N\) and \(u\in U\setminus\{0\}\). Pick a partition \((u_n)_{n\in\N}\) of \(\{u>0\}\), a partition \((\tilde u_n)_{n\in\N}\)
of \(\{u<0\}\), and \((w_n)_{n\in\N},(\tilde w_n)_{n\in\N}\subset U^+\) such that \(u_n u^+ w_n=u_n\) and \(\tilde u_n u^- \tilde w_n=\tilde u_n\)
for all \(n\in\N\). It follows from \eqref{eq:1D_ext_aux1} that \(p(w_n\cdot v+z)-f(w_n\cdot v)\geq b\) and
\(-p(\tilde w_n\cdot v-z)+f(\tilde w_n\cdot v)\leq b\). Multiplying by \(u_n u^+\) and \(\tilde u_n u^-\),
respectively, we obtain \(u_n\cdot(f(v)+u\cdot b)\leq u_n\cdot p(v+u\cdot z)\) and
\[
\tilde u_n\cdot(f(v)+u\cdot b)=\tilde u_n\cdot(f(v)-u^-\cdot b)\leq\tilde u_n\cdot p(v-u^-\cdot z)=\tilde u_n\cdot p(v+u\cdot z),
\]
respectively. Using the glueing property, we conclude that \(\bar f(v+u\cdot z)=f(v)+u\cdot b\leq p(v+u\cdot z)\).
\end{proof}
\begin{theorem}[Hahn--Banach extension theorem for normed modules]\label{thm:Hahn-Banach}
Let \((\mathcal U,U,V)\) be a CR metric \(f\)-structure. Let \(\mathscr M\) be a \(V\)-normed \(U\)-module and let
\(\mathscr N\subsetneq\mathscr M\) be a \(V\)-normed \(U\)-submodule of \(\mathscr M\). Let \(f\colon\mathscr N\to V\)
be a \(U\)-linear map and \(p\colon\mathscr M\to V^+\) a \(U\)-sublinear map with \(f\leq p\) on \(\mathscr N\).
Then there exists a \(U\)-linear map \(\bar f\colon\mathscr M\to V\) such that \(\bar f|_{\mathscr N}=f\) such
that \(\bar f\leq p\) on \(\mathscr M\).
\end{theorem}
\begin{proof}
Let us denote by \(\mathcal F\) the family of all couples \((\mathscr Q,g)\), where \(\mathscr Q\) is a \(V\)-normed \(U\)-submodule of \(\mathscr M\)
containing \(\mathscr N\) and \(g\colon\mathscr Q\to V\) is a \(U\)-linear extension of \(f\) satisfying \(q\leq p\) on \(\mathscr Q\). Clearly
\(\mathcal F\) is non-empty, as it contains \((\mathscr N,f)\). We endow \(\mathcal F\) with the partial order \(\preceq\) defined as follows:
given \((\mathscr Q,g),(\tilde{\mathscr Q},\tilde g)\in\mathcal F\), we declare that \((\mathscr Q,g)\preceq(\tilde{\mathscr Q},\tilde g)\)
provided \(\mathscr Q\subset\tilde{\mathscr Q}\) and \(\tilde g|_{\mathscr Q}=g\). It is easy to check that any totally ordered subset \(\mathcal C\)
of \((\mathcal F,\preceq)\) has an upper bound, namely \((\mathscr Q,g_0)\), where
\(\mathscr Q_0\coloneqq\bigcup_{(\mathscr Q,g)\in\mathcal F}\mathscr Q\) and \(g_0\colon\mathscr Q_0\to V\)
is given by \(g_0(v)\coloneqq g(v)\) for every \((\mathscr Q,g)\in\mathcal F\) with \(v\in\mathscr Q\).
Hence, an application of Zorn's lemma yields the existence of a maximal element \((\mathscr N_0,f_0)\) of \((\mathcal F,\preceq)\).
In order to conclude, we aim to show that \(\mathscr N_0=\mathscr M\). We argue by contradiction: suppose that
\(\mathscr M\setminus\mathscr N_0\neq\varnothing\). Fix any \(\tilde z\in\mathscr M\setminus\mathscr N_0\).
The fact that \(U\) is CR ensures that
\[
\exists\,q\coloneqq\sup\big\{u\in\Idem(U)\;\big|\;u\cdot\tilde z\in\mathscr N_0\big\}\in\Idem(U).
\]
Moreover, the glueing property implies that \(q\cdot\tilde z\in\mathscr N_0\). Then we define \(z\coloneqq(1-q)\cdot\tilde z\in\mathscr M\).
Observe that \(u\cdot z\notin\mathscr N_0\) holds for every \(u\in\Idem(U)\setminus\{0\}\) with \(u\leq1-q=\nchi_{\{z\neq 0\}}\). Therefore,
Lemma \ref{lem:1D_ext} yields the existence of a map \(\bar f\colon\mathscr N_0+U\cdot z\to V\) such that
\((\mathscr N_0+U\cdot z,\bar f)\in\mathcal F\) and \((\mathscr N_0,f_0)\preceq(\mathscr N_0+U\cdot z,\bar f)\),
which leads to a contradiction with the maximality of \((\mathscr N_0,f_0)\).
\end{proof}
\begin{corollary}\label{cor:conseq_Hahn-Banach}
Let \((\mathcal U,U,V,W,Z)\) be a CR dual system of metric \(f\)-structures.
Let \(\mathscr M\) be a \(V\)-Banach \(U\)-module. Let \(v\in\mathscr M\) satisfy \(|v|\in Z\cap U\) and \(\nchi_{\{v\neq 0\}}\in W\).
Then there exists an element \(\omega\in\mathscr M^*\) such that \(\langle\omega,v\rangle=|v|\) and \(|\omega|=\nchi_{\{v\neq 0\}}\).
\end{corollary}
\begin{proof}
Notice that \(U\cdot v\) is a \(V\)-Banach \(U\)-submodule of \(\mathscr M\). We define the map \(T\colon U\cdot v\to Z\) as
\[
T(u\cdot v)\coloneqq u|v|,\quad\text{ for every }u\in U.
\]
Clearly, \(T\) is well-posed and \(U\)-linear. Moreover, we have \(|T(u\cdot v)|=\nchi_{\{v\neq 0\}}|u\cdot v|\) for every
\(u\in U\), thus \(T\in\textsc{Hom}(U\cdot v,Z)\) and \(|T|\leq\nchi_{\{v\neq 0\}}\). Now let us define \(p\colon\mathscr M\to Z^+\)
as \(p(w)\coloneqq\nchi_{\{v\neq 0\}}|w|\) for every \(w\in\mathscr M\). It can be readily checked that \(p\) is \(U\)-sublinear.
Since \(T\leq p\) on \(U\cdot v\), an application of Theorem \ref{thm:Hahn-Banach} yields a \(U\)-linear map \(\omega\colon\mathscr M\to Z\)
satisfying \(\omega|_{U\cdot v}=T\) and \(\omega\leq p\) on \(\mathscr M\). The latter gives \(|\omega(w)|\leq\nchi_{\{v\neq 0\}}|w|\) for
every \(w\in\mathscr M\), which shows that \(\omega\in\mathscr M^*\) and \(|\omega|\leq\nchi_{\{v\neq 0\}}\). Notice also that
\(\langle\omega,v\rangle=|v|\) by construction. Therefore, in order to conclude it suffices to check that \(|\omega|\geq\nchi_{\{v\neq 0\}}\).
To this aim, pick a partition \((u_n)_{n\in\N}\) of \(\{v\neq 0\}\) and a sequence \((w_n)_{n\in\N}\subset U^+\) such that
\(u_n w_n|v|=u_n\) holds for every \(n\in\N\). Now fix any \(g\in W^+\) satisfying \(|\langle\omega,w\rangle|\leq g|w|\) for every \(w\in\mathscr M\).
Then for any \(n\in\N\) we can estimate
\[
u_n g=u_n w_n|v|g=g|(u_n w_n)\cdot v|\geq|\langle\omega,(u_n w_n)\cdot v\rangle|=u_n w_n|v|=u_n,
\]
which implies that \(g\geq\nchi_{\{v\neq 0\}}\) thanks to the arbitrariness of \(n\in\N\). The statement is achieved.
\end{proof}
\subsubsection{Reflexive Banach modules}\label{ss:refl_mod}
Let \((\mathcal U,U,V,W,Z)\) be a complete dual system of metric \(f\)-structures and let \(\mathscr M\) be a \(V\)-Banach \(U\)-module.
Then we denote the \textbf{bidual} of \(\mathscr M\) by \(\mathscr M^{**}\coloneqq(\mathscr M^*)^*\). Here, we are considering the dual
of \(\mathscr M^*\) with respect to the dual system \((\mathcal U,U,W,V,Z)\) (recall Remark \ref{rmk:invert_dual_syst}), so that \(\mathscr M^{**}\)
is a \(V\)-Banach \(U\)-module.
\begin{definition}[Embedding into the bidual]
Let \((\mathcal U,U,V,W,Z)\) be a complete dual system of metric \(f\)-structures. Let \(\mathscr M\) be a \(V\)-Banach \(U\)-module.
Then we define \({\rm J}_{\mathscr M}\colon\mathscr M\to\mathscr M^{**}\) as
\[
\langle{\rm J}_{\mathscr M}(v),\omega\rangle\coloneqq\langle\omega,v\rangle,\quad\text{ for every }v\in\mathscr M\text{ and }\omega\in\mathscr M^*.
\]
\end{definition}

Notice that the map \(\mathscr M\times\mathscr M^*\ni(v,\omega)\mapsto\langle{\rm J}_{\mathscr M}(v),\omega\rangle\in Z\) is \(U\)-bilinear.
Moreover, one has
\[
|\langle{\rm J}_{\mathscr M}(v),\omega\rangle|\leq|\omega||v|,\quad\text{ for every }v\in\mathscr M\text{ and }\omega\in\mathscr M^*.
\]
It follows that \({\rm J}_{\mathscr M}(v)\in\mathscr M^{**}\) for every \(v\in\mathscr M\), that
\({\rm J}_{\mathscr M}\in\textsc{Hom}(\mathscr M,\mathscr M^{**})\), and that \(|{\rm J}_{\mathscr M}|\leq 1\).
Under suitable assumptions, the homomorphism \({\rm J}_{\mathscr M}\) actually preserves the pointwise norm:
\begin{proposition}\label{prop:J_M_isom}
Let \((\mathcal U,U,V,W,Z)\) be a CR dual system of metric \(f\)-structures. Suppose that \({\rm S}(V)\leq{\rm S}(W)\).
Let \(\mathscr M\) be a \(V\)-Banach \(U\)-module. Then it holds that
\[
\big|{\rm J}_{\mathscr M}(v)\big|=|v|,\quad\text{ for every }v\in\mathscr M.
\]
\end{proposition}
\begin{proof}
Let \(v\in\mathscr M\) be fixed. We aim to show that \(|{\rm J}_{\mathscr M}(v)|\geq|v|\). In view of Remark \ref{rmk:relation_H(V)},
we know that \({\rm S}(V)\leq{\rm S}(Z)\). Hence, applying Lemma \ref{lem:tech_localisation} we obtain a partition \((u_n)_{n\in\N}\)
of \(\{v\neq 0\}\) such that \(u_n|v|\in Z\cap U\) and \(\nchi_{\{u_n\cdot v\neq 0\}}\in W\) for every \(n\in\N\). Using Corollary
\ref{cor:conseq_Hahn-Banach}, we can find a sequence \((\omega_n)_{n\in\N}\subset\mathscr M^*\) such that \(\langle\omega_n,u_n\cdot v\rangle=u_n|v|\)
and \(|\omega_n|=\nchi_{\{u_n\cdot v\neq 0\}}\) for all \(n\in\N\). Then
\[
u_n\langle{\rm J}_{\mathscr M}(v),\omega_n\rangle=\langle\omega_n,u_n\cdot v\rangle=u_n|v|=u_n|\omega_n||v|,\quad\text{ for every }n\in\N.
\]
This implies that \(u_n|{\rm J}_{\mathscr M}(v)|\geq u_n|v|\) for every \(n\in\N\), whence it follows that \(|{\rm J}_{\mathscr M}(v)|\geq|v|\).
\end{proof}

In view of Proposition \ref{prop:J_M_isom}, it is then natural to give the following definition of reflexivity:
\begin{definition}[Reflexive Banach module]
Let \((\mathcal U,U,V,W,Z)\) be a CR dual system of metric \(f\)-structures. Suppose that \({\rm S}(V)={\rm S}(W)\).
Then we say that a \(V\)-Banach \(U\)-module \(\mathscr M\) is \textbf{reflexive} provided the embedding operator
\({\rm J}_{\mathscr M}\colon\mathscr M\to\mathscr M^{**}\) is surjective. 
\end{definition}
\subsection{Hilbert modules}\label{ss:Hilbert_mod}
In this section we investigate the properties of Hilbert modules.
Among others, we will prove a Cauchy--Schwarz inequality (Lemma \ref{lem:Cauchy-Schwarz}) and a Riesz-representation-type result (Theorem \ref{thm:RRT}),
we will study orthogonal projections (Theorem \ref{thm:Hilbert_proj} and Proposition \ref{prop:orth_proj}), and we will show that Hilbert modules are reflexive (Proposition \ref{prop:Hilb_are_refl}).
\medskip

Let \((\mathcal U,U,V,V,Z)\) be a dual system of metric \(f\)-structures and \(\mathscr H\) a \(V\)-Hilbert \(U\)-module. Then
\begin{equation}\label{eq:conseq_parall}
v\cdot w=\frac{1}{4}|v+w|^2-\frac{1}{4}|v-w|^2,\quad\text{ for every }v,w\in\mathscr H.
\end{equation}
Indeed, recalling the definition of the pointwise parallelogram rule, we obtain that
\[\begin{split}
\frac{1}{4}|v+w|^2&=\frac{1}{4}|v|^2+\frac{1}{4}|w|^2+\frac{1}{2}v\cdot w,\\
\frac{1}{4}|v-w|^2&=\frac{1}{4}|v|^2+\frac{1}{4}|w|^2-\frac{1}{2}v\cdot w.
\end{split}\]
Subtracting the second identity from the first one, we get \eqref{eq:conseq_parall}.
\begin{lemma}[Cauchy--Schwarz inequality]\label{lem:Cauchy-Schwarz}
Let \((\mathcal U,U,V,V,Z)\) be a dual system of metric \(f\)-structures and \(\mathscr H\) a \(V\)-Hilbert \(U\)-module. Then
\begin{equation}\label{eq:CS}
|v\cdot w|\leq|v||w|,\quad\text{ for every }v,w\in\mathscr H.
\end{equation}
\end{lemma}
\begin{proof}
Using \eqref{eq:conseq_parall} and the fact that \(|v|-|w|\leq|v+w|\leq|v|+|w|\), we obtain that
\[\begin{split}
v\cdot w&=\frac{1}{4}|v+w|^2-\frac{1}{4}|v-w|^2\leq\frac{1}{4}\big((|v|+|w|)^2-(|v|-|w|)^2\big)\\
&=\frac{1}{4}\big(|v|^2+|w|^2+2|v||w|-|v|^2-|w|^2+2|v||w|\big)=|v||w|.
\end{split}\]
We also have that \(-(v\cdot w)=(-v)\cdot w\leq|-v||w|=|v||w|\). Therefore, \eqref{eq:CS} is proved.
\end{proof}

Given a \(V\)-Hilbert \(U\)-module \(\mathscr H\) and a \(V\)-Hilbert \(U\)-submodule \(\mathscr N\) of \(\mathscr H\),
we define the \textbf{orthogonal complement} of \(\mathscr N\) in \(\mathscr H\) as
\[
\mathscr N^\perp\coloneqq\big\{v\in\mathscr H\;\big|\;v\cdot w=0,\text{ for every }w\in\mathscr N\big\}.
\]
One can readily check that \(\mathscr N^\perp\) is a \(V\)-Hilbert \(U\)-submodule of \(\mathscr H\).
\begin{theorem}[Hilbert projection theorem for Hilbert modules]\label{thm:Hilbert_proj}
Let \((\mathcal U,U,V,V,Z)\) be a CR dual system of metric \(f\)-structures. Let \(\mathscr H\) be a \(V\)-Hilbert \(U\)-module satisfying
\begin{equation}\label{eq:Hilb_proj_hp}
\sfd_{\mathscr H}(v,0)^2\leq\sfd_Z(|v|^2,0),\quad\text{ for every }v\in\mathscr H.
\end{equation}
Let \(C\neq\varnothing\) be a closed, convex subset of \(\mathscr H\) such that \(\mathscr G(C)=C\). Let \(v\in\mathscr H\) be fixed. We define
\[
|v-C|\coloneqq\inf\big\{|v-w|\;\big|\;w\in C\big\}\in V^+,\qquad
\sfd_{\mathscr H}(v,C)\coloneqq\inf\big\{\sfd_{\mathscr H}(v,w)\;\big|\;w\in C\big\}\in\R^+.
\]
Then it holds that
\begin{equation}\label{eq:Hilb_proj_cl1}
\sfd_V(|v-C|,0)=\sfd_{\mathscr H}(v,C).
\end{equation}
Moreover, there exists a unique \(P_C(v)\in C\), the \textbf{orthogonal projection} of \(v\) onto \(C\), such that
\[
|v-C|=|v-P_C(v)|.
\]
\end{theorem}
\begin{proof}
Applying Lemma \ref{lem:supporting_element}, we can find an element \(h\in V^+\) such that \(h\leq 1\) and \(\nchi_{\{h=0\}}V=\{0\}\).
By the CR assumption, there exists a sequence \((\tilde w_k)_{k\in\N}\subset C\) such that \(|v-C|=\inf_{k\in\N}|v-\tilde w_k|\).
One can readily check that for any \(n\in\N\) there exists a partition \((u^n_k)_{k\in\N}\) of \(\nchi_{\{h\neq 0\}}\) such that
\[
u^n_k|v-C|\leq u^n_k|v-\tilde w_k|\leq u^n_k\bigg(|v-C|+\frac{1}{n}h\bigg),\quad\text{ for every }k\in\N.
\]
In particular, it holds \((u^n_k,\tilde w_k)_{k\in\N}\in\Adm(\mathscr H)\), thus it makes sense to set
\(\omega_n\coloneqq\sum_{k\in\N}u^n_k\cdot\tilde w_k\in\mathscr H\). Given that \(C\) is closed under the glueing operation,
we have that \((w_n)_{n\in\N}\subset C\). Notice also that
\begin{equation}\label{eq:Hilb_proj_aux1}
|v-C|\leq|v-w_n|\leq|v-C|+\frac{1}{n}h,\quad\text{ for every }n\in\N.
\end{equation}
Since \(\lim_n\sfd_V(n^{-1}h,0)=0\), we deduce from \eqref{eq:Hilb_proj_aux1} that \(\lim_n\sfd_V\big(|v-C|,|v-w_n|\big)=0\) and thus
\[
\sfd_{\mathscr H}(v,C)\leq\lim_{n\to\infty}\sfd_{\mathscr H}(v,w_n)=\lim_{n\to\infty}\sfd_V(|v-w_n|,0)=\sfd_V(|v-C|,0).
\]
On the other hand, we have \(|v-C|\leq|v-w|\), thus \(\sfd_V(|v-C|,0)\leq\sfd_{\mathscr H}(v,w)\), for every \(w\in C\).
By taking the infimum over \(w\in C\), we get \(\sfd_V(|v-C|,0)\leq\sfd_{\mathscr H}(v,C)\). All in all, \eqref{eq:Hilb_proj_cl1} is proved.

The Hilbertianity of \(\mathscr H\) and the convexity of \(C\) ensure that for any \(n,m\in\N\) it holds that
\begin{equation}\label{eq:Hilb_proj_aux2}\begin{split}
|w_n-w_m|^2&\overset{\phantom{\eqref{eq:Hilb_proj_aux1}}}=2|v-w_n|^2+2|v-w_m|^2-4\bigg|v-\frac{w_n+w_m}{2}\bigg|^2\\
&\overset{\phantom{\eqref{eq:Hilb_proj_aux1}}}\leq 2|v-w_n|^2+2|v-w_m|^2-4|v-C|^2\\
&\overset{\eqref{eq:Hilb_proj_aux1}}\leq 2\bigg(\frac{1}{n^2}+\frac{1}{m^2}\bigg)h^2+4\bigg(\frac{1}{n}+\frac{1}{m}\bigg)|v-C|h,
\end{split}\end{equation}
whence it follows that \(\sfd_V(|w_n-w_m|,0)\leq\sqrt{\sfd_Z(|w_n-w_m|^2,0)}\to 0\) as \(n,m\to\infty\). Hence,
\((w_n)_{n\in\N}\) is a Cauchy sequence in \(\mathscr H\), thus \(\lim_n\sfd_{\mathscr H}(w_n,\bar w)=0\) holds for some \(\bar w\in C\).
In particular, we have that \(\sfd_V(|v-C|,|v-\bar w|)=\lim_n\sfd_V(|v-C|,|v-w_n|)=0\), so that \(|v-C|=|v-\bar w|\). To prove that
\(\bar w\) is the unique element of \(C\) with this property, fix any \(\tilde w\in C\) with \(|v-\tilde w|=|v-C|\). Then
\[
0\leq|\tilde w-\bar w|^2=2|v-\tilde w|^2+2|v-\bar w|^2-4\bigg|v-\frac{\tilde w+\bar w}{2}\bigg|^2\leq 2|v-C|^2+2|v-C|^2-4|v-C|^2=0,
\]
which forces the identity \(\tilde w=\bar w\). All in all, the statement is finally achieved.
\end{proof}
The choice of the terminology `orthogonal projection' is justified by the following result:
\begin{proposition}\label{prop:orth_proj}
Let \((\mathcal U,U,V,V,Z)\) be a CR dual system of metric \(f\)-structures.
Let \(\mathscr H\) be a \(V\)-Hilbert \(U\)-module satisfying \eqref{eq:Hilb_proj_hp}.
Let \(\mathscr N\) be a \(V\)-Hilbert \(U\)-submodule of \(\mathscr H\). Then:
\begin{itemize}
\item[\(\rm i)\)] \(v-P_{\mathscr N}(v)\in\mathscr N^\perp\) for every \(v\in\mathscr H\).
\item[\(\rm ii)\)] It holds that \(\mathscr N\oplus\mathscr N^\perp\).
\item[\(\rm iii)\)] The map \(P_{\mathscr N}\colon\mathscr H\to\mathscr N\) belongs to \(\textsc{Hom}(\mathscr H,\mathscr N)\).
\item[\(\rm iv)\)] \(v=P_{\mathscr N}(v)+P_{\mathscr N^\perp}(v)\) and \(|v|^2=|P_{\mathscr N}(v)|^2+|P_{\mathscr N^\perp}(v)|^2\)
for every \(v\in\mathscr H\).
\end{itemize}
\end{proposition}
\begin{proof}
Fix \(w\in\mathscr N\). Denote \(\bar w\coloneqq P_{\mathscr N}(v)\). Pick a partition \((u_n)_{n\in\N}\) of \({\bf 1}_U\) such that
\(u_n(v-\bar w)\cdot w\in U\) for every \(n\in\N\). Since \(\bar w-u\cdot w\in\mathscr N\) for every \(u\in U\), we have that
\(|v-\bar w|\leq|v-\bar w+u\cdot w|\), thus
\begin{equation}\label{eq:orth_proj_aux1}
|v-\bar w|^2\leq|v-\bar w+u\cdot w|^2=|v-\bar w|^2+u|w|^2-2u(v-\bar w)\cdot w.
\end{equation}
Fix any \(n\in\N\) and \(\varepsilon\in\R\) with \(\varepsilon>0\). Plugging \(u=-\varepsilon u_n(v-\bar w)\cdot w\) into
\eqref{eq:orth_proj_aux1}, multiplying both sides by \(\varepsilon^{-1}u_n\), and rearranging the various terms, we obtain that
\[
2u_n|(v-\bar w)\cdot w|^2\leq\varepsilon u_n|(v-\bar w)\cdot w|^2|w|^2.
\]
Letting \(\varepsilon\searrow 0\), we get \(u_n|(v-\bar w)\cdot w|=0\) for every \(n\in\N\),
thus \(v-\bar w\in\mathscr N^\perp\). Then i) is proved.

To prove ii), we aim to show that \(\mathscr N+\mathscr N^\perp=\mathscr H\) and \(\mathscr N\cap\mathscr N^\perp=\{0\}\).
For the former, just observe that any \(v\in\mathscr H\) can be written as \(\big(v-P_{\mathscr N}(v)\big)+P_{\mathscr N}(v)\),
where \(v-P_{\mathscr N}(v)\in\mathscr N^\perp\) by i) and \(P_{\mathscr N}(v)\in\mathscr N\). For the latter, note that
if \(v\in\mathscr N\cap\mathscr N^\perp\), then \(|v|^2=v\cdot v=0\) and thus \(v=0\).

Let us now pass to the verification of iii). Given any \(v,w\in\mathscr H\), we deduce from i) that
\[
\underbrace{\big(v+w-P_{\mathscr N}(v+w)\big)}_{\in\mathscr N^\perp}+\underbrace{P_{\mathscr N}(v+w)}_{\in\mathscr N}=v+w=
\underbrace{\big(v-P_{\mathscr N}(v)+w-P_{\mathscr N}(w)\big)}_{\in\mathscr N^\perp}+\underbrace{P_{\mathscr N}(v)+P_{\mathscr N}(w)}_{\in\mathscr N},
\]
thus accordingly ii) implies that \(P_{\mathscr N}(v+w)=P_{\mathscr N}(v)+P_{\mathscr N}(w)\).
Similarly, for any element \(u\in U\) we have that \((u\cdot v-P_{\mathscr N}(u\cdot v))+P_{\mathscr N}(u\cdot v)=u\cdot v
=u\cdot(v-P_{\mathscr N}(v))+u\cdot P_{\mathscr N}(v)\), which forces the identity \(P_{\mathscr N}(u\cdot v)=u\cdot P_{\mathscr N}(v)\).
All in all, we showed that \(P_{\mathscr N}\) is a \(U\)-linear map. We also have that
\(|v|^2=|v-P_{\mathscr N}(v)|^2+|P_{\mathscr N}(v)|^2\geq|P_{\mathscr N}(v)|^2\) for every \(v\in\mathscr H\),
thus \(P_{\mathscr N}\in\textsc{Hom}(\mathscr H,\mathscr N)\).

Finally, we prove iv). Given any \(w\in\mathscr N^\perp\), we have that \(|w-v|^2=|w|^2+|v|^2\) for every \(v\in\mathscr N\)
and thus \(|w-\mathscr N|^2=|w|^2\), which implies that \(P_{\mathscr N}(w)=0\). Using also iii), we deduce that
\[
\big|v-\big(v-P_{\mathscr N}(v)\big)\big|^2=\big|P_{\mathscr N}(v)\big|^2=\big|P_{\mathscr N}(v)-P_{\mathscr N}(w)\big|^2
=\big|P_{\mathscr N}(v-w)\big|^2\leq|v-w|^2
\]
for every \(v\in\mathscr H\). Hence, \(\big|v-\big(v-P_{\mathscr N}(v)\big)\big|=|v-\mathscr N^\perp|\),
which implies \(P_{\mathscr N^\perp}(v)=v-P_{\mathscr N}(v)\). In particular, one has
\(|v|^2=\big|P_{\mathscr N}(v)+P_{\mathscr N^\perp}(v)\big|^2=|P_{\mathscr N}(v)|^2+|P_{\mathscr N^\perp}(v)|^2\), so also iv) is proved.
\end{proof}
\begin{theorem}[Riesz representation theorem for Hilbert modules]\label{thm:RRT}
Let \((\mathcal U,U,V,V,Z)\) be a CR dual system of metric \(f\)-structures. Let \(\mathscr H\) be a \(V\)-Hilbert \(U\)-module
satisfying \eqref{eq:Hilb_proj_hp}. We define the operator \({\rm R}_{\mathscr H}\colon\mathscr H\to\mathscr H^*\) as
\[
\langle{\rm R}_{\mathscr H}(w),v\rangle\coloneqq v\cdot w,\quad\text{ for every }v,w\in\mathscr H.
\]
Then \({\rm R}_{\mathscr H}\) is an isomorphism of \(V\)-Banach \(U\)-modules, \(\mathscr H^*\) is a \(V\)-Hilbert \(U\)-module, and
\begin{equation}\label{eq:Riesz_preserv_ptwse_norm}
{\rm R}_{\mathscr H}(v)\cdot{\rm R}_{\mathscr H}(w)=v\cdot w,\quad\text{ for every }v,w\in\mathscr H.
\end{equation}
\end{theorem}
\begin{proof}
The properties of the pointwise scalar product and the Cauchy--Schwartz inequality ensure that \({\rm R}_{\mathscr H}(w)\in\mathscr H^*\)
and \(|{\rm R}_{\mathscr H}(w)|\leq|w|\) for all \(w\in\mathscr H\). Then \({\rm R}_{\mathscr H}\in\textsc{Hom}(\mathscr H,\mathscr H^*)\)
and \(|{\rm R}_{\mathscr H}|\leq 1\) (recall Example \ref{ex:UVUV}). To conclude, it remains to prove that \({\rm R}_{\mathscr H}\) is
surjective and that it holds that \(|{\rm R}_{\mathscr H}(w)|\geq|w|\) for every \(v\in\mathscr H\). To this aim, fix any
\(\eta\in\mathscr H^*\setminus\{0\}\). We know that \({\rm ker}(\eta)\) is a \(V\)-Banach \(U\)-submodule of \(\mathscr H\) with
\({\rm ker}(\eta)\neq\mathscr H\), so that there exists \(\tilde w\in{\rm ker}(\eta)^\perp\setminus\{0\}\). We can find a partition
\((u_n)_{n\in\N}\) of \(\nchi_{\{\tilde w\neq 0\}}\) and a sequence \((a_n)_{n\in\N}\subset U^+\) such that
\(u_n|\tilde w|,u_n\langle\eta,\tilde w\rangle\in U\) and \(u_n a_n|\tilde w|=u_n\) for every \(n\in\N\).
Then we define \(w_n\coloneqq(u_n\langle\eta,\tilde w\rangle a_n^2)\cdot\tilde w\in\mathscr H\) for every \(n\in\N\).
Notice that \(|w_n|=u_n a_n|\langle\eta,\tilde w\rangle|\) and \(\langle\eta,w_n\rangle=|w_n|^2\).
In particular, \(|w_n|\leq u_n|\eta|\) for all \(n\in\N\), thus it makes sense to define \(w\coloneqq\sum_{n\in\N}u_n\cdot w_n\in\mathscr H\).
It holds that \(|w|\leq|\eta|\) and \(\langle\eta,w\rangle=|w|^2\). Pick a partition \((\tilde u_k)_{k\in\N}\) of \(\nchi_{\{w\neq 0\}}\)
and \((b_k)_{k\in\N}\subset U^+\) such that \(\tilde u_k|w|^2\in U\) and \(\tilde u_k b_k|w|^2=\tilde u_k\) for every \(k\in\N\).
Given any element \(v\in\mathscr H\), we have that for every \(k\in\N\) it holds that
\begin{equation}\label{eq:Riesz_Hilb_aux}
\tilde u_k\cdot(v\cdot w)=\big(\tilde u_k\cdot v-(\tilde u_k\langle\eta,v\rangle b_k)\cdot w\big)\cdot w+\tilde u_k\langle\eta,v\rangle b_k|w|^2.
\end{equation}
Observe that \(\tilde u_k\cdot v-(\tilde u_k\langle\eta,v\rangle b_k)\cdot w\in{\rm ker}(\eta)\), as a consequence of the following computation:
\[
\big\langle\eta,\tilde u_k\cdot v-(\tilde u_k\langle\eta,v\rangle b_k)\cdot w\big\rangle
=\tilde u_k\langle\eta,v\rangle-\tilde u_k\langle\eta,v\rangle b_k\langle\eta,w\rangle
=\tilde u_k\langle\eta,v\rangle-\tilde u_k\langle\eta,v\rangle b_k|w|^2=0.
\]
It holds \(w\in{\rm ker}(\eta)^\perp\), whence it follows that \(\big(\tilde u_k\cdot v-(\tilde u_k\langle\eta,v\rangle b_k)\cdot w\big)\cdot w=0\),
thus \eqref{eq:Riesz_Hilb_aux} yields
\[
\tilde u_k\cdot(v\cdot w)=\tilde u_k\langle\eta,v\rangle b_k|w|^2=\tilde u_k\langle\eta,v\rangle.
\]
Thanks to the arbitrariness of \(k\in\N\), we finally conclude that \(\langle\eta,v\rangle=v\cdot w\) for every \(v\in\mathscr H\),
which means that \(\eta={\rm R}_{\mathscr H}(w)\) and \(|{\rm R}_{\mathscr H}(w)|=|\eta|\geq|w|\). This completes the proof.
\end{proof}
\begin{proposition}\label{prop:Hilb_are_refl}
Let \((\mathcal U,U,V,V,Z)\) be a CR dual system of metric \(f\)-structures. Let \(\mathscr H\) be
a \(V\)-Hilbert \(U\)-module satisfying \eqref{eq:Hilb_proj_hp}. Then it holds that \(\mathscr H\) is reflexive.
\end{proposition}
\begin{proof}
Given any \(v\in\mathscr H\) and \(\omega\in\mathscr H^*\), we have that
\[\begin{split}
\big\langle{\rm R}_{\mathscr H^*}\big({\rm R}_{\mathscr H}(v)\big),\omega\big\rangle
&=\omega\cdot{\rm R}_{\mathscr H}(v)={\rm R}_{\mathscr H}\big({\rm R}_{\mathscr H}^{-1}(\omega)\big)\cdot{\rm R}_{\mathscr H}(v)
\overset{\eqref{eq:Riesz_preserv_ptwse_norm}}={\rm R}_{\mathscr H}^{-1}(\omega)\cdot v\\
&=\big\langle{\rm R}_{\mathscr H}\big({\rm R}_{\mathscr H}^{-1}(\omega)\big),v\big\rangle
=\langle\omega,v\rangle=\langle{\rm J}_{\mathscr H}(v),\omega\rangle.
\end{split}\]
This shows that \({\rm J}_{\mathscr H}={\rm R}_{\mathscr H^*}\circ{\rm R}_{\mathscr H}\). Since both \({\rm R}_{\mathscr H}\)
and \({\rm R}_{\mathscr H^*}\) are surjective by Theorem \ref{thm:RRT}, we conclude that \({\rm J}_{\mathscr H}\) is surjective
and thus \(\mathscr H\) is reflexive, yielding the sought conclusion.
\end{proof}
\begin{proposition}
Let \((\mathcal U_1,U_1,V_1,V_1,Z_1)\), \((\mathcal U_2,U_2,V_2,V_2,Z_2)\) be CR dual systems of metric \(f\)-structures.
Let \(\phi\colon\mathcal U_1\to\mathcal U_2\) be a homomorphism of dual systems. Let \(\mathscr H\) be a \(V_1\)-Hilbert \(U_1\)-module satisfying
\eqref{eq:Hilb_proj_hp}. Then \({\rm I}_\phi\colon\phi_*\mathscr H^*\to(\phi_*\mathscr H)^*\) is an isomorphism of \(V_2\)-Banach \(U_2\)-modules.
\end{proposition}
\begin{proof}
By Theorem \ref{thm:embed_pshfrwd_dual}, it suffices to check that \({\rm I}_\phi\colon\phi_*\mathscr H^*\to(\phi_*\mathscr H)^*\) is invertible.
Recall from Remark \ref{rmk:pshfrwd_Hilbert} that \(\phi_*\mathscr H\) is a \(V_2\)-Hilbert \(U_2\)-module and \(\phi_*v\cdot\phi_*w=\phi(v\cdot w)\)
for all \(v,w\in\mathscr H\). Then
\[
\big\langle({\rm I}_\phi\circ\phi_*\circ{\rm R}_{\mathscr H})(v),\phi_*w\big\rangle
=\phi(\langle{\rm R}_{\mathscr H}(v),w\rangle)=\phi(v\cdot w)=\phi_*v\cdot\phi_*w
=\big\langle({\rm R}_{\phi_*\mathscr H}\circ\phi_*)(v),\phi_*w\big\rangle
\]
for every \(v,w\in\mathscr H\). Moreover, \(\phi_*\circ{\rm R}_{\mathscr H}\colon\mathscr H\to\phi_*\mathscr H^*\) is linear and satisfies
\(\big|\phi_*({\rm R}_{\mathscr H}(v))\big|=|v|\) for every \(v\in\mathscr H\), thus Proposition \ref{prop:univ_prop} gives an element
\(T\in\textsc{Hom}(\phi_*\mathscr H,\phi_*\mathscr H^*)\) such that
\[\begin{tikzcd}
\mathscr H \arrow[r,"{\rm R}_{\mathscr H}"] \arrow[rd,swap,"\phi_*"] &
\mathscr H^* \arrow[r,"\phi_*"] & \phi_*\mathscr H^* \arrow[d,"{\rm I}_\phi"] \\
& \phi_*\mathscr H \arrow[ru,"T"] \arrow[r,swap,"{\rm R}_{\phi_*\mathscr H}"] & (\phi_*\mathscr H)^*
\end{tikzcd}\]
is a commutative diagram. Hence, \({\rm I}_\phi\) is invertible and \({\rm I}_\phi^{-1}=T\circ{\rm R}_{\phi_*\mathscr H}^{-1}\),
concluding the proof.
\end{proof}
\subsection{Dimensional decomposition}\label{ss:dim_decomp}
Given a commutative ring \(R\) and a non-empty subset \(S\) of an \(R\)-module \(M\), we denote by
\(\langle S\rangle_R\) the \(R\)-submodule of \(M\) generated (in the algebraic sense) by \(S\). Namely, we define
\[
\langle S\rangle_R=\bigg\{\sum_{i=1}^n r_i\cdot v_i\;\bigg|\;n\in\N,\,(r_i)_{i=1}^n\subset R,\,(v_i)_{i=1}^n\subset S\bigg\}.
\]
\begin{definition}[Independence, generators, local basis]
Let \((\mathcal U,U,V)\) be a metric \(f\)-structure and \(\mathscr M\) a \(V\)-Banach \(U\)-module.
Let \(v_1,\ldots,v_n\in\mathscr M\) and \(u\in\Idem(U)\) be given. Then we say that:
\begin{itemize}
\item[\(\rm i)\)] \(v_1,\ldots,v_n\) are \textbf{independent on \(u\)} if for any \(u_1,\ldots,u_n\in U\) it holds that
\[
\sum_{i=1}^n(uu_i)\cdot v_i=0\quad\Longleftrightarrow\quad uu_i=0,\text{ for every }i=1,\ldots,n.
\]
\item[\(\rm ii)\)] \(v_1,\ldots,v_n\) \textbf{generate \(\mathscr M\) on \(u\)} provided \(\mathscr G(\langle u\cdot S\rangle_U)=u\cdot\mathscr M\), where \(S\coloneqq\{v_1,\ldots,v_n\}\).
\item[\(\rm iii)\)] \(v_1,\ldots,v_n\) form a \textbf{local basis of \(\mathscr M\) on \(u\)}
provided they are independent on \(u\) and they generate \(\mathscr M\) on \(u\).
\end{itemize}
For brevity, in the case where \(u=\boldsymbol{1}_U\) we do not specify `on \(\boldsymbol{1}_U\)' in the above terminology.
\end{definition}

In order to provide a well-defined notion of local dimension, we first need to show that two local bases on the same idempotent element must have the same cardinality:
\begin{lemma}\label{lem:dd_well-posed}
Let \((\mathcal U,U,V)\) be a CR metric \(f\)-structure and let \(\mathscr M\) be a \(V\)-Banach \(U\)-module. Let \(v_1,\ldots,v_n\in\mathscr M\)
and \(w_1,\ldots,w_m\in\mathscr M\) be local bases of \(\mathscr M\) on \(u\in\Idem(U)\). Then \(n=m\).
\end{lemma}
\begin{proof}
Observe that it suffices to check that if \(v_1,\ldots,v_n\) generate \(\mathscr M\) on \(u\) and \(w_1,\ldots,w_m\) are independent on \(u\), then \(n\geq m\). Moreover,
thanks to a finite induction argument, it is enough to show that if \(k\leq m\) and \(w_1,\ldots,w_{k-1},v_k,\ldots,v_n\) generate \(\mathscr M\) on some \(u_0\in\Idem(U)\setminus\{0\}\)
with \(u_0\leq u\), then there exists \(u_1\in\Idem(U)\setminus\{0\}\) with \(u_1\leq u_0\) such that \(w_1,\ldots,w_k,v_{k+1},\ldots,v_n\) generate \(\mathscr M\) on \(u_1\)
(up to reordering \(v_k,\ldots,v_n\)). First of all, we can find elements \(\tilde u\in\Idem(U)\setminus\{0\}\) with \(\tilde u\leq u_0\) and \(\tilde u_1,\ldots,\tilde u_n\in U\) such that
\begin{equation}\label{eq:dd_well-posed}
\tilde u\cdot w_k=\sum_{i=1}^{k-1}\tilde u_i\cdot w_i+\sum_{i=k}^n\tilde u_i\cdot v_i.
\end{equation}
Since \(w_1,\ldots,w_k\) are independent on \(\tilde u\), it cannot happen that \(\tilde u_k=\ldots=\tilde u_n=0\). Hence, up to reordering \(v_k,\ldots,v_n\), we can assume
that \(z\coloneqq\nchi_{\{\tilde u_k\neq 0\}}\tilde u\neq 0\). Now take a partition \((z_j)_{j\in\N}\) of \(z\) and elements \((\tilde z_j)_{j\in\N}\subset U\) such that
\(z_j(\tilde u_k\tilde z_j-1)=0\) for every \(j\in\N\). There exists \(j_0\in\N\) such that \(z_{j_0}\neq 0\), so that multiplying both sides of \eqref{eq:dd_well-posed}
by \(z_{j_0}\tilde z_{j_0}\) we obtain that
\[
z_{j_0}\cdot v_k=(z_{j_0}\tilde z_{j_0}\tilde u_k)\cdot v_k=(z_{j_0}\tilde z_{j_0})\cdot w_k-\sum_{i=1}^{k-1}(z_{j_0}\tilde z_{j_0}\tilde u_i)\cdot w_i-\sum_{i=k+1}^n(z_{j_0}\tilde z_{j_0}\tilde u_i)\cdot v_i.
\]
This implies that, letting \(u_1\coloneqq z_{j_0}\), it holds that \(w_1,\ldots,w_k,v_{k+1},\ldots,v_n\) generate \(\mathscr M\) on \(u_1\).
\end{proof}

In view of Lemma \ref{lem:dd_well-posed}, the following definition is thus well-posed:
\begin{definition}[Local dimension]
Let \((\mathcal U,U,V)\) be a CR metric \(f\)-structure and let \(\mathscr M\) be a \(V\)-Banach \(U\)-module. Then we say that \(\mathscr M\)
\textbf{has local dimension \(n\in\N\) on \(u\in\Idem(U)\)} if there exists a local basis \(v_1,\ldots,v_n\) of \(\mathscr M\) on \(u\).
\end{definition}

Finally, we can show that each Banach module admits a (unique) \emph{dimensional decomposition}:
\begin{theorem}[Dimensional decomposition]
Let \((\mathcal U,U,V)\) be a CR metric \(f\)-structure and let \(\mathscr M\) be a \(V\)-Banach \(U\)-module. Then there exists a unique partition
\((D_n)_{n\in\N\cup\{\infty\}}\) of \({\bf 1}_U\) such that:
\begin{itemize}
\item[\(\rm i)\)] \(\mathscr M\) has local dimension \(n\) on \(D_n\), for every \(n\in\N\).
\item[\(\rm ii)\)] Given any \(n\in\N\) and \(u\in\Idem(U)\setminus\{0\}\) such that \(u\leq D_\infty\), it holds that \(\mathscr M\) does not have local dimension \(n\) on \(u\).
\end{itemize}
We say that \((D_n)_{n\in\N\cup\{\infty\}}\subset\Idem(U)\) is the \textbf{dimensional decomposition} of \(\mathscr M\).
\end{theorem}
\begin{proof}
Thanks to the countable representability assumption, it makes sense to define
\[
D_n\coloneqq\sup\big\{u\in\Idem(U)\;\big|\;\mathscr M\text{ has local dimension }n\text{ on }u\big\}\in\Idem(U),\quad\text{ for every }n\in\N,
\]
as well as \(D_\infty\coloneqq{\bf 1}_U-\sup_{n\in\N}D_n\in\Idem(U)\). We know from Lemma \ref{lem:dd_well-posed} that \(D_n\) and \(D_m\) are disjoint whenever
\(n,m\in\N\) and \(n\neq m\), thus it follows that \((D_n)_{n\in\N\cup\{\infty\}}\) is a partition of \({\bf 1}_U\). In order to conclude, it suffices to check
that \(\mathscr M\) has local dimension \(n\) on \(D_n\) for every \(n\in\N\). Using again the CR assumption, we can construct a partition \((u^j)_{j\in\N}\) of
\(D_n\) such that \(\mathscr M\) has local dimension \(n\) on each \(u^j\). For any \(j\in\N\), take a local basis \(v^j_1,\ldots,v^j_n\in u_j\cdot\mathscr M\)
of \(\mathscr M\) on \(u^j\). Thanks to Lemma \ref{lem:tech_localisation} and up to further refining the partition \((u^j)_{j\in\N}\), it is not restrictive to
require also that \(|v^j_i|\in U\) for every \(i=1,\ldots,n\) and \(j\in\N\). Fix an element \(h\) as in Lemma \ref{lem:supporting_element}. The independence of \(v^j_1,\ldots,v^j_n\) on \(u^j\) ensures that
\(\nchi_{\{|v^j_i|>0\}}=u^j\) for every \(i=1,\ldots,n\), thus we can find a partition \((u^j_{i,k})_{k\in\N}\) of \(u^j\) and \((z^j_{i,k})_{k\in\N}\subset U^+\) such that
\[
u^j_{i,k}(z^j_{i,k}|v^j_i|-1)=0,\quad\text{ for every }k\in\N.
\]
Define \(v^j_{i,k}\coloneqq(u^j_{i,k}z^j_{i,k}h)\cdot v^j_i\in\mathscr M\) for every \(k\in\N\). Since \(|v^j_{i,k}|=u^j_{i,k}z^j_{i,k}h|v^j_i|=u^j_{i,k}h\) and
\((u^j_{i,k})_{j,k}\) is a partition of \(D_n\), for any \(i=1,\ldots,n\) it holds that \((u^j_{i,k},v^j_{i,k})_{j,k}\in{\rm Adm}(\mathscr M)\) and thus
\[
\exists\,v_i\coloneqq\sum_{j,k\in\N}u^j_{i,k}\cdot v^j_{i,k}\in\mathscr M.
\]
It is then easy to check that \(v_1,\ldots,v_n\) is local basis of \(\mathscr M\) on \(D_n\), whence the statement follows.
\end{proof}
\section{Applications to functional normed modules}\label{s:applic_fct}
In Section \ref{ss:function_spaces} we introduce the various spaces of functions that are typically used in metric measure geometry.
In Section \ref{ss:example_norm_mod} we explain the relation between the `functional' Banach modules and the axiomatic Banach modules we introduced in this paper.
In Section \ref{ss:reconstruction} we obtain a side result, which states that every localisable \(f\)-algebra can be in fact realised as a space of functions.
\subsection{Function spaces}\label{ss:function_spaces}
A \textbf{measurable space} \((\X,\Sigma)\) is a set \(\X\neq\varnothing\) together with a \(\sigma\)-algebra \(\Sigma\).
Whenever \(\X\) is a topological space, we implicitly assume that \(\Sigma\) is the Borel \(\sigma\)-algebra of \(\X\). We define
\[
\mathcal L^0(\Sigma)\coloneqq\big\{f\colon\X\to\R\;\big|\;f\text{ is measurable}\big\}.
\]
Then \(\mathcal L^0(\Sigma)\) is an \(f\)-algebra if endowed with the following operations: for any \(f,g\in\mathcal L^0(\Sigma)\), we set
\[\begin{split}
(f+g)(x)\coloneqq f(x)+g(x)&,\quad\text{ for every }x\in\X,\\
(\lambda f)(x)\coloneqq\lambda f(x)&,\quad\text{ for every }\lambda\in\R\text{ and }x\in\X,\\
f\leq g&,\quad\text{ if and only if }f(x)\leq g(x)\text{ for every }x\in\X,\\
(fg)(x)\coloneqq f(x)g(x)&,\quad\text{ for every }x\in\X.
\end{split}\]
The space \(\mathcal L^\infty(\Sigma)\) of all bounded measurable real-valued functions on \((\X,\Sigma)\), which is given by
\[
\mathcal L^\infty(\Sigma)\coloneqq\Big\{f\in\mathcal L^0(\Sigma)\;\Big|\;\|f\|_{\mathcal L^\infty(\Sigma)}\coloneqq\sup_{x\in\X}|f|(x)<+\infty\Big\},
\]
is an \(f\)-subalgebra of \(\mathcal L^0(\Sigma)\). Moreover, \((\mathcal L^\infty(\Sigma),\|\cdot\|_{\mathcal L^\infty(\Sigma)})\)
is a Banach space. Given any \(E\in\Sigma\), we denote by \(\1_E\in\mathcal L^\infty(\Sigma)\) the \textbf{characteristic function}
of \(E\), which is defined as
\[
\1_E(x)\coloneqq\left\{\begin{array}{ll}
1,\\
0,
\end{array}\quad\begin{array}{ll}
\text{ if }x\in E,\\
\text{ if }x\in\X\setminus E.
\end{array}\right.
\]
The space \({\sf Sf}(\Sigma)\subset\mathcal L^\infty(\Sigma)\) of all \textbf{simple functions}
on \((\X,\Sigma)\) is then defined in the following way:
\[
{\sf Sf}(\Sigma)\coloneqq\bigg\{\sum_{i=1}^n\lambda_i\1_{E_i}\;\bigg|\;n\in\N,
\,(\lambda_i)_{i=1}^n\subset\R,\,(E_i)_{i=1}^n\subset\Sigma\text{ partition of }\X\bigg\}.
\]
\begin{remark}[Density of simple functions]\label{rmk:density_simple_fcs}{\rm
Given any \(f\in\mathcal L^0(\Sigma)^+\), there exists \((f_n)_{n\in\N}\subset{\sf Sf}(\Sigma)\)
such that \(0\leq f_n\leq f_{n+1}\) for every \(n\in\N\) and \(f=\sup_{n\in\N}f_n\). If in addition \(f\in\mathcal L^\infty(\Sigma)\),
then the sequence \((f_n)_{n\in\N}\) can be also chosen so that \(\|f-f_n\|_{\mathcal L^\infty(\Sigma)}\to 0\),
so that \({\sf Sf}(\Sigma)\) is dense in \(\mathcal L^\infty(\Sigma)\). For example, the functions
\(f_n\coloneqq\sum_{i=0}^{n2^n-1}i2^{-n}\1_{\{i2^{-n}\leq f<(i+1)2^{-n}\}}\in\mathcal L^\infty(\Sigma)^+\) do the job.
\fr}\end{remark}
It is immediate to verify that
\begin{equation}\label{eq:Idem_fcts}\begin{split}
\Idem\big(\mathcal L^0(\Sigma)\big)&=\Idem\big(\mathcal L^\infty(\Sigma)\big)=\{\1_E\;|\;E\in\Sigma\},\\
\mathcal S\big(\mathcal L^0(\Sigma)\big)&=\mathcal S\big(\mathcal L^\infty(\Sigma)\big)={\sf Sf}(\Sigma).
\end{split}\end{equation}
It then follows from \eqref{eq:Idem_fcts} and Remark \ref{rmk:density_simple_fcs} that \(\mathcal L^0(\Sigma)\)
and \(\mathcal L^\infty(\Sigma)\) are localisable \(f\)-algebras.
\subsubsection{Enhanced measurable spaces}
By an \textbf{enhanced measurable space} we mean a triple \((\X,\Sigma,\mathcal N)\), where \((\X,\Sigma)\) is a measurable space and
\(\mathcal N\subset\Sigma\) is a \textbf{\(\sigma\)-ideal}, meaning that it verifies the following conditions:
\begin{itemize}
\item[\(\rm i)\)] \(\varnothing\in\mathcal N\).
\item[\(\rm ii)\)] If \(N\in\mathcal N\) and \(N'\in\Sigma\) satisfy \(N'\subset N\), then \(N'\in\mathcal N\).
\item[\(\rm iii)\)] \(\bigcup_{n\in\N}N_n\in\mathcal N\) for every countable family \((N_n)_{n\in\N}\subset\mathcal N\).
\end{itemize}
The \(\sigma\)-ideal \(\mathcal N\) induces an equivalence relation on \(\mathcal L^0(\Sigma)\): given any \(f,g\in\mathcal L^0(\Sigma)\),
we declare that
\[
f\sim_{\mathcal N}g\qquad\Longleftrightarrow\qquad\{f\neq g\}\coloneqq\big\{x\in\X\;\big|\;f(x)\neq g(x)\big\}\in\mathcal N.
\]
When \(f\sim_{\mathcal N}g\), we say that \(f=g\) holds \textbf{\(\mathcal N\)-almost everywhere}, or \textbf{\(\mathcal N\)-a.e.}\ for short. We define
\[
L^0(\mathcal N)\coloneqq\mathcal L^0(\Sigma)/\sim_{\mathcal N},\qquad L^\infty(\mathcal N)\coloneqq\mathcal L^\infty(\Sigma)/\sim_{\mathcal N}.
\]
We denote by \([f]_{\mathcal N}\in L^0(\mathcal N)\) the equivalence class of \(f\in\mathcal L^0(\Sigma)\).
For brevity, we set
\[
\1^{\mathcal N}_E\coloneqq[\1_E]_{\mathcal N},\quad\text{ for every }E\in\Sigma.
\]
Passing to the quotient, \(L^0(\mathcal N)\) and \(L^\infty(\mathcal N)\) inherit an \(f\)-algebra structure
from \(\mathcal L^0(\Sigma)\) and \(\mathcal L^\infty(\Sigma)\), respectively. Moreover, \(L^\infty(\mathcal N)\)
becomes a Banach space if endowed with the quotient norm
\[
\|f\|_{L^\infty(\mathcal N)}\coloneqq\inf_{\bar f\in[f]_{\mathcal N}}\|\bar f\|_{\mathcal L^\infty(\Sigma)},
\quad\text{ for every }f\in L^\infty(\mathcal N).
\]
The map \(\mathcal L^\infty(\Sigma)\ni f\mapsto[f]_{\mathcal N}\in L^\infty(\mathcal N)\) is linear \(1\)-Lipschitz.
Note that \(\{\varnothing\}\) is a \(\sigma\)-ideal of \(\Sigma\) and
\[
L^0(\{\varnothing\})=\mathcal L^0(\Sigma),\qquad
\big(L^\infty(\{\varnothing\}),\|\cdot\|_{L^\infty(\{\varnothing\})}\big)=\big(\mathcal L^\infty(\Sigma),\|\cdot\|_{\mathcal L^\infty(\Sigma)}\big).
\]
\begin{example}[of enhanced measurable space]{\rm
Let \((\X,\Sigma)\) be a measurable space. Consider the restriction \(\mu\colon\Sigma\to[0,+\infty]\)
of an outer measure on \(\X\). Then the set \(\mathcal N_\mu\) of \textbf{\(\mu\)-null sets}, given by
\[
\mathcal N_\mu\coloneqq\big\{N\in\Sigma\;\big|\;\mu(N)=0\big\},
\]
is a \(\sigma\)-ideal, thus \((\X,\Sigma,\mathcal N_\mu)\) is an enhanced measurable space. In this case, we abbreviate
\(L^0(\mathcal N_\mu)\) and \(L^\infty(\mathcal N_\mu)\) to \(L^0(\mu)\) and \(L^\infty(\mu)\), respectively. Similarly
for \(\sim_\mu\), \([\cdot]_\mu\), and \(\1^\mu_E\).
\fr}\end{example}
\subsubsection{\(\sigma\)-finite measure spaces}
A \textbf{measure space} \((\X,\Sigma,\mu)\) is a measurable space \((\X,\Sigma)\) together with a \(\sigma\)-additive
measure \(\mu\). We assume that the measure \(\mu\) is \(\sigma\)-finite. Given any exponent \(p\in[1,\infty)\), we define
\[
\mathcal L^p(\mu)\coloneqq\bigg\{f\in\mathcal L^0(\Sigma)\;\bigg|\;\|f\|_{\mathcal L^p(\mu)}\coloneqq\int|f|^p\,\d\mu<+\infty\bigg\}.
\]
It holds that \(\mathcal L^p(\mu)\) is a vector subspace of \(\mathcal L^0(\Sigma)\) and
\(\big(\mathcal L^p(\mu),\|\cdot\|_{\mathcal L^p(\mu)}\big)\) is a complete seminormed space.
The \textbf{\(p\)-Lebesgue space} \(\big(L^p(\mu),\|\cdot\|_{L^p(\mu)}\big)\) on \((\X,\Sigma,\mu)\)
is the Banach space defined as
\[
L^p(\mu)\coloneqq\big\{[f]_\mu\in L^0(\mu)\;\big|\;f\in\mathcal L^p(\mu)\big\}=\mathcal L^p(\mu)/\sim_\mu
\]
together with the quotient norm \(\|\cdot\|_{L^p(\mu)}\), which is given as follows: for any \(f\in L^p(\mu)\), one has
\[
\|f\|_{L^p(\mu)}\coloneqq\|\bar f\|_{\mathcal L^p(\mu)},\quad\text{ for some (thus, for any) representative }\bar f\in[f]_\mu.
\]
Observe that \(L^p(\mu)\) is also a Riesz subspace of \(L^0(\mu)\) for every exponent \(p\in[1,\infty)\).
\medskip

We endow the space \(L^0(\mu)\) with the following distance: fix a finite measure \(\tilde\mu\) on
\((\X,\Sigma)\) such that \(\mu\ll\tilde\mu\ll\mu\), whose existence is guaranteed by the \(\sigma\)-finiteness of \(\mu\);
then we define
\[
\sfd_{L^0(\mu)}(f,g)\coloneqq\int|f-g|\wedge 1\,\d\tilde\mu\quad\text{ for every }f,g\in L^0(\mu).
\]
Whenever \(\mu\) is finite already, we implicitly choose \(\tilde\mu\coloneqq\mu\). The distance \(\sfd_{L^0(\mu)}\) is complete and
the inclusion \(L^p(\mu)\hookrightarrow L^0(\mu)\) is continuous with dense image for every \(p\in[1,\infty]\). Notice also that the distance
\(\sfd_{L^0(\mu)}\) is not canonical (since it depends on the chosen auxiliary measure \(\tilde\mu\)), but the induced topology is independent
of the specific \(\tilde\mu\). Given any \((f_n)_{n\in\N}\subset L^0(\mu)\) and \(f\in L^0(\mu)\), it holds that \(\lim_{n\to\infty}\sfd_{L^0(\mu)}(f_n,f)=0\)
if and only if there exists a subsequence \((f_{n_i})_{i\in\N}\) of \((f_n)_{n\in\N}\) satisfying \(f(x)=\lim_{i\to\infty}f_{n_i}(x)\) for \(\mu\)-a.e.\ \(x\in\X\).
\medskip

Albeit well-known, we report a proof of the following result, for the reader's usefulness:
\begin{proposition}\label{prop:L0_CR}
Let \((\X,\Sigma,\mu)\) be a \(\sigma\)-finite measure space. Then the space \(L^0(\mu)\) is countably representable.
In particular, the space \(L^p(\mu)\) is countably representable for every \(p\in[1,\infty]\).
\end{proposition}
\begin{proof}
Suppose that \(\{f_i\}_{i\in I}\subset L^0(\mu)\) and \(g\in L^0(\mu)\) satisfy \(f_i\leq g\) for every \(i\in I\). Fix a finite measure \(\tilde\mu\geq 0\)
on \((\X,\Sigma)\) such that \(\mu\ll\tilde\mu\ll\mu\) and let us define \(\bar f_i\coloneqq\arctan\circ f_i\colon\X\to(-\frac{\pi}{2},\frac{\pi}{2})\) for every \(i\in I\).
Notice that \(s\coloneqq\sup\big\{\int\tilde f\,\d\tilde\mu\,:\,\tilde f\in\mathcal F\big\}<+\infty\), where we define
\[
\mathcal F\coloneqq\big\{\bar f_{i_1}\vee\ldots\vee\bar f_{i_k}\;\big|\;k\in\N,\,(i_j)_{j=1}^k\subset I\big\}.
\]
Now pick a sequence \((\tilde f_n)_{n\in\N}\subset\mathcal F\) such that \(s=\sup_{n\in\N}\int\tilde f_n\,\d\tilde\mu\). We can assume without loss of generality that
\(\tilde f_n\leq\tilde f_{n+1}\) for every \(n\in\N\), so that \(h\coloneqq\sup_{n\in\N}\tilde f_n\leq\arctan\circ g\) verifies \(\int h\,\d\tilde\mu=s\) thanks to the monotone convergence
theorem. Recalling the very definition of \(\mathcal F\), we can find a countable subset \(C\) of \(I\) such that \(\sup_{i\in C}\bar f_i=\sup_{n\in\N}\tilde f_n=h\). Notice that
\(h=\sup_{i\in I}\bar f_i\) in \(L^0(\tilde\mu)\) by construction, whence it follows that \(\sup_{i\in I}(\tan\circ\bar f_i)=\tan\circ h=\sup_{i\in I}f_i\) in \(L^0(\mu)\).
\end{proof}
\subsubsection{Submodular outer measures}
Given a metric space \((\X,\sfd)\), we denote by \(\mathscr B(\X)\) its Borel \(\sigma\)-algebra. If \(\mu\) is an outer measure on the set \(\X\), then we say that:
\begin{itemize}
\item \(\mu\) is \textbf{boundedly finite} if \(\mu(B)<+\infty\) whenever \(B\in\mathscr B(\X)\) is bounded.
\item \(\mu\) is \textbf{submodular} if it verifies
\[
\mu(E\cup F)+\mu(E\cap F)\leq\mu(E)+\mu(F)\quad\text{ for every }E,F\in\mathscr B(\X).
\]
\end{itemize}
The integral of a Borel function \(f\colon\X\to[0,+\infty]\) with respect to an outer measure \(\mu\) on \(\X\) can be defined via Cavalieri's formula, in the following way:
\[
\int_E f\,\d\mu\coloneqq\int_0^{+\infty}\mu\big(\big\{x\in E\,:\,f(x)>t\big\}\big)\,\d t\quad\text{ for every }E\in\mathscr B(\X).
\]
Observe that the above integral is well-defined because \([0,+\infty)\ni t\mapsto\mu\big(\big\{x\in E\,:\,f(x)>t\big\}\big)\) is a non-increasing function.
As proved in \cite[Chapter 6]{denneberg2010non} (see also \cite[Theorem 1.5]{debin2019quasicontinuous}), a given outer measure \(\mu\) is submodular if and only if the associated integral
is subadditive, meaning that
\[
\int_\X f+g\,\d\mu\leq\int_\X f\,\d\mu+\int_\X g\,\d\mu\quad\text{ for every }f,g\colon\X\to[0,+\infty]\text{ Borel.}
\]
Two classes of boundedly finite, submodular outer measures are particularly relevant to us:
\begin{itemize}
\item The outer measure induced (via Carath\'{e}odory construction) by a boundedly-finite Borel measure on \(\X\).
\item The Sobolev \(p\)-capacity \({\rm Cap}_p\) on a metric measure space, see e.g.\ \cite{HKST15}.
\end{itemize}
Given any boundedly finite, submodular outer measure \(\mu\) on \((\X,\sfd)\), we introduce a distance \(\sfd_{L^0(\mu)}\) as follows: we fix an increasing sequence
\((B_n)_n\) of bounded open subsets of \(\X\) with the property that each bounded subset \(B\) of \(\X\) is contained in \(B_n\) for some \(n\in\N\), and we define
\[
\sfd_{L^0(\mu)}(f,g)\coloneqq\sum_{n\in\N}\frac{1}{2^n(\mu(B_n)\vee 1)}\int_{B_n}|f-g|\wedge 1\,\d\mu\quad\text{ for every }f,g\in L^0(\mu).
\]
The submodularity of \(\mu\) guarantees that \(\sfd_{L^0(\mu)}\) is a distance. Arguing as in \cite[Proposition 1.10]{denneberg2010non},
one can see that \((f_i)_{i\in\N}\subset L^0(\mu)\) satisfies \(\sfd_{L^0(\mu)}(f_i,f)\to 0\) for some \(f\in L^0(\mu)\) if and only if
\[
\lim_{i\to\infty}\mu\big(B\cap\big\{|f_i-f|>\varepsilon\big\}\big)=0\quad
\text{ for every }\varepsilon>0\text{ and }B\in\mathscr B(\X)\text{ bounded.}
\]
In particular, arguing as in \cite[Proposition 1.12]{debin2019quasicontinuous} one can see that if \(\sfd_{L^0(\mu)}(f_i,f)\to 0\),
then there exists a subsequence \((i_j)_{j\in\N}\subset\N\) such that \(f(x)=\lim_j f_{i_j}(x)\) holds for \(\mu\)-a.e.\ \(x\in\X\).
The converse implication -- differently from what happens with boundedly-finite Borel measures -- might fail (see e.g.\ \cite[Remark 1.13]{debin2019quasicontinuous}).

\begin{example}\label{ex:L0Cap_no_CR}{\rm
Consider the real line \(\R\) equipped with the Sobolev \(2\)-capacity \({\rm Cap}_2\). Since all singletons have positive capacity,
we have \(\mathcal N_{{\rm Cap}_2}=\{\varnothing\}\) and thus \(L^0({\rm Cap}_2)=\mathcal L^0(\mathscr B(\R))\). Then
\[
L^0({\rm Cap}_2)\quad\text{ is not countably representable.}
\]
Indeed, the set \(\{\1_{\{t\}}\,:\,t\in\R\}\subset L^0({\rm Cap}_2)\) has an upper bound and \(\sup_{t\in\R}\1_{\{t\}}=\1_\R\),
but whenever \(C\subset\R\) is a countable set we have \(\sup_{t\in C}\1_{\{t\}}=\1_C\neq\1_\R\). In fact, since
\(L^0({\rm Cap}_2)\) is the space of (\({\rm Cap}_2\)-a.e.\ equivalence classes of) Borel functions from \(\R\) to \(\R\), we have that
\(L^0({\rm Cap}_2)\) is not even Dedekind complete: given any \(Q\subset\R\) that is not Borel, we have that \(\{\1_{\{t\}}\,:\,t\in Q\}\)
has an upper bound, but does not admit a supremum (which ought to be \(\1_Q\)) in \(L^0({\rm Cap}_2)\).
\fr}\end{example}
\subsection{Normed modules over function spaces}\label{ss:example_norm_mod}
\subsubsection{Examples of metric \(f\)-algebras}
Some examples of metric \(f\)-algebras:
\begin{itemize}
\item Let \((\X,\Sigma,\mathcal N)\) be an enhanced measurable space. Then \(\big(L^\infty(\mathcal N),\|\cdot\|_{L^\infty(\mathcal N)}\big)\) is a metric \(f\)-algebra.
If \(\mathcal N=\mathcal N_\mu\) for some \(\sigma\)-finite measure \(\mu\) on \((\X,\Sigma)\), then \(L^\infty(\mu)\) is CR.
\item Let \((\X,\Sigma,\mu)\) be a \(\sigma\)-finite measure space. Then \(\big(L^0(\mu),\sfd_{L^0(\mu)}\big)\) is a CR metric \(f\)-algebra.
\item Let \(\mu\) be a boundedly finite, submodular outer measure on some metric space \((\X,\sfd)\). Then \(\big(L^0(\mu),\sfd_{L^0(\mu)}\big)\) is a metric \(f\)-algebra.
\end{itemize}
Notice also that if \((\X,\Sigma,\mu)\) is a \(\sigma\)-finite measure space and \(p\in[1,\infty)\), then \(\big(L^p(\mu),\|\cdot\|_{L^p(\mu)}\big)\) is a CR metric Riesz space.
\subsubsection{Examples of metric \(f\)-structures}
Below we list some important examples of metric \(f\)-structures:
\begin{itemize}
\item Let \((\X,\Sigma,\mathcal N)\) be an enhanced measurable space. Then
\[
\big(L^\infty(\mathcal N),L^\infty(\mathcal N),L^\infty(\mathcal N)\big)
\]
is a metric \(f\)-structure. It is CR if \(\mathcal N=\mathcal N_\mu\) for some \(\sigma\)-finite measure \(\mu\) on \((\X,\Sigma)\).
\item Let \((\X,\Sigma,\mu)\) be a \(\sigma\)-finite measure space and \(p\in[1,\infty)\). Then
\[
\big(L^0(\mu),L^0(\mu),L^0(\mu)\big),\quad\big(L^0(\mu),L^\infty(\mu),L^p(\mu)\big)
\]
are CR metric \(f\)-structures.
\item Let \(\mu\) be a boundedly finite, submodular outer measure on some metric space \((\X,\sfd)\). Then
\[
\big(L^0(\mu),L^0(\mu),L^0(\mu)\big)
\]
is a metric \(f\)-structure.
\end{itemize}
\subsubsection{Examples of dual systems}
Some examples of dual systems of metric \(f\)-structures:
\begin{itemize}
\item Let \((\X,\Sigma,\mathcal N)\) be an enhanced measurable space. Then
\[
\big(L^\infty(\mathcal N),L^\infty(\mathcal N),L^\infty(\mathcal N),L^\infty(\mathcal N),L^\infty(\mathcal N)\big)
\]
is a dual system of metric \(f\)-structures. It is CR if \(\mathcal N=\mathcal N_\mu\) for some \(\sigma\)-finite measure \(\mu\).
\item Let \((\X,\Sigma,\mu)\) be a \(\sigma\)-finite measure space. Fix \(p,q,r\in[1,\infty]\) such that \(\frac{1}{p}+\frac{1}{q}=\frac{1}{r}\). Then
\[
\big(L^0(\mu),L^0(\mu),L^0(\mu),L^0(\mu),L^0(\mu)\big),\quad\big(L^0(\mu),L^\infty(\mu),L^p(\mu),L^q(\mu),L^r(\mu)\big)
\]
are CR metric \(f\)-structures.
\item Let \(\mu\) be a boundedly finite, submodular outer measure on some metric space \((\X,\sfd)\). Then
\[
\big(L^0(\mu),L^0(\mu),L^0(\mu),L^0(\mu),L^0(\mu)\big)
\]
is a dual system of metric \(f\)-structures.
\end{itemize}
\subsubsection{Examples of functional normed modules}
Below we list some classes of Banach modules that have been studied in the literature and fall into the category of Banach modules over a metric \(f\)-structure:
\begin{itemize}
\item \(L^p(\mu)\)-Banach \(L^\infty(\mu)\)-modules, where \((\X,\Sigma,\mu)\) is a \(\sigma\)-finite measure space and \(p\in[1,\infty]\), which have been introduced
in \cite[Definition 1.2.10]{Gigli14}. We point out that in \cite{Gigli14} the terminology is different: every \(L^p(\mu)\)-normed \(L^\infty(\mu)\)-module is assumed
to be complete (thus, a Banach module with our definitions) and non-complete normed modules are not considered.
\item \(L^0(\mu)\)-Banach \(L^0(\mu)\)-modules, where \((\X,\Sigma,\mu)\) is a \(\sigma\)-finite measure space. The definition was given in \cite[Definition 2.6]{Gigli17},
but the concept appeared previously in \cite[Section 1.3]{Gigli14}.
\item \(L^0(\mu)\)-Banach \(L^0(\mu)\)-modules, where \(\mu\) is a boundedly-finite, submodular outer measure on a metric space \((\X,\sfd)\), see \cite[Definition 2.1]{BPS21}.
In the particular case where \(\mu\) is the Sobolev \(2\)-capacity on a metric measure space, it appeared previously in \cite[Definition 3.1]{debin2019quasicontinuous}.
\item \(\mathcal L^\infty(\Sigma)\)-Banach \(\mathcal L^\infty(\Sigma)\)-modules, where \((\X,\Sigma)\) is a measurable space. See \cite[Definition 3.1]{DMLP21}.
\item \(L^\infty(\mathcal N)\)-Banach \(L^\infty(\mathcal N)\)-modules, where \((\X,\Sigma,\mathcal N)\) is an enhanced measurable space, which were introduced in \cite[Definition 4.3]{GLP22}.
\end{itemize}
Notice that, given a \(\sigma\)-finite measure space \((\X,\Sigma,\mu)\) and an \(L^\infty(\mu)\)-Banach \(L^\infty(\mu)\)-module \(\mathscr M\), two different notions of duals
of \(\mathscr M\) have been considered (corresponding to two different underlying dual systems of metric \(f\)-structures):
\begin{itemize}
\item The dual of \(\mathscr M\) in the sense of \cite[Definition 1.2.6]{Gigli14} is an \(L^1(\mu)\)-Banach \(L^\infty(\mu)\)-module, since the dual system under consideration
is \(\big(L^0(\mu),L^\infty(\mu),L^\infty(\mu),L^1(\mu),L^1(\mu)\big)\).
\item The dual of \(\mathscr M\) in the sense of \cite[Definition 4.15]{GLP22} is an \(L^\infty(\mu)\)-Banach \(L^\infty(\mu)\)-module, since the dual system under consideration
is \(\big(L^\infty(\mu),L^\infty(\mu),L^\infty(\mu),L^\infty(\mu),L^\infty(\mu)\big)\).
\end{itemize}
\subsubsection{Some applications}\label{ss:applications_fct_nm}
We list some examples of known constructions that follow from Theorems \ref{thm:module_generated}, \ref{thm:main_ext_hom}, and \ref{thm:pshfrwd_mod}:
\begin{itemize}
\item \textsc{Cotangent module}. Let \((\X,\sfd,\mu)\) be a metric measure space and \(p\in(1,\infty)\). We denote by \(W^{1,p}(\X)\) the \(p\)-Sobolev space of \((\X,\sfd,\mu)\)
and by \(|Df|\in L^p(\mu)\) the minimal weak upper gradient of \(f\in W^{1,p}(\X)\) (e.g.\ in the sense of \cite{Cheeger00,Shanmugalingam00,AmbrosioGigliSavare11}; all these
approaches are equivalent by \cite{AmbrosioGigliSavare11-3}). Consider the metric \(f\)-structure \(\big(L^0(\mu),L^\infty(\mu),L^p(\mu)\big)\), as well as the map \(\psi_p\colon W^{1,p}(\X)\to L^p(\mm)^+\)
given by \(\psi_p(f)\coloneqq|Df|\) for every \(f\in W^{1,p}(\X)\). Then the \textbf{cotangent module} \(L^p(T^*\X)\) and the \textbf{differential} operator \(\d\colon W^{1,p}(\X)\to L^p(T^*\X)\) are
\[
\big(L^p(T^*\X),\d\big)\cong(\mathscr M_{\langle\psi_p\rangle},T_{\langle\psi_p\rangle}).
\]
The cotangent module (for \(p=2\)) has been introduced in \cite[Definition 2.2.1]{Gigli14} and refined in \cite[Theorem/Definition 2.8]{Gigli17}.
Many other generalisations appeared later: for example, one can drop the \(L^p\)-integrability assumption (see \cite[Proposition 4.18]{GP20}),
one can construct the capacitary tangent module on an \({\sf RCD}(K,\infty)\) space (see \cite[Theorem 3.6]{debin2019quasicontinuous}), or
one can consider the cotangent modules induced by axiomatic classes of Sobolev-type spaces (see \cite[Theorem 3.2]{GIGLI2019}). Concerning the latter
notion, we point out that -- thanks to Theorem \ref{thm:module_generated} -- the strong locality assumption on the \(D\)-structure in
\cite[Theorem 3.2]{GIGLI2019} can be removed.
\item \textsc{Pullback module.} Let \((\X,\Sigma_\X,\mu_\X)\), \((\Y,\Sigma_\Y,\mu_\Y)\) be \(\sigma\)-finite measure spaces and \(\varphi\colon\X\to\Y\)
a map of bounded compression, i.e.\ \(\varphi\) is measurable and there exists a constant \(C>0\) such that \(\varphi_\#\mu_\X\leq C\mu_\Y\).
Notice that \(\varphi\) induces a homomorphism of metric \(f\)-structures 
\[
\boldsymbol\varphi\colon\big(L^0(\mu_\Y),L^\infty(\mu_\Y),L^p(\mu_\Y)\big)\to\big(L^0(\mu_\X),L^\infty(\mu_\X),L^p(\mu_\X)\big)
\]
for every exponent \(p\in[1,\infty)\) via pre-composition. Namely, given any \(f\in L^0(\mu_\Y)\) we define
\[
\boldsymbol\varphi(f)\coloneqq[\bar f\circ\varphi]_{\mu_\X},\quad\text{ for any }\bar f\in\mathcal L^0(\Sigma_\Y)\text{ with }[\bar f]_{\mu_\Y}=f.
\]
Let \(\mathscr M\) be an \(L^p(\mu_\Y)\)-Banach \(L^\infty(\mu_\Y)\)-module. Then the \textbf{pullback module} is given by
\[
(\varphi^*\mathscr M,\varphi^*)\cong(\boldsymbol\varphi_*\mathscr M,\boldsymbol\varphi_*).
\]
The pullback module was introduced in \cite[Definition 1.6.2]{Gigli14}, \cite[Theorem/Definition 2.23]{Gigli17} and has many generalisations:
for example, for \(L^0\)-Banach \(L^0\)-modules and under the weaker assumption \(\varphi_\#\mu_\X\ll\mu_\Y\) (see \cite[Theorem/Definition 3.2]{GR17})
or for \(L^\infty\)-Banach \(L^\infty\)-modules (see \cite[Theorem/Definition 4.11]{GLP22}).
\item \textsc{\(L^0\)-completion.} Let \((\X,\Sigma,\mu)\) be a \(\sigma\)-finite measure space and let \(\mathscr M\) be an \(L^p(\mu)\)-Banach \(L^\infty(\mu)\)-module,
for some exponent \(p\in[1,\infty]\). Then the \textbf{\(L^0\)-completion} of \(\mathscr M\) (in the sense of \cite[Theorem/Definition 1.7]{Gigli17}) is given by
\[
(\bar{\mathscr M},\iota)\cong(\mathscr M_{\langle\psi_{\mathscr M}\rangle},T_{\langle\psi_{\mathscr M}\rangle}),
\]
where we define \(\psi_{\mathscr M}\colon\mathscr M\to L^0(\mu)^+\) as \(\psi_{\mathscr M}(v)\coloneqq|v|\) for every \(v\in\mathscr M\).
\item \textsc{von Neumann lifting.} Let \((\X,\Sigma,\mu)\) be a complete \(\sigma\)-finite measure space, \(\ell\) a von Neumann lifting of \(\mu\), and \(\mathscr M\)
an \(L^\infty(\mu)\)-Banach \(L^\infty(\mu)\)-module. Then the \textbf{von Neumann lifting} of \(\mathscr M\) (see \cite[Theorem 3.5]{DMLP21}) is given by
\[
(\ell\mathscr M,\ell)\cong(\mathscr M_{\langle\psi_\ell\rangle},T_{\langle\psi_\ell\rangle}),
\]
where we define \(\psi_\ell\colon\mathscr M\to\mathcal L^\infty(\Sigma)^+\) as \(\psi_\ell(v)\coloneqq\ell(|v|)\) for every \(v\in\mathscr M\).
\end{itemize}
We point out that, in addition to the objects we discussed above, also the associated existence results for homomorphisms (e.g.\ the universal property of pullback
modules \cite[Proposition 1.6.3]{Gigli14} or the lifting of a homomorphism \cite[Proposition 4.14]{GLP22}) can be deduced from Proposition \ref{prop:module_generated_hom}.
\subsection{The Reconstruction Theorem}\label{ss:reconstruction}
A \textbf{Boolean ring} is a ring \((R,+,\cdot)\) such that \(r^2=r\) for every \(r\in R\). In particular, \(r=-r\) and \(rs=sr\)
for every \(r,s\in R\). A \textbf{Boolean algebra} is a a Boolean ring \((A,+,\cdot)\) with a multiplicative identity \(1_A\).
A ring homomorphism \(\phi\colon A\to B\) between two Boolean algebras \(A\) and \(B\) is said to be a \textbf{Boolean homomorphism}
provided it is also \textbf{uniferent}, meaning that \(\phi(1_A)=1_B\).
\medskip

Given a set \(\X\neq\varnothing\) and an algebra \(\Sigma\) of subsets of \(\X\), the triple \((\Sigma,\Delta,\cap)\) is a Boolean algebra
with zero \(\varnothing\) and identity \(\X\). The following fundamental result -- which is known as the \textbf{Stone's Representation
Theorem for Boolean algebras} -- states that in fact any Boolean algebra can be expressed as an algebra of sets. We will employ it in the
proof of Proposition \ref{prop:Idem_as_sigma-algebra}.
\begin{theorem}[Stone's Theorem]\label{thm:Stone}
Let \(A\) be a Boolean algebra. Then there exist a set \(\X\) and an algebra \(\Sigma\) of subsets of \(\X\) such that
\((A,+,\cdot)\) and \((\Sigma,\Delta,\cap)\) are isomorphic as Boolean algebras.
\end{theorem}

Let \(U\) be a given \(f\)-algebra. Then we define the operations \(\boxplus\colon\Idem(U)\times\Idem(U)\to\Idem(U)\)
and \(\boxtimes\colon\Idem(U)\times\Idem(U)\to\Idem(U)\) on \(\Idem(U)\) as
\[
u\boxplus v\coloneqq u+v-2uv,\qquad u\boxtimes v\coloneqq uv,\qquad\text{for every }u,v\in\Idem(U).
\]
Their well-posedness follows from items i) and ii) of Lemma \ref{lem:prop_idem}. It is easy to check that the triple
\(\big(\Idem(U),\boxplus,\boxtimes\big)\) is a Boolean algebra with zero element \(0\) and multiplicative identity \(1\).

\begin{proposition}\label{prop:Idem_as_sigma-algebra}
Let \(U\) be a Dedekind \(\sigma\)-complete \(f\)-algebra whose multiplication map is \(\sigma\)-order-continuous on \(U^+\times U^+\).
Then the space \(\big(\Idem(U),\boxplus,\boxtimes\big)\) is Boolean isomorphic to a \(\sigma\)-algebra.
\end{proposition}
\begin{proof}
Thanks to Stone's Representation Theorem \ref{thm:Stone}, we can find a set \(\X\neq\varnothing\), an algebra \(\Sigma\) of subsets of \(\X\),
and a Boolean isomorphism \({\rm I}\colon\big(\Idem(U),\boxplus,\boxtimes\big)\to(\Sigma,\Delta,\cap)\). We claim that
\begin{equation}\label{eq:reconstr_thm_1}
{\rm I}\Big(\sup_{n\in\N}u_n\Big)=\bigcup_{n\in\N}{\rm I}(u_n),\quad\text{ for every }(u_n)_{n\in\N}\subset\Idem(U).
\end{equation}
Call \(u\coloneqq\sup_{n\in\N}u_n\). Recall that \(u\in\Idem(U)\) by Lemma \ref{lem:order_idem}. Given any \(n\in\N\),
it holds that \(u_n\leq u\) and thus Remark \ref{rmk:char_idem_ineq} gives \({\rm I}(u_n)\cap{\rm I}(u)={\rm I}(u_n u)={\rm I}(u_n)\),
which yields \(\bigcup_{n\in\N}{\rm I}(u_n)\subset{\rm I}(u)\).
Conversely, pick any set \(E\in\Sigma\) with \({\rm I}(u_n)\subset E\) for every \(n\in\N\). Calling \(v\coloneqq{\rm I}^{-1}(E)\),
we have that \({\rm I}(u_n v)={\rm I}(u_n)\cap{\rm I}(v)={\rm I}(u_n)\), so that \(u_n v=v_n\). Hence, it holds that
\[
u=\sup_{n\in\N}u_n=\sup_{n\in\N}u_n v=v\sup_{n\in\N}u_n=uv,
\]
thus \({\rm I}(v)\cap{\rm I}(u)={\rm I}(uv)={\rm I}(u)\). We obtain that \({\rm I}(u)\subset{\rm I}(v)=E\), whence
\eqref{eq:reconstr_thm_1} follows. We deduce that \(\Sigma\) is a \(\sigma\)-algebra, so that \((\X,\Sigma)\) is a
measurable space. This completes the proof.
\end{proof}
\begin{theorem}[Reconstruction Theorem]\label{thm:reconstruction}
Let \(U\) be a localisable \(f\)-algebra. Then there exists a measurable space \((\X,\Sigma)\) such
that \(U\) is isomorphic (as an \(f\)-algebra) to an \(f\)-subalgebra of \(\mathcal L^0(\Sigma)\).

More precisely, the measurable space \((\X,\Sigma)\) can be chosen so that \((\Sigma,\Delta,\cap)\)
is isomorphic (as a Boolean algebra) to \(\big(\Idem(U),\boxplus,\boxtimes\big)\).
\end{theorem}
\begin{proof}
Proposition \ref{prop:Idem_as_sigma-algebra} yields a measurable space \((\X,\Sigma)\) such that \((\Sigma,\Delta,\cap)\)
and \(\big(\Idem(U),\boxplus,\boxtimes\big)\) are isomorphic as Boolean algebras. We introduce the mapping
\(\iota\colon\mathcal S^+(U)\to\mathcal L^0(\Sigma)\) as follows: given any simple element \(u=\sum_{i=1}^k\lambda_i u_i\in\mathcal S^+(U)\),
we define the function \(\iota(u)\colon\X\to[0,+\infty]\) as
\[
\iota(u)(x)\coloneqq\sum_{i=1}^k\lambda_i\1_{{\rm I}(u_i)}(x),\quad\text{ for every }x\in\X,
\]
where \({\rm I}\colon\Idem(U)\to\Sigma\) is some fixed Boolean isomorphism. Observe that \(\iota(u)\) belongs to \({\sf Sf}(\Sigma)\).
In order to extend the mapping \(\iota\) to \(U^+\), we first need to prove the following two auxiliary results:
\begin{itemize}
\item[\(\rm a)\)] If \(u\in U^+\) and \((u_n)_{n\in\N}\subset\mathcal S^+(U)\) is a non-decreasing sequence satisfying \(u=\sup_{n\in\N}u_n\),
then it holds that \(\sup_{n\in\N}\iota(u_n)(x)<+\infty\) for every \(x\in\X\).
\item[\(\rm b)\)] If \(u\in U^+\) and \((u_n)_{n\in\N},(v_n)_{n\in\N}\subset\mathcal S^+(U)\) are non-decreasing sequences with
\(u=\sup_{n\in\N}u_n\) and \(u=\sup_{n\in\N}v_n\), then it holds that \(\sup_{n\in\N}\iota(u_n)(x)=\sup_{n\in\N}\iota(v_n)(x)\)
for every \(x\in\X\).
\end{itemize}
To prove a), we argue by contradiction: suppose that \(\sup_{n\in\N}\iota(u_n)(x_0)=+\infty\) for some \(x_0\in\X\). For any \(n\in\N\),
we can find \(\lambda_n\in[0,+\infty)\) and \(\tilde u_n\in\Idem(U)\) with \(\lambda_n\tilde u_n=\tilde u_n u_n\) and \(x_0\in{\rm I}(\tilde u_n)\).
One has \(\lambda_n=\iota(u_n)(x)\to+\infty\) as \(n\to\infty\). Define \(E\coloneqq\bigcap_{n\in\N}{\rm I}(\tilde u_n)\in\Sigma\) and
\(w\coloneqq{\rm I}^{-1}(E)\in\Idem(U)\). Notice that \(x_0\in E\) and \(w\leq\tilde u_n\) for every \(n\in\N\).
Given \(k\in\N\), there is \(n_k\in\N\) with \(\lambda_{n_k}\geq k\), thus
\[
kw\leq\lambda_{n_k}\tilde u_{n_k}=\tilde u_{n_k}u_{n_k}\leq u.
\]
Since \(U\) is Archimedean by Proposition \ref{prop:Dedekind_is_Archimedean}, we deduce that \(w=0\) and thus \(E=\varnothing\).
This leads to a contradiction with the fact that \(x_0\in E\), so that a) is proved. We pass to the verification of b). Fix any point
\(x\in\X\). For any \(n\in\N\), we can pick \(\lambda_n,\mu_n\in[0,+\infty)\) and \(\tilde u_n,\tilde v_n\in\Idem(U)\) such that
\(\lambda_n\tilde u_n=\tilde u_n u_n\), \(\mu_n\tilde v_n=\tilde v_n v_n\), and \(x\in{\rm I}(\tilde u_n)\cap{\rm I}(\tilde v_n)\).
Setting \(E\coloneqq\bigcap_{n\in\N}{\rm I}(\tilde u_n)\cap{\rm I}(\tilde v_n)\in\Sigma\), we have \(x\in E\) and thus
\(w\coloneqq{\rm I}^{-1}(E)\neq 0\). By the \(\sigma\)-order-continuity of the multiplication, we obtain
\[\begin{split}
\Big(\sup_{n\in\N}\iota(u_n)(x)\Big)w&=\Big(\sup_{n\in\N}\lambda_n\Big)w=\sup_{n\in\N}\lambda_n w=\sup_{n\in\N}\lambda_n\tilde u_n w
=\sup_{n\in\N}u_n\tilde u_n w=\Big(\sup_{n\in\N}u_n\Big)w=uw\\
&=\Big(\sup_{n\in\N}v_n\Big)w=\sup_{n\in\N}\mu_n\tilde v_n w=\Big(\sup_{n\in\N}\mu_n\Big)w=\Big(\sup_{n\in\N}\iota(v_n)(x)\Big)w,
\end{split}\]
where we used that \(\tilde u_n w=\tilde v_n w=w\). This yields \(\sup_{n\in\N}\iota(u_n)(x)=\sup_{n\in\N}\iota(v_n)(x)\), proving b).

We now define the function \(\iota(u)\colon\X\to[0,+\infty)\) for any \(u\in U^+\) in the following way: given any non-decreasing sequence
\((u_n)_{n\in\N}\subset\mathcal S^+(U)\) such that \(u=\sup_{n\in\N}u_n\) -- whose existence is guaranteed by the assumption
that the \(f\)-algebra \(U\) is localisable -- we define
\[
\iota(u)(x)\coloneqq\sup_{n\in\N}\iota(u_n)(x),\quad\text{ for every }x\in\X.
\]
The properties a) and b) ensure that \(\iota(u)\) is well-posed. Notice that \(\iota(u)\in\mathcal L^0(\Sigma)\), as a countable
supremum of elements of \(\mathcal L^0(\Sigma)\). The \(\sigma\)-order-continuity of the sum and multiplication maps gives
\begin{equation}\label{eq:reconstr_aux}
\iota(u)+\iota(v)=\iota(u+v),\qquad\iota(uv)=\iota(u)\iota(v),\qquad\text{ for every }u,v\in U^+.
\end{equation}
Finally, we extend \(\iota\) to a mapping \(\iota\colon U\to\mathcal L^0(\Sigma)\) by setting
\[
\iota(u)\coloneqq\iota(u^+)-\iota(u^-),\quad\text{ for every }u\in U.
\]
Using \eqref{eq:reconstr_aux}, one can easily show that \(\iota\) is a homomorphism of \(f\)-algebras. In order to conclude, it only
remains to check that \(\iota\) is injective. To this aim, fix any \(u\in U\) such that \(\iota(u)=0\). We want to show that \(u=0\).
Since \(u^+ u^-=0\) by \eqref{eq:basic_prop_f-alg_7}, we deduce that \(\iota(u^+)\iota(u^-)=\iota(u^+ u^-)=0\), which yields
\(\iota(u^+)=\iota(u^-)=0\). Hence, it suffices to prove the implication \(\iota(u)=0\Longrightarrow u=0\) in the case where \(u\in U^+\).
Choose a non-decreasing sequence \((u_n)_{n\in\N}\subset\mathcal S^+(U)\) such that \(u=\sup_{n\in\N}u_n\). We have that
\(\sup_{n\in\N}\iota(u_n)=0\), whence it follows that \(u_n=0\) for every \(n\in\N\) and thus \(u=0\).
\end{proof}
\small
\def\cprime{$'$} \def\cprime{$'$}

\end{document}